\documentclass[a4paper,11pt]{amsart}

\usepackage[all]{xy}
\SelectTips{cm}{}
\usepackage{graphicx}
\usepackage{psfrag}
\usepackage{tikz}
\usepackage{tikz-cd}
\usepackage{mathtools}
\usepackage{amsmath,amssymb,tikz-cd}
\usepackage[colorlinks, linkcolor=blue,anchorcolor=Periwinkle,
citecolor=blue,urlcolor=Emerald]{hyperref}

\usepackage[export]{adjustbox}
\usepackage{calligra}

\usetikzlibrary{decorations.pathreplacing}
\usetikzlibrary{matrix,arrows}

\newcommand{\ol}[1]{\overline{#1}}

\newtheorem{theorem}{Theorem}[section]
\numberwithin{equation}{theorem}

\newtheorem{definition}[theorem]{Definition}
\newtheorem{classification-theorem}[subsubsection]{Classification Theorem}
\newtheorem{decomposition-theorem}[subsubsection]{Decomposition Theorem}
\newtheorem{proposition-definition}[theorem]{Proposition-Definition}
\newtheorem{definition-proposition}[theorem]{Definition-Proposition}
\newtheorem{example-definition}[theorem]{Example-Definition}
\newtheorem{periodicity-conjecture}[subsubsection]{Periodicity Conjecture}
\newtheorem{lemma}[theorem]{Lemma}
\newtheorem{proposition}[theorem]{Proposition}
\newtheorem{corollary}[theorem]{Corollary}

\newtheorem{example}[theorem]{Example}
\newtheorem{remark}[theorem]{Remark}

\newtheorem*{definition2}{Definition}
\newtheorem*{theoremA}{Theorem~A}
\newtheorem*{theoremB}{Theorem~B}

\newtheorem{Definition-Proposition}[theorem]{D\'efinition-Proposition}

\newcommand{\reminder}[1]{}

\newcommand{\rep}{\mathrm{rep}}

\newcommand{\Mod}{\mathrm{Mod}\,}

\newcommand{\pd}{\mathrm{pd}}
\newcommand{\proj}{\mathrm{proj}\,}

\newcommand{\per}{\mathrm{per} }

\newcommand{\pvd}{\mathrm{pvd} }
\newcommand{\add}{\mathrm{add} }

\newcommand{\tr}{\mathrm{tr}}

\newcommand{\Cone}{\mathrm{Cone}}
\newcommand{\Com}{\mathrm{Com}}

\newcommand{\dgcat}{\mathrm{dgcat}}
\newcommand{\Hqe}{\mathrm{Hqe}}

\renewcommand{\rep}{\mathrm{rep}}

\newcommand{\pretr}{\mathrm{pretr} }

\newcommand{\Set}{\mathrm{Set}}

\newcommand{\cok}{\mathrm{cok} }

\renewcommand{\ker}{\mathrm{ker} }
\newcommand{\obj}{\mathrm{obj} }

\newcommand{\Q}{\mathbb{Q}}

\newcommand{\iso}{\xrightarrow{_\sim}}

\newcommand{\Cosp}{\mathrm{Cosp}}
\newcommand{\Sp}{\mathrm{Sp}}
\newcommand{\Sq}{\mathrm{Sq}}
\newcommand{\Cat}{\mathrm{Cat}}
\newcommand{\Quiv}{\mathrm{Quiv}}
\newcommand{\Id}{\mathrm{id}}
\newcommand{\Dia}{\mathrm{Dia}}
\newcommand{\Res}{\mathrm{Res}}
\newcommand{\Iso}{\mathrm{Iso}}
\newcommand{\Ex}{\mathrm{Ex}}
\newcommand{\Ab}{\mathrm{Ab}}

\newcommand{\Def}{\mathrm{def}\kern 0.1em}
\renewcommand{\Q}{\mathcal{Q}}
\newcommand{\Str}{\mathrm{Str}}

\newcommand{\D}{\mathcal {D}}
\newcommand{\A}{\mathcal {A}}
\newcommand{\B}{\mathcal {B}}
\newcommand{\C}{\mathcal {C}}

\newcommand{\F}{\mathcal {F}}

\newcommand{\I}{\mathcal {I}}
\newcommand{\M}{\mathcal {M}}

\newcommand{\s}{\mathcal S}
\newcommand{\T}{\mathcal T}
\renewcommand{\P}{\mathcal P}

\newcommand{\Vect}{\mathrm{Vect}}
\newcommand{\Fun}{\mathrm{Fun}}
\newcommand{\Hom}{\mathrm{Hom}}

\newcommand{\RHom}{\mathrm{RHom}}

\newcommand{\End}{\mathrm{End}}
\newcommand{\rad}{\mathrm{rad}}

\newcommand{\Mor}{\mathrm{Mor}}

\renewcommand{\phi}{\varphi}

\newcommand{\hker}{\mathrm{hker}}

\newcommand{\pr}{\mathrm{pr}}

\renewcommand{\tilde}[1]{\widetilde{#1}}

\renewcommand{\H}{\mathcal H}
\newcommand{\Ai}{{A}}

\begin{document}
\title[Exact dg categories I: foundations]{Exact dg categories I: foundations}
\author[Xiaofa Chen]{Xiaofa Chen} 

\address{University of Science and Technology of China, Hefei, P.~R.~China}
\email{cxf2011@mail.ustc.edu.cn}

\date{\today}


\maketitle

\begin{abstract}
We introduce the notion of exact dg category, which provides a differential
graded enhancement of Nakaoka--Palu's notion of extriangulated category. 
We give a definition in complete analogy with Quillen's but 
where the category of kernel-cokernel pairs is replaced with a more sophisticated homotopy category. 
We introduce the notion of stable dg category, and prove that the $H^0$-category of an exact dg category $\A$ is triangulated if and only if $\A$ is stable.
We illustrate our theory with several examples including the homotopy category of two-term complexes and Amiot's fundamental domain for generalized cluster categories.
\end{abstract}

\setcounter{secnumdepth}{3}
\setcounter{tocdepth}{3 }
\tableofcontents

\section{Introduction}  
In this paper, we introduce and study the notion of {\em exact dg category}.
Let us illustrate the position of this class of categories among the other types of 
categories that are relevant in the following diagram
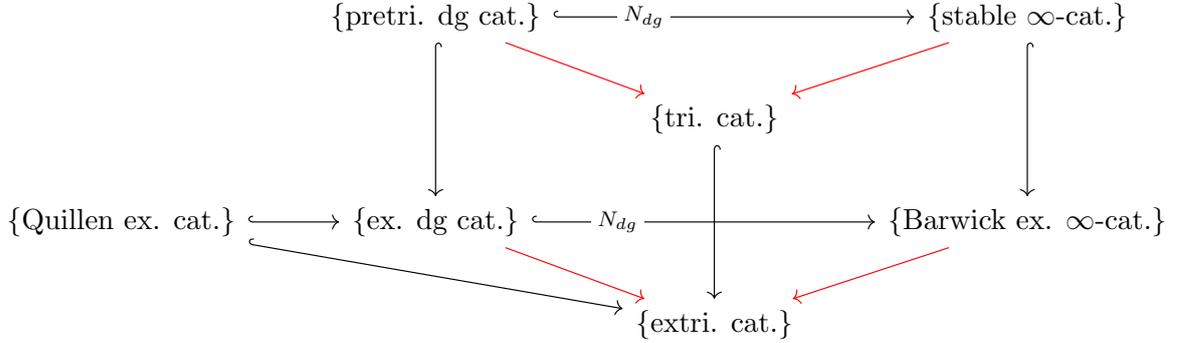
\[
\begin{tikzcd}
&\{\text{pretri. dg cat.}\}\ar[rd,red]\ar[rr,"N_{dg}"{description,near start},hook]\ar[dd,hook]&&\{\text{stable $\infty$-cat.}\}\ar[dd,hook]\ar[ld,red]\\
&&\{\text{tri. cat.}\}\ar[dd,hook]&\\
\{\text{Quillen ex. cat.}\}\ar[r,hook]\ar[rrd,hook]&\{\text{ex. dg cat.}\}\ar[rr,hook,"N_{dg}"{description,near start}]\ar[rd,red]&&\{\text{Barwick ex. $\infty$-cat.}\}\ar[ld,red]\\
&&\{\text{extri. cat.}\}&
\end{tikzcd}
\]
Here, the black arrows denote inclusions of classes and the red arrows send a dg category (respectively an
$\infty$-category)  $\A$ to $H^0(\A)$ (respectively $h\A$). The {dg nerve functor} $N_{dg}$~\cite[Construction 1.3.1.6]{LurieHA} sends a dg category to an $\infty$-category. As we see in the diagram, the notion of exact dg category is
\begin{itemize}
\item[1)] a dg version of Barwick's \cite{Barwick15} notion of exact $\infty$-category;
\item[2)] a simultaneous generalization of the notions of exact category in the sense of Quillen \cite{Quillen73} and of 
pretriangulated dg category in the sense of Bondal--Kapranov \cite{BondalKapranov90}, and
\item[3)] a dg enhancement of the notion of extriangulated category in the sense of Nakaoka--Palu 
\cite{NakaokaPalu19} in analogy with Bondal--Kapranov's dg enhancement of the notion of triangulated category.
\end{itemize}
Let us emphasize that our notion of exact dg category is completely different from Positselski's notion introduced in~\cite{Positselski21}: an exact structure in our sense can be transported along a quasi-equivalence, cf.~Remark~\ref{truncationexactdgstructure} b), which is not the case for an exact structure in Positselski's sense.

The notion of an extriangulated category was introduced by Nakaoka--Palu in~\cite{NakaokaPalu19}. 
It is a simultaneous generalization of Quillen's notion of exact category~\cite{Quillen73} and 
Grothendieck--Verdier's notion of a triangulated category~\cite{Verdier96}. 
Their aim was to give a convenient setup for writing down proofs 
which apply to both exact categories and triangulated categories, and more generally to
extension-closed subcategories of triangulated categories. 
The theory of extriangulated categories has developed a lot since their introduction in 2019 
and many notions and constructions have been generalized to this setting (or the more general 
setting of $n$-exangulated categories \cite{HerschendLiuNakaoka21}). 
The notion of extriangulated category plays an important role in the additive categorification of cluster algebras with coefficients, see for example in~\cite{Wu23a,KellerWu23, WangWeiZhang23,FangGorskyPaluPlamondonPressland23a}.
Our aim in this work is to introduce the notion of exact dg category. It enhances Nakaoka--Palu's notion of extriangulated category, cf.~Theorem~\hyperref[intro:truncationextriangulated]{A}.

Let us start by introducing the notion of {\em homotopy short exact sequence}. It is the key ingredient in the definition of an {exact dg category}. 
Let $\A$ be an {\em additive} dg category, i.e.~the category $H^0(\A)$ is additive.
The category $\mathcal{H}_{3t}(\A)$ of {\em 3-term homotopy complexes over $\A$} is the $H^0$-category of the dg category of strictly unital $\Ai_{\infty}$-functors 
from a certain index category to the dg category $\A$, cf.~Definition~\ref{def:3termhomotopy}. 
Its objects are identified with 3-term h-complexes, i.e.~diagrams in $\A$
\begin{equation}\label{intro:F}
\begin{tikzcd}
&A_0\ar[r,"f"]\ar[rr,bend right = 8ex,"h"swap]&A_1\ar[r,"j"]&A_2,
\end{tikzcd}
\end{equation}
where $|f|=|j|=0$, $|h|=-1$ and $d(f)=0$, $d(j)=0$ and $d(h)=-jf$.
For a morphism $m:A\rightarrow B$ in $\A$, we denote by $m^{\wedge}:A^{\wedge}\rightarrow B^{\wedge}$ its image under the Yoneda dg functor. A 3-term h-complex over $\A$ (\ref{intro:F}) yields a canonical morphism in $\C(\A)$ 
\[
u:A_0^{\wedge}\rightarrow \Sigma^{-1}\Cone(j^{\wedge}:A_1^{\wedge}\rightarrow A_2^{\wedge}),
\]
where $u=\begin{bmatrix}f^{\wedge}\\h^{\wedge}\end{bmatrix}$.
\begin{definition2}[{Lemma~\ref{lem:3termhcomplex}}] \label{intro:homotopyleftexact}
A $3$-term h-complex (\ref{intro:F}) is homotopy left exact if the map $\tau_{\leq 0}(u)$ is a quasi-isomorphism of dg $\tau_{\leq 0}(\A)$-modules. Dually, we define the notion of homotopy right exact sequence and then the notion of homotopy short exact sequence.
\end{definition2}

Using the above terminology, we define an {\em exact structure} on an additive dg category $\A$ 
to be a class $\mathcal S$ of homotopy short exact sequences in $\A$, satisfying certain axioms analogous to those of Quillen. In this case, we call the pair $(\A,\mathcal S)$, or simply $\A$, an {\em exact dg category}.
Roughly speaking, an exact dg category is a homotopical version of Quillen's exact category in the context of dg categories.
We refer to Section~\ref{sec:exactdgcategory} for more details.

Our first result states that exact dg categories do enhance extriangulated categories. More precisely, we have the following theorem.
\begin{theoremA}[Theorem~\ref{truncationextriangulated}]\label{intro:truncationextriangulated}
Let $\A$ be an exact dg category. Then $H^0(\A)$ has a canonical extriangulated structure. 
\end{theoremA}
Extriangulated categories of this form are called {\em algebraic}. An analogous result for exact $\infty$-categories is due to Nakaoka--Palu \cite{NakaokaPalu20}.
Since the notion of exact structure on an additive dg category $\A$ is defined in terms of homotopy short exact sequences, which depends only on the connective cover $\tau_{\leq 0}\A$ of $\A$, the inclusion $\tau_{\leq 0}\A\rightarrow \A$ induces a bijection between the exact structures on both sides.
Consequently, every algebraic extriangulated category possesses a connective dg enhancement. 
It is worth noting that an extension-closed full dg subcategory of an exact dg category inherits a canonical exact structure. As a result, both Yilin Wu's Higgs categories \cite{Wu23a} and Haibo Jin's categories of Cohen--Macaulay dg modules \cite{Jin20} fall into the framework of algebraic extriangulated categories.

 The following new notion is a dg analog of Lurie's notion of stable $\infty$-category~\cite{LurieHA}. 
\begin{definition2}[Definition~\ref{def:stable}]\label{intro:stable}
Let $\A$ be an additive dg category. 
It is {\em  stable} if the following conditions are satisfied:
\begin{itemize}
\item[(a)] the dg category $\A$ admits homotopy kernels and homotopy cokernels;
\item[(b)] a  3-term homotopy complex is homotopy left exact if and only if it is homotopy right exact.
\end{itemize} 
\end{definition2}
We will show in a subsequent paper that a dg category $\A$ is stable if and only if $N_{dg}(\A)$ is stable as an $\infty$-category. A pretriangulated dg category is necessarily stable. But the converse is not true, e.g.~the $\tau_{\leq 0}$-truncation of a pretriangulated dg category is stable, but never pretriangulated. We refer to Section \ref{sec:stable} for more details.
It turns out that a stable dg category has a canonical exact structure, cf.~Proposition~\ref{prop:stablestructure}. 
Based on this, we can characterise when the extriangulated structure in Theorem~\hyperref[intro:truncationextriangulated]{A} is triangulated.
\begin{theoremB}[{Theorem~\ref{thm:stabletriangulated}}]
Let $\A$ be a small exact dg category.
Then the extriangulated structure on $H^0(\A)$ of Theorem~\hyperref[intro:truncationextriangulated]{A} is a triangulated structure if and only if $\A$ is a stable dg category with the canonical exact structure.
\end{theoremB}

In a subsequent paper, we will study in detail the relations between exact dg categories and Barwick exact $\infty$-categories, characterize the class of algebraic extriangulated categories, and prove that each connective exact dg category embeds into its bounded dg derived category. 
In particular, we will show that a connective stable dg category is quasi-equivalent to the $\tau_{\leq 0}$-truncation of a pretriangulated dg category.   

The paper is structured as follows. In Section~\ref{subsection:notations}, we collect basic notations and terminology about dg categories, $\Ai_{\infty}$-categories and homotopy diagrams. In Section~\ref{sec:homotopyshortexactsequence}, we introduce the notion of homotopy short exact sequence in a dg category. We then prove several useful diagram lemmas.
In Section~\ref{sec:exactdgcategory}, we introduce the central notion in this paper: that of exact dg category. We show that for a small exact dg category $\A$, there exists a canonical extriangulated structure $(H^0(\A),\mathbb E,\mathfrak s)$ on $H^0(\A)$.
 In Section~\ref{sec:example}, we illustrate our theory by the example of dg categories of two-term complexes of finitely generated projective modules over a ring.
 In Section~\ref{sec:stable}, we study a special class of additive dg categories: the class of stable dg categories. We show that each stable dg category carries a canonical exact structure whose associated extriangulated structure is triangulated.
 
 Throughout the paper, we assume that all the categories considered are small.

	\section{Recollections on dg categories, $\Ai_{\infty}$-categories and homotopy diagrams}\label{subsection:notations}
	In this section, we collect basic notations and terminology needed in this paper. 
		\subsection{Dg categories}
	Throughout we fix a commutative ring $k$. The standard references for dg categories are \cite{Keller94, Keller06d, Drinfeld04,Toen11, BondalKapranov90}.

	We write $\otimes$ for the tensor product over $k$.  
	We denote by $\C_{dg}(k)$ the dg category of complexes of $k$-modules and by $\C(k)$ the category of complexes of $k$-modules. 
	
	Let $\A$ be a dg $k$-category.
	For two objects $A_1$ and $A_2$, the Hom complex is denoted by $\Hom_{\A}(A_1,A_2)$ or 
	$\A(A_1,A_2)$.	An element $f$ of $\A(A_1,A_2)^{p}$ will be called a {\em homogeneous} morphism of degree $p$ with the notation $|f|=p$. 
	A homogeneous morphism $f:A_1\rightarrow A_2$ is {\em closed} if we have $d(f)=0$.
	
                 Let $M$ be complex of $k$-modules. We put
	\[
	\begin{tikzcd}
	\tau_{\leq 0}M=(\cdots\ar[r]&M^{-2}\ar[r]&M^{-1}\ar[r]&Z^0M\ar[r]&0\ar[r]&\cdots).
	\end{tikzcd}
	\] 
                 For a dg category $\A$, denote by $\tau_{\leq 0}\A$ the dg category with the same objects as $\A$ 
                 and whose morphism complexes are given by
                 \[
                 (\tau_{\leq 0} \A)(A_1, A_2) = \tau_{\leq 0}(\A(A_1, A_2)).
                  \]
                 The composition is naturally induced by that of $\A$. 
                
 A dg category $\A$ is {\em connective} if for each pair of objects $A_1, A_2\in \A$, the 
 complex $\Hom_{\A}(A_1, A_2)$  has cohomology concentrated in non-positive degrees.
It is called {\em strictly connective} if the components of the complex $\Hom_{\A}(A_1, A_2)$ vanish
in all positive degrees.
 Following To\"en, the dg category $\tau_{\leq 0}\A$ is called the {\em connective cover of $\A$}.
 
	We denote by $Z^0(\A)$ the category with the same objects as $\A$ and whose morphism
spaces are defined by
	\[
	(Z^0\A)(A_1,A_2)=Z^0(\A(A_1,A_2)),
	\]
 where $Z^0$ is the kernel of $d:\A(A_1,A_2)^0\rightarrow \A(A_1,A_2)^1$.
 Similarly, we denote by $H^0(\A)$ the category with the same objects as $\A$ and whose 
 morphism spaces are given by 
 \[
 (H^0\A)(A_1,A_2)=H^0(\A(A_1,A_2)),
 \]
 where $H^0$ denotes the zeroth cohomology of the complex.
 
	The {\em oppositie dg category} $\A^{op}$ has the same objects as $\A$ and its morphism spaces are defined by
	\[
	\A^{op}(X,Y)=\A(Y,X);
	\] 
                 the composition of $f\in\A^{op}(Y,X)^{p}$ with $g\in \A^{op}(Z,Y)^{q}$ is given by $(-1)^{pq}gf$.
	By a dg $\A$-module $M$, we mean a right dg $\A$-module, i.e.~a dg functor $\A^{op}\rightarrow \C_{dg}(k)$. 
	We denote by $\C_{dg}(\A)$ the dg category of right dg $\A$-modules. The category of dg 
	$\A$-modules is
	\[
	\C(\A)=Z^0(\C_{dg}(\A)).
	\] 
	The category up to homotopy of dg $\A$-modules is 
	\[
	\mathcal H(\A)=H^0(\C_{dg}(\A)).
	\]

An object $A$ in $\A$ is {\em contractible} if it is a zero object in $H^0(\A)$, or equivalently if we have $\Id_{A}=d(h)$ for some morphism $h:A\rightarrow A$ of degree $-1$.
A dg $\A$-module is {\em contractible} if it is contractible in the dg category $\C_{dg}(\A)$.

 For each object $A\in\A$, we have the right dg module {\em represented by} $A$
 \[
 A^{\wedge}=\Hom_{\A}(-,A).
 \]
 A dg $\A$-module $M$ is {\em representable} if it is isomorphic to $A^\wedge$ for some
 $A\in\A$ and {\em quasi-representable} if it is quasi-isomorphic to $A^\wedge$ for some $A \in \A$.
 
For a dg $\A$-module $M$, put $IM=\Cone(\Id_M)$ and $PM=\Cone(\Id_{\Sigma^{-1}M})$. 
Explicitly, we have $IM=\Sigma M\oplus M$ and $PM=M\oplus \Sigma^{-1}M$ as graded $\A$-modules. 
We have the natural inclusion $i=[0,1]^{\intercal}:M\rightarrow IM$ and projection $p=[1,0]:PM\rightarrow M$.
 
 A dg functor $F:\A\rightarrow \B$ is a {\em quasi-equivalence} if
 \begin{itemize}
 \item[a)] it is {\em quasi-fully faithful}, i.e.~for all objects $A, A'\in A$, the morphism
  \[
 F_{A,A'}:\A(A,A')\rightarrow \B(FA,FA').    
 \] 
 is a quasi-isomorphism, and
 \item[b)] the induced functor  $H^0(F):H^0(\A)\rightarrow H^0(\B)$ is an equivalence of categories.
 \end{itemize}
 
 The {\em Yoneda dg functor}
 \[
 \A\rightarrow \C_{dg}(\A),\;\; A\mapsto A^{\wedge}
 \]
 is fully faithful.
 Let $\A'$ be the full dg subcategory of $\C_{dg}(\A)$ consisting of objects $A^{\wedge}$, $\Cone(\Id_{A^{\wedge}})$ and $\Cone(-\Id_{\Sigma^{-1}A^{\wedge}})$ for each $A\in\A$.
 The dg category $\A'$ has contractible pre-covers and contractible pre-envelopes and the inclusion $\A\rightarrow \A'$ is a quasi-equivalence. 
 Thus, up to quasi-equivalence, we may assume $\A$ has contractible pre-covers and contractible pre-envelopes. 
 
  The category $\dgcat$ of small dg categories admits the Dwyer-Kan model structure (cf.~\cite{Tabuada05}), 
 whose weak equivalences are the quasi-equivalences. Its homotopy category is denoted by $\Hqe$. 
 There exists a cofibrant replacement functor $Q$ on $\dgcat$ such that for any $\A\in\dgcat$, the natural dg functor $Q(\A)\rightarrow \A$ is the identity on the set of objects. 
 In particular, the Hom-complexes of $Q(\A)$ are cofibrant over $k$.
 The {\em tensor product} $\A\otimes \B$ of two dg categories $\A$ and $\B$ has the class of objects  $\obj(\A)\times \obj(\B)$ and the morphism spaces
 \[
 (\A\otimes\B)((A,B),(A',B'))=\A(A,A')\otimes \B(B,B')
 \]
 with the natural compositions and units. 
 This defines a symmetric monoidal structure $-\otimes -$ on $\dgcat$ which is closed.
 For $\A, \B\in\dgcat$, put $\A\otimes^{\mathbb L}\B=\A\otimes Q(\B)$. 
 This extends to a bifunctor $-\otimes^{\mathbb L}-:\dgcat\times \dgcat\rightarrow \dgcat$ and then passes through the homotopy categories
 \[
 -\otimes^{\mathbb L}-:\Hqe\times \Hqe\rightarrow \Hqe.
 \]
 
 Let $\Cat$ be the category of small categories. 
 Let $[\Cat]$ be the category with objects small categories and with morphisms isomorphism classes of functors.
 In particular the isomorphism class of an equivalence of categories is an isomorphism in $[\Cat]$.
 So the functor $H^0:\dgcat\rightarrow \Cat$ induces a functor
 \[
 H^0:\Hqe\rightarrow [\Cat].
 \]
 
 For a category $\C$, we denote by $\Iso(\C)$ the class of its isomorphism classes of objects.
 
 For a dg category $\A$, we denote by $\D(\A)$ its {\em derived category}, a triangulated category. 
 By definition, $\D(\A)$ is the localization of $\C(\A)$ at the class of {\em quasi-isomorphisms}, i.e.~morphisms of dg $\A$-modules which induce quasi-isomorphisms of complexes when evaluated at any object in $\A$.
 Let $\pi:\C(\A)\rightarrow \D(\A)$ be the quotient functor. 
 For a morphism $j:M\rightarrow N$ in $\C(\A)$, we denote by $\overline{\jmath}=\pi(j):M\rightarrow N$ the corresponding morphism in $\D(\A)$.
 A dg functor $F:\A\rightarrow \B$ induces a triangle functor $F_*:\D(\A)\rightarrow \D(\B)$ which is an equivalence of triangulated categories if $F$ is a quasi-equivalence.
 The {\em dg derived category} $\D_{dg}(\A)$ of $\A$ is defined to be the full dg subcategory of $\C_{dg}(\A)$ consisting of cofibrant dg $\A$-modules in the {\em projective model structure} of $\C(\A)$ (cf.~\cite[Theorem 3.2]{Keller06d}).
 The canonical functor $H^0(\D_{dg}(\A))\rightarrow \D(\A)$ is an equivalence of triangulated categories.
 
 For dg categories $\A$ and $\B$, we define $\rep(\B,\A)$ to be the full subcategory of $\D(\A\otimes^{\mathbb L} \B^{op})$ whose objects are the dg bimodules $X$ such that $X(-,B)$ is quasi-representable for each object $B$ of $\B$. 
 We define the canonical dg enhancement $\rep_{dg}(\B,\A)$ to be the full dg subcategory of $\D_{dg}(\A\otimes^{\mathbb L}\B^{op})$ whose objects are those of $\rep(\B,\A)$.
 
 A dg category $\A$ is {\em pretriangulated} if the canonical inclusion $H^0(\A)\rightarrow \D(\A)$ is a triangulated subcategory.
Note that $\rep_{dg}(\B,\A)$ is pretriangulated if $\A$ is pretriangulated.
For a dg category $\A$, its {\em pretriangulated hull} $\pretr(\A)$ (\cite{BondalKapranov90, Drinfeld04, BondalLarsenLunts04}) is defined as follows. 
The objects of $\pretr(\A)$ are ``one-sided twisted complexes'', i.e.~formal expressions $(\oplus_{i=1}^{n}A_i[r_i],q)$, where $A_i\in \A$, $r_i\in\mathbb Z$, $n\geq 0$, $q=(q_{ij})$, $q_{ij}\in \Hom_{\A}^{r_i-r_j+1}(A_j,A_i)$, $q_{ij}=0$ for $i\geq j$ and $dq+q^2=0$. 
Here we adopt the convention $(dq)_{ij}=(-1)^{r_i}d_{\A}(q_{ij})$ and $(q^2)_{ij}=\sum_{k}q_{ik}\circ q_{kj}$.
If $A$ and $A'$ are two objects in $\pretr(\A)$ with $A=(\oplus_{i=1}^{n}A_i[r_i],q)$ and $A'=(\oplus_{i'=1}^{n'}A_{i'}'[r_{i'}'],q')$, the complex $\Hom_{\pretr(\A)}(A,A')$ has as the degree $m$ component the space of matrices $f=(f_{ij})$, $f_{ij}\in\Hom_{\A}^{m+r_{i}'-r_j}(A_{j}, A_{i}')$. 
The differential $d$ carries a morphism $f=(f_{ij})$ of degree $m$ to $df=((df)_{ij})$ where 
\[
(df)_{ij}=(-1)^{r_i'}d_{\A}(f_{ij})+\sum_{k}q'_{ik}\circ f_{kj}-(-1)^{m} \sum_{k} f_{ik}\circ q_{kj}.
\]
The composition map 
\[
\Hom_{\pretr(\A)}(A',A'')\otimes \Hom_{\pretr(\A)}(A,A')\rightarrow \Hom_{\pretr(\A)}(A,A'')
\]
 is the matrix multiplication: $f'\otimes f\mapsto f''$ where $f''_{ij}=\sum_{k}f'_{ik}\circ f_{kj}$.
 We denote the triangulated category $H^0(\pretr(\A))$ by $\tr(\A)$.
The Yoneda dg functor $\A\rightarrow \C_{dg}(\A)$ extends to a fully faithful dg functor $\pretr(\A)\rightarrow \D_{dg}(\A)$.
It induces a fully faithful triangle functor $\tr(\A)\rightarrow \D(\A)$ whose essential image is the triangulated subcategory of $\D(\A)$ generated by the representable dg modules.

A {\em quiver} (or {\em directed graph}) $Q$ is a quadruple $(Q_0,Q_1,s,t)$ where $Q_0$ is a set whose elements are called {\em objects} or {\em vertices} of $Q$, $Q_1$ is a set of {\em arrows} $f$ (or {\em edges}), and $s,t$ are maps $Q_1\rightarrow Q_0$ where $s(f)$ is the {\em source} of the arrow $f$ and $t(f)$ is the {\em target} of $f$. 
A morphism between quivers is defined in the obvious way. 
We denote the category of quivers by $\mathrm{Quiv}$.

Each small category $\C$ has an {\em underlying quiver} $F(\mathcal C)$ and this extends to the forgetful functor $F:\Cat\rightarrow\Quiv$. 
This functor admits a left adjoint $P:\Quiv\rightarrow \Cat$ which associates to a quiver $Q$ the {\em path category} $P(Q)$ of $Q$. 
The path category $P(Q)$ has the same objects as $Q$ and its morphisms $x_1\rightarrow x_n$ are finite paths
\[
\begin{tikzcd}
x_1\ar[r,"f_1"]&x_2\ar[r]&\cdots\ar[r,"f_{n-1}"]&x_n
\end{tikzcd}
\]
consisting of $n\geq 1$ objects $x_1,\cdots,x_n$ of $Q$ which are connected by arrows $f_i:x_i\rightarrow x_{i+1}$ of $Q$. 
The composite of two paths is defined by concatenation: $\mathrm{concat}(f,g)=g\circ f$.

Let $Q$ be a quiver and $R$ a function which assigns to each pair of objects $x,y$ of $Q$ a binary relation $R_{x,y}$ on the set of finite paths from $x$ to $y$. 
Then the path category of the quiver $Q$ with relations $R$ is defined to be the quotient category $P(Q)/\overline{R}$, where $\overline{R}$ is the smallest family of equivalence relations $\overline{R}_{x,y}$, $x,y\in Q_0$, containing $R$ and stable under pre- and postcomposition with morphisms. 
Often we only write down nontrivial relations in $R$.

For a small category $\I$, we have the $k$-category $k\I$ whose set of objects is the same as that of $\I$ and for each pair of objects $x,y$ in $\C$, the space $k\I(i,j)$ is the free $k$-module generated by the set $\I(i,j)$. 
In particular, if we view $k\I$ as a dg category concentrated in degree 0, then it is {\em $k$-cofibrant}, i.e.~the Hom complexes are $k$-cofibrant. 
Let $\B$ be a $k$-cofibrant dg category. 
We have a quasi-equivalence $\A\otimes ^{\mathbb L}\B\iso\A\otimes \B$.
A dg functor $F:\B\rightarrow\A$ gives rise to a $\B$-$\A$-bimodule $_{F}\A_{\A}$ defined by
\[
(A,B)\mapsto \Hom_{\A}(A,FB).
\]
The assignment $F\mapsto  {_{F}\A_{\A}}$ passes to the following map
\[
\Hqe(\B,\A)\rightarrow \Iso(\rep(\B,\A))
\]
which is a bijection, cf.~\cite{Toen07}.
It is shown in loc.~cit~that $\rep_{dg}(\B,\A)$ is the internal Hom of the monoidal category $(\Hqe,-\otimes^{\mathbb L}-)$, i.e.~for small dg categories $\A$, $\B$ and $\C$, we have 
\[
\Hqe(\A\otimes^{\mathbb L}\B,\C)\iso\Hqe(\A,\rep_{dg}(\B,\C)).
\]
The following notion will be useful for the study of exact dg categories.
\begin{definition}\label{def:connective quasi-equivalence}
Let $\A$ and $\B$ be small dg categories and $F:\A\rightarrow \B$ a dg functor. 
It is a {\em connective quasi-equivalence} if the induced dg functor $\tau_{\leq 0}F:\tau_{\leq 0}\A\rightarrow\tau_{\leq 0}\B$ is a quasi-equivalence.
 \end{definition}
Let $\Hqe_{\leq 0}$ be the localization of $\dgcat$ at the class of connective quasi-equivalences.
We denote by $\dgcat_{\leq 0}$ the class of small and strictly connective dg categories.
 and by $\Q{e}$ the class of quasi-equivalences between connective dg categories.
A dg functor $F:\A\rightarrow \B$ in $\dgcat_{\leq 0}$ is a {\em fibration} \cite[Definition 3.8]{Tabuada10} if:  
\begin{itemize}
\item[$\mathrm{(F1)_{\leq 0}}$] for all objects $x$, $y$ in $\A$, the morphism of connective complexes
\[
F(x,y): \A(x,y)\rightarrow \B(x,y)
\]
is surjective in negative degrees and,
\item[$\mathrm{(F2)_{\leq 0}}$] given an object $x$ in $\A$ and a homotopy equivalence $v:Fx\rightarrow y$ in $\B$, there exists a homotopy equivalence $x\rightarrow x'$ in $\A$, such that $F(u)=v$.
\end{itemize}
A dg functor $F:\A\rightarrow \B$ in $\dgcat_{\leq 0}$ is a {\em cofibration} \cite[Definition 3.9]{Tabuada10} if it has the left lifting property with respect to the dg functors $\dgcat_{\leq 0}$ which are simultaneous quasi-equivalences and fibrations.
By \cite[Theorem 3.10]{Tabuada10}, the category $\dgcat_{\leq 0}$ endowed with the above defined notions of weak equivalence, fibration, and cofibration, is a Quillen model category. Moreover, this model structure is cofibrantly generated, cf.~\cite[11.1.2]{Hirschhorn03}.
Let $\Hqe_{\leq 0}'$ be the localization of $\dgcat_{\leq 0}$ at the class of quasi-equivalences.
We have the following lemma, whose proof is similar to that of Lemma~\ref{lem:trun}.
\begin{lemma}\label{lem:trun2}
The adjunction
\[
\begin{tikzcd}
\dgcat_{\leq 0}\ar[r,"i", shift left =0.8ex]&\dgcat,\ar[l,"\tau_{\leq 0}", shift left=0.8ex]
\end{tikzcd}
\]
where the left adjoint $i:\dgcat_{\leq 0}\rightarrow \dgcat$ is the inclusion functor and the right adjoint $\tau_{\leq 0}:\dgcat\rightarrow \dgcat_{\leq 0}$ sends a small dg category $\A$ to its connective cover $\tau_{\leq 0}\A$, induces a pair of equivalences of categories which is quasi-inverse to each other
\[
\begin{tikzcd}
\Hqe_{\leq 0}'\ar[r,"i", shift left =0.8ex]&\Hqe_{\leq 0}.\ar[l,"\tau_{\leq 0}", shift left=0.8ex]
\end{tikzcd}
\]
\end{lemma}

\subsection{$\Ai_{\infty}$-categories}
We recall some basic facts and notation for $\Ai_{\infty}$-categories. Our standard references for $\Ai_{\infty}$-categories are \cite{Keller01,Keller02, Keller06c,Lefevre03}. We follow the sign conventions of \cite{GetzlerJones90}. Note that we keep the assumption that $k$ is a commutative ring.
\begin{definition}An $\Ai_{\infty}$-category $\A$ is the datum of 
\begin{itemize}
\item a class of objects $\obj(\A)$,
\item for all $A, B\in \A$, a $\mathbb Z$-graded $k$-module $\Hom_{\A}(A,B)$, often denoted by $(A,B)$,
\item for all $n\geq 1$ and all $A_0$, \ldots, $A_n$, a graded map
\[
m_n:(A_{n-1},A_n)\otimes(A_{n-2},A_{n-1})\otimes \cdots\otimes(A_0,A_1)\rightarrow (A_0,A_n)
\]
of degree $2-n$
\end{itemize}
such that for each $n\geq 1$ and all $A_0, \ldots, A_n\in\A$, we have the identity
\[
\sum (-1)^{r+st}m_u(\boldsymbol{1}^{\otimes r}\otimes m_s\otimes \boldsymbol{1}^{\otimes t})=0
\]
of maps
\[
(A_{n-1},A_n)\otimes(A_{n-2},A_{n-1})\otimes \cdots\otimes(A_0,A_1)\rightarrow (A_0,A_n)
\]
where the sum runs over all decompositions $n=r+s+t$ and we put $u=r+1+t$.
\end{definition}
\begin{remark}The composition $m_2$ induces an associative composition 
\[
\mu: H^0(\A)(B,C)\otimes H^0(\A)(A,B)\rightarrow H^0(\A)(A,C)
\]
for each triple of objects $(A,B,C)$ in $\A$.
\end{remark}
\begin{definition}Let $\A$ be an $\Ai_{\infty}$-category and $A\in\A$. A morphism $e_{A}\in \Hom^{0}_{\A}(A,A)$ is a {\em strict identity} if we have 
\[
m_2(f,e_A)=f \mbox{ and } m_2(e_A,g)=g
\]
whenever these make sense and 
\[
m_n(\ldots,e_A,\ldots)=0
\]
for all $n\neq 2$.
In particular, $e_A$ is a cycle of the complex $(\Hom_{\A}(A,A),m_1)$. 
Clearly, if $e_A$ exists, it is unique. 
The $\Ai_{\infty}$-category $\A$ {\em has strict identities} if there is a strict identity $e_A$ for each object $A\in \A$.
\end{definition}
\begin{definition}Let $\A$ and $\B$ be two $\Ai_{\infty}$-categories. An $\Ai_{\infty}$-functor $F:\A\rightarrow \B$ is the datum of
\begin{itemize}
\item a map $F:\obj (\A)\rightarrow \obj (\B)$,
\item for all $n\geq 1$ and all $A_0$, \ldots, $A_n\in\A$ a graded map
\[
F_n: (A_{n-1},A_n)\otimes (A_{n-2}, A_{n-1})\otimes \cdots \otimes (A_0,A_1)\rightarrow \Hom_{\B}(FA_0,FA_n)
\]
of degree $1-n$
\end{itemize}
such that 
\[
\sum (-1)^{r+st}F_u(\boldsymbol{1}^{\otimes r}\otimes m_s\otimes \boldsymbol1^{\otimes t})=\sum (-1)^sm_r(F_{i_1}\otimes F_{i_2}\otimes \cdots \otimes F_{i_r})
\]
where the first sum runs over all decompositions $n=r+s+t$, we put $u=r+1+t$, and the second sum runs over all $1\leq r\leq n$ and all decompositions $n=i_1+\cdots+i_r$; the sign on the right hand side is given by
\[
s=(r-1)(i_1-1)+(r-2)(i_2-1)+\cdots+2(i_{r-2}-1)+(i_{r-1}-1).
\]
The $A_{\infty}$-functor $F$ is {\em strict} if we have $F_n=0$ for all $n\geq 2$. Clearly the above datum with $F_n=0$ for $n\geq 2$ defines a strict $A_{\infty}$-functor if and only if for each $k\geq 1$, we have
\[
F_1m_k=m_k(F_1\otimes \cdots\otimes F_1).
\]
\end{definition}
\begin{definition}Let $\A$ and $\B$ be two $\Ai_{\infty}$-categories with strict identities. An $\Ai_{\infty}$-functor $F$ is {\em strictly unital} if 
\[
F_1(e_A)=e_{FA}
\]
and
\[
F_i(\cdots\otimes e_A\otimes\cdots)=0
\]
for all $i\geq 2$.
\end{definition}

Let $\B$ be a small $\Ai_{\infty}$-category and $\A$ an $\Ai_{\infty}$-category. 
It is known, cf.~\cite{Kontsevich98}, that there is a natural $\Ai_{\infty}$-category $\Fun_{\Ai_{\infty}}(\B,\A)$ of strictly unital $\Ai_{\infty}$-functors between $\B$ and $\A$. 
It was constructed in \cite[8.1.3]{Lefevre03} using twists of $\Ai_{\infty}$-structures. 
Keller further developed an idea in \cite{Lyubashenko03} and interpreted it in \cite[5.7]{Keller06c} as an internal $\Hom$-object in the tensor category of cocomplete augmented dg cocategories.
 Note that when the target $\A$ is a dg category, the $\Ai_\infty$-category $\Fun_{\infty}(\B,\A)$ is again a dg category. 
 The reason is that we use the $m_j^{\A}$, $j\geq i$, to define the map $m_i$ of the $\Ai_{\infty}$-category 
$\Fun_{\Ai_{\infty}}(\B,\A)$ and when $\A$ is a dg category, then the $m_{j}^{\A}$ vanish for all $j\geq 3$.
\begin{example}\label{Mor(A)}
The dg category $\Fun_{\Ai_{\infty}}(k\Mor,\A)$ is isomorphic to the {\em dg morphism category} $\Mor(\A)$ introduced in \cite[2.9]{Drinfeld04}. More precisely, $\Mor(\A)$ has as objects the closed morphisms $f:A\rightarrow B$ of degree 0 in $\A$. The morphism complex $\Mor(\A)((A,B,f),(A',B',f'))$ has as the degree $m$ component the space of matrices
$\begin{pmatrix}
j&0\\
h&l
\end{pmatrix}$
where the morphisms $j:A\rightarrow A'$, $l:B\rightarrow B'$ are of degree $m$ and the morphism $h:A\rightarrow B'$ is of degree $m-1$. The differential $d$ carries a morphism 
$\begin{pmatrix}
j&0\\
h&l
\end{pmatrix}$
to
$\begin{pmatrix}
-d(j)&0\\
d(h)+f'j-(-1)^{m}lf&d(l)
\end{pmatrix}$.
The composition map is the multiplication of matrices.
\end{example}

\subsection{Homotopy diagrams}	\label{subsec:homotopydiagrams}
	Throughout we fix a dg $k$-category $\A$. 
	Our first aim is to formulate proper definitions of homotopy (co)cartesian squares in the dg category $\A$ (not  $\D(\A)$).

We are mainly concerned with the case when the small category $\I$ is the category with one object and one morphism or one of the following categories:

The {\em square} category $\mathrm{Sq}$ is the path category of the quiver 
\[
\begin{tikzcd}
00\ar[r,"f"]\ar[d,"g"swap] & 01\ar[d,"j"] \\
10\ar[r,"k"swap] & 11
\end{tikzcd}
\]
with the commutativity relation $jf\sim kg$;

the {\em cospan} category $\mathrm{Cosp}$ is the path category of the quiver
\[
\begin{tikzcd}
&01\ar[d]\\
10\ar[r]&11;
\end{tikzcd}
\]

the {\em span} category $\mathrm{Sp}$ is the path category of the quiver
\[
\begin{tikzcd}
00\ar[r]\ar[d]&01\\
10&
\end{tikzcd};
\]

the {\em composition} category $\mathrm{Com}$ is the path category of the quiver
\[
\begin{tikzcd}
0\ar[r]&1\ar[r]&2;
\end{tikzcd}
\]

the {\em morphism} category $\mathrm{Mor}$ is the path category of the quiver
\[
\begin{tikzcd}
0\ar[r]&1.
\end{tikzcd}
\]

We identify objects $X\in\D( \A \otimes k(\Sq^{op}))$ with commutative squares
\begin{equation}
\begin{tikzcd}\label{D(Sq)}
X_{00}\ar[r,"f"]\ar[d,"g"swap]&X_{01}\ar[d,"j"]\\
X_{10}\ar[r,"k"swap]&X_{11}
\end{tikzcd}
\end{equation}
in $\C(\A)$. 

In general we identify an object $X\in \D(\A\otimes k\I^{op} )$ with an $\I$-shaped diagram $X$ in $\C(\A)$, i.e.~a functor 
\[
X:\I\rightarrow \C(\A),\;\; i\mapsto X_i.
\] 
 Let $Q:\C(\A)\rightarrow \D(\A)$ be the canonical quotient functor. 
The map sending $X$ to $Q\circ X$ defines the canonical {\em diagram functor} $\Dia:\D(\A\otimes k\I^{op})\rightarrow\Fun(\I,\D(\A))$ which is {\em conservative}, i.e.~detects isomorphisms. Note that it is not full in general. 

We have the following diagram with the obvious functors $i:\Cosp\rightarrow\Sq$ and $s:\Mor\rightarrow\Cosp$. 
\[
\begin{tikzcd}
0\ar[d,""{name=1}]&\;\ar[d,""{name=4}, white]&01\ar[d,""{name=2}]&00\ar[rd,phantom,"="]\ar[r,"f"]\ar[d,"g"{name=3,swap}] & 01\ar[d]\ar[d,"j"] \\
1&10\ar[r]&11&10\ar[r,"k"swap] & 11\ar[r, from=2, to=3,red,"i"]\ar[r,from=1,to=4,red,"s"]
\end{tikzcd}\;.
\]

Write $L:\D(\A \otimes k\Sq^{op})\rightarrow \D(\A\otimes k\Cosp^{op})$ for the restriction functor along $i$, $U$ for its left adjoint and $R$ for its right adjoint.
Write $L':\D(\A\otimes k\mathrm{Cosp}^{op})\rightarrow \D( \A\otimes k\mathrm{Mor}^{op})$ for the restriction functor along $s$, $U'$ for its left adjoint and $R'$ for its right adjoint. 

We describe these functors at the level of objects. The functor $U'$ sends an object $Y\xrightarrow{j} Z$ to 
\[
\begin{tikzcd}
&Y\ar[d,"j"]\\
0\ar[r]&Z
\end{tikzcd}\;.
\]
The functor $R'$ sends $Y\xrightarrow{j} Z$ to 
\[
\begin{tikzcd}
&Y\ar[d,"j"]\\
Z\ar[r,equal]&Z
\end{tikzcd}\;.
\]
The functor $U$ sends 
\[
\begin{tikzcd}
&Y\ar[d,"j"]\\
W\ar[r,"k"swap]&Z
\end{tikzcd}
\] 
to 
\[
\begin{tikzcd}
0\ar[r]\ar[d]&Y\ar[d,"j"]\\
W\ar[r,"k"swap]&Z
\end{tikzcd}\;. 
\]
Let us describe the functor $R$. 
Let $S$ be the object
\[
\begin{tikzcd}  & B\ar[d,"{j}"] \\ 
                       C\ar[r,"k"swap] & D
                       \end{tikzcd}
\]
in $\D(\A\otimes k\mathrm{Cosp}^{op})$. 
Then the functor $R$ takes $S$ to the dg $\A\otimes k\Sq^{op}$-module
\[ 
\RHom_{k\mathrm{Cosp^{op}}}(M,S),
\]
where $M$ is the dg $k(\Sq)\otimes k\Cosp^{op}$-module $\Hom_{k(\Sq)^{op}}(i(-),-)$.
We will construct a quasi-isomorphism $M'\rightarrow M$ of bimodules such that $M'(-,j)$ is cofibrant for each $j\in \Sq$.
The dg $k\Cosp^{op}$-module $M(-,00)$ is given by the diagram
\[
\begin{tikzcd}
&k\ar[d,"\Id"]\\
k\ar[r,"\Id"swap]&k
\end{tikzcd}\;.
\]
It fits into a short exact sequence in $\C(k\mathrm{Cosp}^{op})$
\[
0\rightarrow 11^{\wedge}\rightarrow 01^{\wedge}\oplus 10^{\wedge}\rightarrow M(-,00)\rightarrow 0.
\]
So it is quasi-isomorphic to the cokernel $N$ of the graded-injective map 
\[
11^{\wedge}\rightarrow \Cone(\Id_{11^{\wedge}})\oplus 10^{\wedge}\oplus 01^{\wedge}.
\]
Let $M'$ be the dg $k(\Sq)\otimes k\Cosp^{op}$-module given by the following diagram in $\C(k\Cosp^{op})$
\[
\begin{tikzcd}
N&10^{\wedge}\ar[l]\\
01^{\wedge}\oplus \Cone(\Id_{11^{\wedge}})\ar[u]&11^{\wedge}\ar[l]\ar[u]\mathrlap{\;.}
\end{tikzcd}
\]
Then we get the desired quasi-isomorphism $M'\rightarrow M$.
So the functor $R$ takes $S$ to $\Hom_{k\Cosp^{op}}(M',S)$ which is given by the following diagram in $\C(\A)$
\[
\begin{tikzcd}
K\ar[r]\ar[d]&B\ar[d,"{j}"]\\
C\oplus \Cone(-\Id_{\Sigma^{-1}D})\ar[r,"(k\ p)"swap]&D
\end{tikzcd}
\]
such that $p:\Cone(-\Id_{\Sigma^{-1}D})\rightarrow D$ is the canonical projection map and that the diagram fits into the following graded split short exact sequence in $\C(\A)$
\[
0\rightarrow K\rightarrow \Cone(-\Id_{\Sigma^{-1}D})\oplus B\oplus C\xrightarrow{[p\;\;j\;\;k]} D\rightarrow 0.
\]

In summary, we have the following diagram of triangle functors 
\begin{equation}
\begin{tikzcd}\label{adjointtriple}
\D(\A\otimes k\mathrm{Mor}^{op})\ar[shift left =2ex, rr,"U'"{description}]\ar[shift right=2ex,rr,"R'"{description}]
&&\D(\A\otimes k\mathrm{Cosp}^{op} )\ar[ll,"L'"{description}]\ar[shift left=2ex, rr,"U"{description}]\ar[shift right=2ex, rr,"R"{description}]
&&\D(\A \otimes k\mathrm{Sq}^{op}).\ar[ll,"L"{description}]
\end{tikzcd}
\end{equation}

\begin{lemma}[\cite{KellerNicolas13,BeligiannisReiten07}] \label{lemma:tstructure}
 Let $\A$ be a dg category.
There is a canonical t-structure 
\[
(\D(\A)^{\leq 0},\D(\A)^{\geq 0})
\]
on $\D(\A)$ such that $\D(\A)^{\geq 0}$ is formed by those dg modules whose cohomology is concentrated in non-negative degrees. 
\end{lemma}
Note that we have $A^{\wedge}\in\D(\A)^{\leq 0}$ for each $A\in\A$.
Note also that when $\A$ is connective, then $\D(\A)^{\leq 0}$ is formed by those dg modules whose cohomology is concentrated in non-positive degrees.

Let $\B$ be a connective dg $k$-category. 
Put $\T=\D(\A\otimes^{\mathbb L} \B^{op})$. 
Let $(\T^{\leq 0},\T^{\geq 0})$ be the canonical t-structure on $\T$. 
Let $X$ be a dg $\B$-$\A$-bimodule. It is clear that $X\in \T^{\geq 0}$ if and only if $X(-,B)\in \D(\A)^{\geq 0}$ for each $B\in\mathcal B$.
\begin{lemma}\label{cha}
We have $X\in \T^{\leq 0}$ if $X(-,B)\in \D(\A)^{\leq 0}$ for each $B\in\B$.
\end{lemma}
\begin{proof}
Suppose $X(-,B)\in \D(\A)^{\leq 0}$ for each $B\in\B$. Let $Y\in \T^{\geq 1}$. 
Then we have $\Hom_{\D(\A)}(X(-,B),Y(-,B'))=0$ for $B,B'\in\B$. 
So, we have $\tau_{\leq 0}\RHom_{\A}(X,Y)=0$. Therefore, we have
\[
\begin{aligned}\RHom_{\A \otimes \B^{op}}(X,Y)=&\RHom_{\B\otimes \B^{op}}(\B,\RHom_{\A}(X,Y))\\
=&\RHom_{\B\otimes \B^{op}}(\B,\tau_{\leq 0}\RHom_{\A}(X,Y))\\
=&0.
\end{aligned}
\]
Hence $X\in\T^{\leq 0}$.
\end{proof}
In particular, for a small category $\mathcal I$, the category $\rep(k\I, \A)$ is a full subcategory of $\D(\A\otimes k\I^{op})^{\leq 0}$.
Recall that for a small category $\I$, the canonical diagram functor 
\[
\Dia:\D(\A \otimes k\I^{op})\rightarrow \mathrm{Fun}(\I,\D({\A}))
\]
sends an object $X:\I\rightarrow \C(\A)$ to the corresponding diagram $\Dia(X)=\pi \circ X$ where $\pi:\C(\A)\rightarrow \D(\A)$ is the canonical quotient functor. 
Recall also that, for a category $\C$, we denote by $\Iso(\C)$ the class of isomorphism classes of objects in $\mathcal{C}$. 
The following lemma is well-known. 
\begin{lemma}\label{epi}
Let $\I$ be the path category $P(Q)$ of a quiver $Q$ (without relations).
The functor $\Dia:\D(\A\otimes k\I^{op})\rightarrow \Fun(\I,\D(\A))$ is an epivalence, i.e.~it is full and dense and detects isomorphisms. 
In particular, it induces a bijection 
\[
\Iso(\D(\A\otimes k\I^{op}))\xrightarrow{\sim}\Iso(\mathrm{Fun}(\I,\D(\A))).
\]
\end{lemma}

\section{Homotopy pullbacks and homotopy pushouts}\label{sec:homotopyshortexactsequence}
In this section, we introduce the notions of homotopy (co)cartesian square over a dg $k$-category $\A$ using the category $\rep(\B,\A)$ of representations up to homotopy, where $\B$ is a certain dg category. We then introduce the category $\H_{3t}(\A)$ of $3$-term homotopy complexes over $\A$, using (strictly unital) $\Ai_{\infty}$-functors.
Building upon a result of Canonaco--Ornaghi--Stellari and Faonte concerning the dg category $\Fun_{\Ai_{\infty}}(\B,\A)$ of $\Ai_{\infty}$-functors and $\rep_{dg}(\B,\A)$, we interpret the notions of homotopy (co)cartesian square in $\A$ into the notions of homotopy left (right) exact sequence in $\A$. 
Additionally, we establish several diagram lemmas that will prove useful in the subsequent sections.
\subsection{Homotopy pullbacks and homotopy pushouts}
\begin{definition}\label{maindef}
An object $X\in \rep(k\mathrm{Sq},\A)$ is a \text{\em{{homotopy cartesian square}}} with respect to $\A$ if the canonical map $X_{00}\rightarrow \Sigma^{-1}\mathrm{Cone}((-j,k))$ induces an isomorphism   
\[
\tau_{\leq 0}\RHom(A^{\wedge},X_{00})\rightarrow \tau_{\leq 0}\RHom(A^{\wedge}, \Sigma^{-1}\mathrm{Cone}((-j,k)))
\]
in $\D(k)$ for each $A$ in $\A$.
\end{definition}

\begin{remark}\label{subcategory}
Let $\A'$ be a full dg subcategory of $\A$. 
Let  $F:\rep(k\Sq,\A')\rightarrow \rep(k\Sq,\A)$ be the inclusion functor induced by the inclusion dg functor $\A'\rightarrow \A$.
Let $X$ be an object in $\rep(k\Sq,\A')$. 
If $F(X)$ is homotopy cartesian with respect to $\A$, then $X$ is homotopy cartesian with respect to $\A'$. 
\end{remark}
In the rest of the chapter, we will only use this relative version of homotopy (co)cartesian squares.

For an object $X\in \rep(k\mathrm{Sq},\A)$, by Lemma \ref{cha}, the unit morphism $X\rightarrow RL X$ induces a canonical morphism $X\rightarrow \tau_{\leq 0} RLX$ in $\D(\A\otimes k\mathrm{Sq}^{op})$.

\begin{lemma} \label{adj}Let $X$ be an object in $\rep(k\mathrm{Sq},\A)$. If the canonical map $X\rightarrow \tau_{\leq 0} RLX$ is an isomorphism, then $X$ is homotopy cartesian. If $\A$ is a connective dg category, then the converse also holds.
\end{lemma}
\begin{proof}
Let $X$ be the object in $\rep(k\mathrm{Sq},\A)$ corresponding to the following commutative square in $\C(\A)$
\[ \begin{tikzcd}A\ar[r,"f"]\ar[d,"g"swap,""{name=2}]&B\ar[d,"j"]\\
C\ar[r,"k"swap]&D
\end{tikzcd}.
\]
Let $p:\Cone(-\Id_{\Sigma^{-1}D})\rightarrow D$ be the canonical projection. 
Then the canonical map $X\rightarrow RLX$ can be described by the following commutative diagram in $\C(\A)$:
\[
\begin{tikzcd}
A\ar[r,"\begin{bmatrix}f\\g\end{bmatrix}"]\ar[d,"u"swap]&B\oplus C\ar[d,"\begin{bmatrix}0\ 1\ 0\\0\ 0\ 1\end{bmatrix}"]\ar[r,"{[}-j{,}k{]}"]&D\ar[d,equal]\\
K\ar[r, tail]&\Cone(-\Id_{\Sigma^{-1}D})\oplus B\oplus C\ar[r,"{[}p{,}-j{,}k{]}"swap, two heads]&D
\end{tikzcd}
\]
where the second row is a short exact sequence in $\C(\A)$. 
By Lemma \ref{cha}, we find that $(\tau_{\leq 0} RLX)_{00}\iso \tau_{\leq 0}K$.
We also have a canonical commutative diagram in $\D(\A)$
\[
\begin{tikzcd}
A\ar[r]\ar[d,"u"swap]&\Sigma^{-1}\Cone(-j,k)\ar[d,"\sim"]\\
K\ar[r,"\sim"]&\Sigma^{-1}\Cone(p,-j,k)
\end{tikzcd}.
\]
The map $u$ induces a canonical map $A\rightarrow \tau_{\leq 0}K$ in $\D(\A)$. 
It is an isomorphism in $\D(\A)$ if and only if the canonical map $X\rightarrow \tau_{\leq 0}RLX$ is an isomorphism in $\D(\A \otimes k\Sq^{op})$.

The object $X$ is a homotopy cartesian square if and only if for each object $A'$ in $\A$, the canonical map $u:A\rightarrow K$ induces an isomorphism in $\D(k)$ 
\[
\tau_{\leq 0}\RHom(A'^{\wedge},A)\iso \tau_{\leq 0}\RHom(A'^{\wedge}, K). 
\]
Note that we have
\[
\tau_{\leq 0}\RHom(A'^{\wedge},\tau_{\leq 0}K) \iso \tau_{\leq 0}\RHom(A'^{\wedge}, K).
\]
So if the canonical map $A\rightarrow \tau_{\leq 0}K$ is an isomorphism in $\D(\A)$, then $X$ is homotopy cartesian. 

When the dg category $\A$ is connective, we have 
\[
\tau_{\leq 0}\RHom(A'^{\wedge},A)\iso \RHom(A'^{\wedge},A)
\]
and
\[
\tau_{\leq 0}\RHom(A'^{\wedge},\tau_{\leq 0}K)\iso \RHom(A'^{\wedge},\tau_{\leq 0}K).
\]
So the converse also holds.

\end{proof}

\begin{example}
Each graded-split short exact sequence  $0\rightarrow A\xrightarrow{f} B\xrightarrow{j} C\rightarrow 0$ in $\mathcal A$ gives rise to a homotopy cartesian square 
\[
\begin{tikzcd}
A^{\wedge}\arrow[r,"f^{\wedge}"]\arrow[d]&B^{\wedge}\arrow[d,"j^{\wedge}"]\\
0\arrow[r]&C^{\wedge}
\end{tikzcd}.
\]
Indeed, the canonical map $A^{\wedge}\rightarrow \Sigma^{-1}\Cone(j^{\wedge})$ is already a quasi-isomorphism.
\end{example}
\begin{example}\label{ordinary}
Let $\A$ be an additive category which we consider as a dg category concentrated in degree zero. 
Then the square
\[
 \begin{tikzcd}
 A^{\wedge}\arrow[r,"f^{\wedge}"]\arrow[d]&B^{\wedge}\arrow[d,"j^{\wedge}"]\\0\arrow[r]&C^{\wedge}
 \end{tikzcd}
 \]
  with $A,B,C\in \A$ is homotopy cartesian if and only if in the sequence 
  \[
  \begin{tikzcd}0\arrow[r]&A\arrow[r,"f"]&B\arrow[r,"j"]&C\end{tikzcd},
  \]
 $f$ is a kernel of $j$. Indeed, we identify the pretriangulated hull $\pretr(\A)$ of $\A$ with the dg category $\C^{b}_{dg}(\A)$ of bounded complexes over $\A$. 
 Then the object $\Sigma^{-1}\mathrm{Cone}(j^{\wedge})$ is identified with the complex $\ldots\rightarrow0\rightarrow B\rightarrow C\rightarrow 0\rightarrow \ldots$ where $B$ is in degree $0$. 
  For each $A'\in \A$, the space $\tau_{\leq 0}\Hom(A'^{\wedge},\Sigma^{-1}\mathrm{Cone}(j^{\wedge}))$ is then identified with the kernel of the map $\Hom_{\A}(A',B)\rightarrow \Hom_{\A}(A',C)$.
  The conclusion follows at once.
\end{example}

The following lemma was obtained independently by Genovese, Lowen and Van den Bergh, c.f. \cite[Lemma 2.5.5]{GenoveseLowenVandenBergh22}. 
\begin{lemma}\label{lem:trun}
Let $\B$ be a connective dg category. Then the natural functor 
\[
F: \rep(\B,\tau_{\leq 0}\A)\rightarrow\rep(\B,\A)
\]
is an equivalence of categories. 
Let $H:\D(\A\otimes^{\mathbb L}\B^{op})\rightarrow \D(\tau_{\leq 0}\A\otimes^{\mathbb L} \B^{op})$ be the restriction functor. 
Then a quasi-inverse of $F$ is given by $G=\tau_{\leq 0}\circ H$. 
Moreover, for any $X, Y\in \rep(\B,\tau_{\leq 0}\A)$, we have $\tau_{\leq 0}\RHom(X, Y)\xrightarrow{\sim}\tau_{\leq 0}\RHom(FX,FY)$. 
\end{lemma} 
\begin{proof}
Let $\dgcat_{\leq 0}$ be the category of small strictly connective dg categories. 
Then we have the following adjunction
\[
\begin{tikzcd}
\dgcat_{\leq 0}\ar[r,"i", shift left =0.8ex]&\dgcat\ar[l,"\tau_{\leq 0}", shift left=0.8ex]
\end{tikzcd}
\]
where the left adjoint $i:\dgcat_{\leq 0}\rightarrow \dgcat$ is the inclusion functor and the right adjoint $\tau_{\leq 0}:\dgcat\rightarrow \dgcat_{\leq 0}$ sends a small dg category $\A$ to $\tau_{\leq 0}\A$.

Let $\Hqe_{\leq 0}'$ be the localization of $\dgcat_{\leq 0}$ at the class of quasi-equivalences of strictly connective dg categories.
The functors $i$ and $\tau_{\leq 0}$ clearly preserve quasi-equivalences. 
Hence by the universal properties of localizations of categories, the above adjunction induces the following adjunction
\[
\begin{tikzcd}
\Hqe_{\leq 0}'\ar[r,"i", shift left =0.8ex]&\Hqe\ar[l,"\tau_{\leq 0}", shift left=0.8ex]
\end{tikzcd}
\]
where $i$ is fully faithful.

We claim that the canonical morphism in $\Hqe_{\leq 0}$
\[
\tau_{\leq 0}\rep_{dg}(\B,\tau_{\leq 0}\A)\rightarrow \tau_{\leq 0}\rep_{dg}(\B,\A)
\]
is an isomorphism.

Let $\C$ be an arbitrary small strictly connective dg category.
Recall that the tensor product of strictly connective dg categories remains strictly connective, and 
that each connective dg category admits a cofibrant replacement which is strictly connective.
Recall also that $\rep_{dg}(\B,\A)$ is the internal Hom of the monoidal category $(\Hqe,-\otimes^{\mathbb L}-)$.
Then we have
\[
\begin{aligned}
\Hom_{\Hqe_{\leq 0}'}(\C,\tau_{\leq 0}\rep_{dg}(i(\B),i(\tau_{\leq 0}\A)))
&=\Hom_{\Hqe}(i(\C),\rep_{dg}(i(\B),i(\tau_{\leq 0}\A)))\\
&=\Hom_{\Hqe}(i(\C)\otimes^{\mathbb L}i(\B),i(\tau_{\leq 0}\A))\\
&=\Hom_{\Hqe_{\leq 0}'}(\C\otimes^{\mathbb L}\B,\tau_{\leq 0}\A)\\
&=\Hom_{\Hqe}(i(\C)\otimes^{\mathbb L}i(\B),\A)\\
&=\Hom_{\Hqe}(i(\C),\rep_{dg}(i(\B),\A))\\
&=\Hom_{\Hqe_{\leq 0}'}(\C,\tau_{\leq 0}\rep_{dg}(i(\B),\A)).
\end{aligned}
\]

\end{proof}
The following lemma allows us to reduce, in many cases, to the connective case when considering homotopy (co)cartesian squares.
\begin{lemma}\label{truncationhomotopycartesian}
Consider the equivalence of categories $F:\rep(k\Sq,\tau_{\leq 0}\A)\iso \rep(k\Sq,\A)$ as in Lemma \ref{lem:trun}.
An object $X$ in $\rep(k\mathrm{Sq},\tau_{\leq 0}\A)$ is homotopy cartesian
 if and only if its image under $F$ in $\rep(k\mathrm{Sq},\A)$ is homotopy cartesian.
\end{lemma}
\begin{proof} Let $X$ be the commutative square of quasi-representable dg $\tau_{\leq 0}\A$-modules given by the diagram (\ref{D(Sq)}).
Consider the following diagram 
\[
\begin{tikzcd}
X_{00}\ar[d,"{[}f{,}g{,}0{]}^{\intercal}"swap]\ar[rd,"{[}f{,}g{]}^{\intercal}"]&&&\\
\Sigma^{-1}\mathrm{Cone}(\psi)\ar[r]&X_{10}\oplus X_{01}\ar[r,"\psi={[}-j{,}k{]}"]& X_{11}\ar[r]&\mathrm{Cone}(\psi)
\end{tikzcd}\;.
\] 
By applying $\Hom_{\D(\tau_{\leq 0}\A)}(A^{\wedge},-)$ to the above triangle for each $A\in \tau_{\leq 0}\A$, we get the associated long exact sequence 
\[
\begin{tikzcd}
&&\Hom(A^{\wedge},X_{00})\ar[d]&&\\
\dots\ar[r]&\Hom(A^{\wedge},\Sigma^{-1}X_{11})\ar[r]&\Hom(A^{\wedge},\Sigma^{-1}\Cone(\psi)) \ar[r]&\Hom(A^{\wedge},X_{10}\oplus X_{01})\ar[r]&\dots
\end{tikzcd}
\]
We have $\Hom(A^{\wedge}, \Sigma^{-n}X_{ij})\iso\Hom(FA^{\wedge},\Sigma^{-n}FX_{ij})$ for $n\geq 0$, where $F:\D(\tau_{\leq 0}\A)\rightarrow \D(\A)$ is the induction functor.
By the Five-Lemma, it is then clear that an object $X$ in $\rep(k\mathrm{Sq},\tau_{\leq 0}\A)$ is homotopy cartesian if and only if its image under $F$ in $\rep(k\mathrm{Sq},\A)$ is homotopy cartesian.
\end{proof}
\begin{definition}\label{homotopypullbackdef}
An object $S$ in $\rep(k\mathrm{Cosp},\A)$ is said to admit a {\em homotopy pullback} if there exists a homotopy cartesian square $X$ in $\rep(k\mathrm{Sq},\A)$ with an isomorphism $\phi:LX\rightarrow S$ in $\D(\A\otimes k\mathrm{Cosp}^{op})$. 
In this case, the pair $(X,\phi)$ is called a homotopy pullback of $S$. 
We will also abuse the notation and say that $X$ is a {\em homotopy pullback} of $S$.
\end{definition}

\begin{remark}\label{rmk:isomorphism}
The isomorphism $\varphi:LX\rightarrow S$ in the definition of a homotopy pullback is in general not the identity. 
For example, let $\A$ be the dg category of two-term complexes of finitely generated projective modules over a finite dimensional algebra $A$. 
The cospan
\[
\begin{tikzcd}
&A\ar[d]\\
0\ar[r]&\Cone(\Id_A)
\end{tikzcd}
\]
where $A$ denotes the corresponding stalk complex, does not admit a homotopy pullback $(X,\phi)$ with $\phi$ being the identity. 
It admits a homotopy pullback with $X$ given by
\[
\begin{tikzcd}
A\ar[r,equal]\ar[d]&A\ar[d]\\
0\ar[r]&0
\end{tikzcd}
\]
and $\phi^{-1}$ given by the obvious morphism 
\[
\begin{tikzcd}
         &A\ar[d]                                  &\ar[d, white,""{name=4}]&  &\ar[d,white,""{name=3}]&           &A\ar[d]\ar[r,blue,from=4,to=3]\\
0\ar[r]&\Cone({\Id_A})       &\                                     &  &\                                    &0\ar[r] &0
\end{tikzcd}.
\]  
\end{remark}
Obviously, the restriction dg functor $L:\D_{dg}(\A\otimes k\mathrm{Sq}^{op})\rightarrow \D_{dg}(\A\otimes k\mathrm{Cosp}^{op})$ restricts to a dgfunctor $L:\rep_{dg}(k\Sq,\A)\rightarrow \rep_{dg}(k\Cosp,\A)$.
\begin{proposition}\label{pullbackunique}
The dg functor $\tau_{\leq 0}L:\tau_{\leq 0}\rep_{dg}(k\Sq,\A)\rightarrow \tau_{\leq 0}\rep_{dg}(k\Cosp,\A)$ restricted to
 the full subcategory consisting of homotopy cartesian squares is quasi-fully faithful. 
In particular, the homotopy pullback of an object $S$ in $\rep(k\mathrm{Cosp},\A)$, 
if it exists, is unique up to a unique isomorphism. 
\end{proposition}
\begin{proof}
By Lemmas~\ref{lem:trun} and~\ref{truncationhomotopycartesian}, we may assume that the dg category $\A$ is connective. 
Let $X$ and $X'$ be two homotopy cartesian squares. 

Let $f:\Sigma^{i} LX\rightarrow LX'$ be a morphism in $\D(\A\otimes k\Cosp^{op})$, where $i\geq 0$.
From the adjunction 
\[
\begin{tikzcd}
\D(\A\otimes k{\mathrm{Sq}}^{op})\arrow[shift left=1ex, r, "L"]&\D(\A\otimes k\mathrm{Cosp}^{op})\arrow[shift left=1ex, l,"R"] 
\end{tikzcd},
\]
we have a morphism $g:\Sigma^{i}X\rightarrow RLX'$. 
Since the object $X'$ is homotopy cartesian, by Lemma \ref{adj}, 
the map $\Sigma^{i}X\rightarrow RLX'$ induces a morphism $h:\Sigma^{i}X\rightarrow \tau_{\leq 0}RLX'\xleftarrow{\sim} X'$.
It is easy to see that $L(h)=f$.

Let $h:\Sigma^{i}X\rightarrow X'$ be a morphism in $\D(\A\otimes k\Sq^{op})$, where $i\geq 0$.
We have the following diagram in $\D(\Sq)$:
\[
\begin{tikzcd}
\Sigma^{i}X\ar[r]\ar[d,"h"swap]&\tau_{\leq 0}(\Sigma ^{i}RLX)\ar[d]\ar[r]&\Sigma^{i} RLX\ar[d,"RL(h)"]\\
X'\ar[r,"\sim"]&\tau_{\leq 0}RLX'\ar[r]&RLX'.\\
\end{tikzcd}
\] 
Therefore, if the morphism $L(h)$ is zero, then $h$ is also zero.

\end{proof}
\begin{remark}[Universal property of homotopy pullbacks]
\label{pullbackuniversal}
Similarly, one can show that given a homotopy cartesian square $X'$ and an object $X$ in $\rep(k\Sq,\A)$, if we have a map $\varphi: LX\rightarrow LX'$ in $\rep(k\mathrm{Cosp},\A)$, then we have a canonical map $\theta:X\rightarrow X'$ in $\rep(k\mathrm{Sq},\A)$ such that $L(\theta)=\varphi$. 
\end{remark}

Recall that we have the pair of duality functors 
\[
\begin{tikzcd}
\rep(k\I,\A)\arrow[shift left=1ex,rr,"\mathrm{RHom}(- \text{,} \mathcal{A})"]&&\rep(k\I^{op},\A^{op})^{op}\arrow[shift left=1ex, ll,"\RHom(-\text{,}\A)"]
\end{tikzcd}
\]
Then we can form the dual definitions of a homotopy cocartesian square with respect to $\A$ and of a homotopy pushout of an object in $\D(\A\otimes k\mathrm{Sp}^{op})$.
\begin{definition}
An object $X\in \rep(k\mathrm{Sq},\A)$ given by the diagram (\ref{D(Sq)}) is a \text{\em{{homotopy}}} \text{\em{{cocartesian square with respect to $\A$}}} if the canonical map $\mathrm{Cone}((f,g)^{\intercal})\rightarrow X_{11}$ induces an isomorphism in $\D(k) $  
\[
\tau_{\leq 0}\RHom(X_{11},A^{\wedge})\rightarrow \tau_{\leq 0}\RHom(\mathrm{Cone}((f,g)^{\intercal}),A^{\wedge})
\]
for each $A$ in $\A$.

\end{definition}
Consider the adjunction \begin{tikzcd}[column sep=huge]
 \D(\A\otimes k\mathrm{Sp}^{op})
  \arrow[shift left=1ex,"L"]{r}[name=D]{} &
 \D(\A\otimes k\mathrm{Sq}^{op})
 \arrow[shift left=1ex,"R"]{l}[name=U]{}
 \end{tikzcd}
induced by the inclusion $\mathrm{Sp}\rightarrow\mathrm{Sq}$. 
\begin{definition}
An object $S$ in $\rep(k\mathrm{Sp},\A)$ is said to admit a {\em homotopy pushout} if there is a homotopy cartesian square $X$ in $\rep(k\mathrm{Sq},\A)$ with an isomorphism $\phi:S\rightarrow RX$ in $\D(\A\otimes k\mathrm{Sp}^{op})$. In this case, the pair $(X,\phi)$ is called a {\em homotopy pushout of} $S$. We will also abuse the notation and say that $X$ is a {\em homotopy pushout} of $S$.
\end{definition}

\subsection{An $A_{\infty}$-theoretic description} 
Our main aim in this subsection is to provide a description of the {\em homotopy
category $\mathcal H_{3t}(\A)$ of homotopy $3$-term complexes} using the $\Ai_{\infty}$-formalism, cf.~Definition~\ref{def:3termhomotopy}. 
Building upon Theorem~\ref{thm:Ainfinityrep}, we establish the existence of a canonical fully faithful functor $\mathcal H_{3t}(\A)\hookrightarrow\rep(k\Sq,\A)$, as elaborated in Lemma~\ref{equivalences}.
Additionally, for later use, we introduce the category $\mathcal H_{\Cosp}(\A)$ of homotopy cospans and the category $\mathcal H_{\Sq}(\A)$ of homotopy squares in $\A$, detailed in Definitions~\ref{def:homotopyspans} and \ref{def:homotopysquares}.
They are equivalent to the categories $\rep(k\Cosp,\A)$ and $\rep(k\Sq,\A)$, respectively.
In the remaining part of this subsection, we interpret the definitions from the previous subsection using these notions.

Let $\A$ be a dg category and $\B$  the dg $k$-path category of the following graded quiver with relations
\begin{equation}\label{3term}
\begin{tikzcd}
0\ar[r,"a"]&1\ar[r,"b"]&2
\end{tikzcd}
\end{equation}
where $|a|=|b|=0$, $d(a)=0$, $d(b)=0$ and $ba=0$.

Our aim is to describe the dg category $\Fun_{\Ai_{\infty}}(\B,\A)$ following the construction in \cite{Lefevre03}. It has strictly unital $\Ai_{\infty}$-functors as its objects and strictly unital morphisms as its morphisms.
We identify a strictly unital $\Ai_{\infty}$-functor $F:\B\rightarrow\A$ with the corresponding diagram in $\A$
\begin{equation}\label{F}
\begin{tikzcd}
&A_0\ar[r,"f"]\ar[rr,bend right = 8ex,"h"swap]&A_1\ar[r,"j"]&A_2
\end{tikzcd}
\end{equation}
where $|f|=|j|=0$, $|h|=-1$ and $d(f)=0$, $d(j)=0$ and $d(h)=-jf$.
Let $F$ and $F'$ be two strictly unital $\Ai_{\infty}$-functors from $\B$ to $\A$. 
We identify them with their corresponding diagrams. 
A strictly unital morphism $H:F\rightarrow F'$ of degree $n$ is identified with the following diagram
\begin{equation}\label{mordegreen}
\begin{tikzcd}
&A_0\ar[r,"f"]\ar[d,"r_0"swap]\ar[rd,"s_1"red,swap ,red]\ar[rrd,"t"blue,blue]\ar[rr,bend left = 8ex,"h"]&A_1\ar[d,"r_1"swap]\ar[rd,"s_2"red,red]\ar[r,"j"]&A_2\ar[d,"r_2"]\\
&A_0'\ar[r,"f'"swap]\ar[rr,bend right = 8ex,"h'"swap]&A_1'\ar[r,"j'"swap]&A_2'
\end{tikzcd}
\end{equation}
where $|r_i|=n$ for $i=0,1,2$, $|s_1|=|s_2|=n-1$, $|t|=n-2$.
Let $H\in \Hom^n(F,F')$. Put $H'=d(H)\in \Hom^{n+1}(F,F')$. 
Through the identification with diagrams, we observe that the corresponding diagram for $d(H)$ is as follows:
\[
 \begin{tikzcd}
&A_0\ar[r,"f"]\ar[d,"r_0'"swap]\ar[rd,"s_1'"red,swap ,red]\ar[rrd,"t'"blue,blue]\ar[rr,bend left = 8ex,"h"]&A_1\ar[d,"r_1'"swap,near end]\ar[rd,"s_2'"red,red]\ar[r,"j"]&A_2\ar[d,"r_2'"]\\
&A_0'\ar[r,"f'"swap]\ar[rr,bend right = 8ex,"h'"swap]&A_1'\ar[r,"j'"swap]&A_2'
\end{tikzcd}
 \]
where 
\[
r_i'=d(r_i)
\]
 for $i=0,1,2$, and
 \[
 s_1'=d(s_1)+(-1)^{n+1}(f'\circ r_0)+(-1)^{n}r_1\circ f,
 \]
 \[
 s_2'=d(s_2)+(-1)^{n+1}(j'\circ r_1)+(-1)^{n}r_2\circ j,
 \]
 and 
 \[
 t'=d(t)+(-1)^{n}(j'\circ s_1)+h'\circ r_0+(-1)^{n}(s_2\circ f)+(-1)^{n+1}(r_2\circ h).\]
 Clearly we have $d^2(H)=0$.
 
 We now describe the compositions. 
 Let $F$, $F'$ and $F''$ be strictly unital $\Ai_{\infty}$-functors from $\B$ to $\A$. Let $H'\in \Hom^r(F,F')$ and $H''\in \Hom^s(F',F'')$ be two morphisms in the dg category $\Fun_{\Ai_{\infty}}(\B,\A)$. 
Assume the corresponding diagram is given by
\[
\begin{tikzcd}
&A_0\ar[r,"f"]\ar[d,"r_0'"swap]\ar[rd,"s_1'"red,swap ,red]\ar[rrd,"t'"blue,blue]\ar[rr,bend left = 8ex,"h"]&A_1\ar[d,"r_1'"swap,near end]\ar[rd,"s_2'"red,red]\ar[r,"j"]&A_2\ar[d,"r_2'"]\\
&A_0'\ar[r,"f'"{description}]\ar[d,"r_0''"swap]\ar[rrd,"t''"blue,blue]\ar[rd,"s_1''"red,swap,red] &A_1'\ar[rd,"s_2''"red, red]\ar[d,"r_1''"swap, near end]\ar[r,"j'"{description}]&A_2'\ar[d,"r_2''"]\\
&A_0''\ar[r,"f''"swap]\ar[rr,bend right = 8ex,"h''"swap]&A_1''\ar[r,"j''"swap]&A_2''
\end{tikzcd},
\]
where, in order to keep the diagram readable, we omit the arrow $h':A_0'\rightarrow A_2'$.
We denote by $\tilde{H}\in \Hom^{r+s}(F,F'')$ the composition of $H$ and $H'$. 
 Then the corresponding diagram for $\tilde{H}$ is given as follows:
  \[
\begin{tikzcd}
&A_0\ar[r,"f"]\ar[d,"\tilde{r_0}"swap]\ar[rd,"\tilde{s_1}"red,swap ,red]\ar[rrd,"\tilde{t}"blue,blue]\ar[rr,bend left = 8ex,"h"]&A_1\ar[d,"\tilde{r_1}"swap,near end]\ar[rd,"\tilde{s_2}"red,red]\ar[r,"j"]&A_2\ar[d,"\tilde{r_2}"]\\
&A_0''\ar[r,"f''"swap]\ar[rr,bend right = 8ex,"h''"swap]&A_1''\ar[r,"j''"swap]&A_2''
\end{tikzcd},
\]
 where 
 \[
 \tilde{r_i}=r_i''r_i'
 \]
 for $i=1,2,3$, and
 \[
 \tilde{s_1}=r_1''s_1'+(-1)^{r}s_1''r_0', 
 \]
 \[
 \tilde{s_2}=r_2''s_2'+(-1)^{r}s_2''r_1',
 \]
 and 
 \[
 \tilde{t}=r_2''t'+t''r_0'+(-1)^{r+1}s_2''s_1'.
 \]
 Note that in the expression of $\tilde{t}$, we only have length $2$ paths with suitable signs. 
 This is due to the vanishing of $m_j^{\A}$ for $j\geq 3$.

\begin{definition}\label{def:3termhomotopy}
Let $\A$ be a dg category. 
The {\em homotopy category $\mathcal H_{3t}(\A)$ of 3-term homotopy complexes (= 3-term $h$-complexes, or 3-term complexes up to homotopy) over $\A$} is defined to be $H^0\Fun_{\Ai_{\infty}}(\B,\A)$ where $\B$ is the dg category (\ref{3term}).
More precisely, its objects are given by diagrams $X$ in $\A$ of the form
\[
\begin{tikzcd}
&A_0\ar[r,"f"]\ar[rr,bend right = 8ex,"h"swap]&A_1\ar[r,"j"]&A_2
\end{tikzcd}
\]
where $|f|=|j|=0$, $|h|=-1$ and $d(f)=0$, $d(j)=0$ and $d(h)=-jf$.

Suppose $X'$ is given by the diagram with objects and morphisms with superscript a prime symbol. 

A morphism from $X$ to $X'$ is given by a homotopy equivalence class of 6-tuples $(r_i,s_j,t)$, $i=0,1,2$, $j=1,2$ where 
\[
r_i:A_i\rightarrow A_i', |r_i|=0, d(r_i)=0,
\]
\[
s_j:A_{j-1}\rightarrow A_{j}', |s_j|=-1, d(s_1)=f'r_0-r_1f,
\]
\begin{equation}
 d(s_2)=j'r_1-r_2j, \label{s2} 
\end{equation} 
and 
\begin{equation}
t:A_0\rightarrow A_2', |t|=-2, d(t)=r_2\circ h-h'\circ r_0-s_2\circ f-j'\circ s_1.\label{tt} 
\end{equation}
Two 6-tuples $(r_i,s_j,t)$ and $(r_i',s_j',t')$ are {\em homotopy equivalent} if there exists a 6-tuple $(r''_i,s_j'',t'')$ such that 
\[
|r_i''|=-1, r_i-r_i'=d(r_i''),
\]
\[
|s_j''|=-2, s_1-s_1'=d(s_1'')+f'\circ r_0''-r_1''\circ f,
\]
\[
s_2-s_2'=d(s_2'')+j'\circ r_1''-r_2''\circ j,
\]
and 
\[
|t''|=-3, t-t'=d(t'')-j'\circ s_1''+h'\circ r_0''-s_2''\circ f+r_2''\circ h.
\]
The composition of two morphisms given by homotopy equivalence classes of 6-tuples $(r_i',s_j',t'):F\rightarrow F'$ and $(r_i'',s_j'',t''):F'\rightarrow F'$, is the morphism given by the 6-tuple $(\tilde{r}_i,\tilde{s}_j,\tilde{t}):F\rightarrow F''$ where 
\begin{align*}
\tilde{r}_i  &=r_i''r_i' ,\\
\tilde{s}_1 &=r_1''s_1'+s_1''r_0', \\
\tilde{s}_2 &=r_2''s_2'+s_2''r_1' \\ 
\tilde{t}      &=r_2''t'+t''r_0'-s_2''s_1'.
\end{align*}
The identity morphism for $X$ is given by the homotopy equivalence class of the 6-tuple $(r_i,s_j,t)$ where $r_i=\Id$, $s_j=0$ and $t=0$ for $i=0,1,2$ and $j=1,2$. 
\end{definition}
From the definition, we see that $\mathcal H_{3t}(\A)$ is isomorphic to $\mathcal H_{3t}(\tau_{\leq 0}(\A))$. 
Thus we may assume $\A$ is strictly connective when needed.

Clearly if a 6-tuple $(r_i,s_j,t)$ gives rise to an isomorphism in $\mathcal H_{3t}(\A)$, then $r_i$ is a homotopy equivalence for each $i=0,1,2$. 
The converse is also true by the following proposition.
\begin{proposition}[\cite{Lefevre03}, {Proposition 8.2.2.3}] \label{termwiseequivalence}
Let $\theta:F\rightarrow F'$ be a morphism in $\mathcal H_{3t}(\A)$ given by the homotopy equivalence class of a 6-tuple $(r_i,s_j,t)$ as above. 
If $r_i$ is a homotopy equivalence in $\A$ for each $i=0,1,2$, then $\theta$ is an isomorphism. 
\end{proposition}

The following theorem is claimed by Kontsevich, modulo some results by To\"en.

\begin{theorem}[\cite{CanonacoOrnaghiStellari18} cf.~also \cite{Faonte17a}]
\label{thm:Ainfinityrep}
We have a canonical isomorphism 
\[
\Fun_{\Ai_{\infty}}(-,-)\rightarrow \rep_{dg}(-,-)
\]
of bifunctors $\Hqe_{k-\mathrm{cf}}^{op}\times \Hqe\rightarrow \Hqe$, where $\Hqe_{k-\mathrm{cf}}$ is the
full subcategory of $\Hqe$ whose objects are the dg categories whose morphism
complexes are cofibrant dg $k$-modules.
\end{theorem}

Denote by $\tilde{\mathrm{Com}}$ the dg path $k$-category of the quiver 
\begin{equation}\label{3term'}
\begin{tikzcd}
0\ar[rr,bend left=8ex,"c"]\ar[r,"a"swap]&1\ar[r,"b"swap]&2 
\end{tikzcd}
\end{equation}
where ${|}a{|}=0$, ${|}b{|}=0$, ${|}c{|}=-1$, and $d(c)=ba$.
We identify objects in $\D(\A\otimes \tilde{\Com}^{op})$ with sequences
\[
\begin{tikzcd}
X\ar[r,"f"]\ar[rr,"h"swap, bend right=8ex]&Y\ar[r,"j"]&Z
\end{tikzcd}
\]
in $\C_{dg}(\A)$ where $f$ and $j$ are closed morphisms of degree 0 and $h$ is a morphism of degree $-1$ such that $d(h)=jf$. 

Denote by $\tilde{\mathrm{Sq}}$ the dg quotient of $k\Sq$ by its full dg subcategory consisting of the object 10, i.e.~the dg path $k$-category of the quiver
\begin{equation}
\begin{tikzcd}
00\ar[rd,phantom,"="]\ar[r,"f"]\ar[d,"g"swap]&01\ar[d,"j"]\\
10\ar[loop left,"h"]\ar[r,"l"swap]&11
\end{tikzcd}
\end{equation}
where ${|}f{|}=0={|}j{|}={|}g{|}={|}l{|}$, ${|}h{|}=-1$, $d(h)=\Id_{10}$, and $jf=lg$.
We identify objects in $\D(\A\otimes\tilde{\Sq}^{op})$ with diagrams
\[
\begin{tikzcd}
X\ar[rd,phantom,"="]\ar[r,"f"]\ar[d,"g"swap]&Y\ar[d,"j"]\\
N\ar[loop left,"h"]\ar[r,"l"swap]&Z
\end{tikzcd}
\]
in $\C_{dg}(\A)$ where $f,g,j,l$ are closed morphisms of degree 0 such that $jf=lg$ and $h$ is a morphism of degree $-1$ such that $d(h)=\Id_{N}$.

Let $F:\tilde{\mathrm{Com}}\rightarrow \tilde{\mathrm{Sq}}$ be the dg functor such that $F(0)=00$, $F(1)=01$, $F(2)=11$ and $F(a)=f$, $F(b)=j$ and $F(c)=lhg$.
Let $G:k\mathrm{Sq}\rightarrow \tilde{\mathrm{Sq}}$ be the obvious dg quotient functor.
We have the following cospan of dg functors
\[
\begin{tikzcd}
&\tilde{\mathrm{Com}}\ar[d,"F"]\\
k\mathrm{Sq}\ar[r,"G"swap]&\tilde{\mathrm{Sq}}
\end{tikzcd}\;.
\]
By \cite[Proposition 1.6.3, Proposition 4.6]{Drinfeld04}, the dg functor $\Id\otimes G^{op}:\A\otimes k\Sq^{op}\rightarrow \A\otimes\tilde{\Sq}^{op}$ is a dg quotient of $\A\otimes k\Sq^{op}$ modulo $\A\otimes \{01\}$ and hence a localization functor, i.e.~the restriction functor along $\Id\otimes G^{op}$ is fully faithful.
Let $\mathcal S$ be the right orthogonal of the objects $(A, 01)^{\wedge}$ in $\D(\A\otimes k\Sq^{op})$, where $A$ runs through objects in $\A$. 
More precisely, $\mathcal S$ is the full triangulated subcategory of $\D(\A\otimes k\mathrm{Sq}^{op})$ consisting of squares 
\[
\begin{tikzcd}
X\ar[r]\ar[d]&Y\ar[d]\\
N\ar[r]&Z
\end{tikzcd}
\] 
where $N$ is an acyclic dg $\A$-module.
\begin{lemma}\label{equivalences}
$F$ is a quasi-equivalence and $G$ induces an equivalence $\mathcal S\xrightarrow{\sim}\D(\A\otimes\tilde{\Sq}^{op})$. 
\end{lemma}
\begin{proof}
It is enough to show that $F$ induces a quasi-isomorphism of complexes
\[
\Hom_{\tilde{\mathrm{Com}}}(0,2)\rightarrow \Hom_{\tilde{\mathrm{Sq}}}(00,11).
\]
On the left hand side, the complex is given by
\[
\cdots\rightarrow 0\rightarrow k\cdot c\xrightarrow{\sim} k\cdot(ba)\rightarrow 0\rightarrow \cdots,
\]
which is contractible. On the right hand side, the complex is given by
\[
\cdots\rightarrow k\cdot(lh^3g)\xrightarrow{\sim} k\cdot(lh^2g)\xrightarrow{0} k\cdot(lhg)\xrightarrow{\sim} k\cdot(lg)\rightarrow 0\rightarrow\cdots,
\]
which is also contractible.
Thus $F$ is a quasi-equivalence. 
The rest follows from~\cite[Proposition 4.6 (ii)]{Drinfeld04}.
\end{proof}
As a consequence of Lemma \ref{equivalences}, we have a fully faithful functor
\[
\begin{tikzcd}
\D(\A\otimes \tilde{\mathrm{Com}}^{op})\ar[r] &\D(\A\otimes k\mathrm{Sq}^{op}),
\end{tikzcd}
\]
which sends an object
\begin{equation}
\label{sequence}\tag{$\bigstar\bigstar$}
\begin{tikzcd}
&X\ar[r,"f"]\ar[rr,bend right = 8ex,"h"swap]&Y\ar[r,"j"]&Z
\end{tikzcd}
\end{equation}
to
\begin{equation}\label{square}\tag{$\bigstar\bigstar\bigstar$}
\begin{tikzcd}
&X\ar[r,"f"]\ar[d,"i"swap]&Y\ar[d,"j"]\\
&IX\ar[r,"{[}h{,}d(h){]}"swap]&Z
\end{tikzcd}\;.
\end{equation}
Recall that the homotopy category $\mathcal H_{3t}(\A)$ of 3-term h-complexes over $\A$ is defined to be $H^0(\Fun_{\Ai_{\infty}}(\B,\A))$, where $\B$ is the dg category defined by (\ref{3term}). 
We have an obvious quasi-equivalence $\tilde{\Com}\iso \B$.
Hence, by Theorem~\ref{thm:Ainfinityrep}, we have a fully faithful functor
\[
F: \mathcal H_{3t}(\A)\iso \rep(\B,\A)\iso \rep(\tilde{\Com},\A)\rightarrow \rep(k\Sq,\A),
\]
sending a 3-term h-complex over $\A$
\begin{equation}\label{3t}
\begin{tikzcd}
A_0\ar[r,"f"]\ar[rr,bend right = 8ex,"h"swap]&A_1\ar[r,"j"]&A_2,
\end{tikzcd}
\end{equation}
where $|f|=|j|=0$, $|h|=-1$, and $d(f)=0$, $d(j)=0$, and $d(h)=-jf$, to the following object in $\rep(k\Sq,\A)$
\begin{equation}\label{3tsquare}
\begin{tikzcd}
A_0^{\wedge}\ar[r,"f^{\wedge}"]\ar[d,"i"swap]                                              &A_1^{\wedge}\ar[d,"j^{\wedge}"]\\
IA_0^{\wedge}\;\;\ar[r,"{[}-h^{\wedge}{,}-d(h^{\wedge}){]}"swap]                  &\;\;A_2^{\wedge}\mathrlap{.}
\end{tikzcd}
\end{equation}

\begin{definition}\label{homotopyses}
A 3-term h-complex over $\A$, as given by (\ref{3t}), is a {\em homotopy (left, right) short exact sequence} if the corresponding object in $\rep(k\Sq,\A)$, as given by (\ref{3tsquare}), is homotopy bicartesian (cartesian, cocartesian).
\end{definition}
\begin{remark}
The properties of being homotopy left exact, homotopy right exact, or homotopy short exact are invariant under isomorphisms in $\mathcal H_{3t}(\A)$.
\end{remark}

Let $j:V\rightarrow W$ be a morphism of complexes of $k$-modules. Its {\em homotopy kernel} is the universal 3-term h-complex in $\C_{dg}(k)$ of the form
\[
\begin{tikzcd}
\hker(j)\ar[r,"f"]\ar[rr,bend right = 8ex,"h"swap]&V\ar[r,"j"]&W.
\end{tikzcd}
\]
 Explicitly, we have $\hker(j)=\Sigma^{-1}\Cone(j)$, $f=[1,0]$, and $h=[1,0]$.
 
 Consider a 3-term h-complex in $\A$ as given by (\ref{3t}).
 For each object $A$ in $\A$, we apply the dg functor $\A(A,-)$ to the diagram (\ref{3t}) and obtain a 3-term h-complex in $\C_{dg}(k)$
\[
 \begin{tikzcd}
 \A(A,A_0)\ar[r,"\A{(}A{,}f{)}"]\ar[rr,bend right = 8ex,"\A{(}A{,}h{)}"swap]&\A(A,A_1)\ar[r,"\A{(}A{,}j{)}"]&\A(A,A_2).
 \end{tikzcd}
 \]
 By the universal property of homotopy kernel, we obtain a canonical morphism 
 \[
 \alpha:\A(A,A_0)\rightarrow \hker(\A(A,j)).
 \]
 Dually, we obtain a canonical morphism 
 \[
 \beta:\A(A_2,A)\rightarrow \hker(\A(f,A)).
 \]

\begin{lemma} \label{lem:3termhcomplex}
A 3-term h-complex, as given by (\ref{3t}), is a homotopy left (resp.~right) exact sequence if and only if for each object $A$ in $\A$, the map $\tau_{\leq 0}(\alpha)$ (resp.~$\tau_{\leq 0}(\beta)$)is a quasi-isomorphism of complexes of $k$-modules.
It is a homotopy short exact sequence if and only if it is both homotopy left exact and homotopy right exact.
\end{lemma}
We can restate the lemma in a diagram chasing style as follows: the 3-term h-complex (\ref{3t}) is homotopy left exact if and only if for each $n\leq 0$ and each pair of morphisms $(v,w)\in Z^{n}\A(A,A_1)\times \A^{n-1}(A,A_2)$ such that $d(w)=-j\circ v$, there exists a morphism $u\in Z^n \A(A,A_0)$, unique up to a coboundary, such that there exists a pair of morphisms $(v',w')\in \A^{n-1}(A,A_1)\times \A^{n-2}(A,A_2)$ satisfying $v-f\circ u=d(v')$, $w-h\circ u=-d(w')-j\circ v'$. 

Next, we reformulate the notions of homotopy (co)cartesian square into the notions of homotopy (pushout) pullback square in $\A$. 
Recall that for a dg category $\A$, we denote by $\Mor(\A)$ the dg morphism category, cf.~Example~\ref{Mor(A)}. Objects in $\Mor(\Mor(\A))$ correspond to squares in $\A$
\begin{equation}
\begin{tikzcd} \label{square:homotopypullback}
{B'}\ar[r,"{p'}"]\ar[d,"{b}"swap]\ar[rd,"s"blue,blue]&{C'}\ar[d,"{c}"]\\
{B}\ar[r,"{p}"swap]&{C}
\end{tikzcd}
\end{equation}
where $B$, $C$, $B'$, $C'$ are objects in $\A$, and $p$, $c$, $p'$ and $b$ are closed morphisms of degree $0$, and $s$ is a morphism of degree $-1$ such that $d(s)=cp'-pb$. 
\begin{definition}\label{def:homotopysquares}
The category $\mathcal H_{\Sq}(\A)$ of {\em homotopy squares} in $\A$ is defined as 
\[
H^0(\Mor(\Mor(\A))).
\]
\end{definition}

\begin{lemma}\label{squareepivalence}
We have a quasi-equivalence
\[
\rep_{dg}(k\Sq,\A)\iso\Mor(\Mor(\A)).
\]
In particular, we have an epivalence
\[
\rep(k\Sq,\A)\rightarrow\Fun(\Mor,H^0(\Mor(\A))).
\]
\end{lemma}
\begin{proof}
Recall that we have an isomorphism of dg categories $k\Sq\iso k\Mor\otimes k\Mor$ and a quasi-equivalence $\Mor(\A)\iso \rep_{dg}(k\Mor,\A)$.
Then for any small dg category $\B$, we have canonical bijections
\begin{align*}
\Hqe(\B,\rep_{dg}(k\Sq,\A))&\iso\Hqe(k\Sq\otimes \B,\A)\\
&\iso\Hqe(k\Mor\otimes k\Mor\otimes \B,\A)\\
&\iso\Hqe(\B,\rep_{dg}(k\Mor,\Mor(\A)).
\end{align*}
Therefore we have a quasi-equivalence  
\[
\rep_{dg}(k\Sq,\A)\iso \rep_{dg}(k\Mor,\Mor(\A)).
\]
Thus, by Lemma~\ref{epi}, we have an epivalence
\[
\rep(k\Sq,\A)\iso \rep(k\Mor,\Mor(\A))\xrightarrow{\Dia} \Fun(\Mor,H^0(\Mor(\A))).
\]
\end{proof}

By abuse of terminology, we have the following definition.
\begin{definition} 
A homotopy square in $\A$, as given by (\ref{square:homotopypullback}), is a {\em homotopy pullback square} if the corresponding object in $\rep(k\Sq,\A)$, under the quasi-equivalence established in Lemma~\ref{squareepivalence}, is a homotopy cartesian square, cf.~Definition~\ref{maindef}. Dually, we define the notion of {\em homotopy pushout square}.
\end{definition}
We also refer to a homotopy pullback (resp.~pushout) square as a {\em homotopy cartesian} (resp.~{\em cocartesian}) {\em square}. 
A homotopy square that is both homotopy cartesian and homotopy cocartesian will be called {\em homotopy bicartesian}.
If $Z^0(\A)$ is an additive category, then the object (\ref{square:homotopypullback}) is a homotopy pullback square if and only if the following 3-term homotopy complex in $\A$
\[
\begin{tikzcd}
B'\ar[r,"{[}b{,}p'{]}^{\intercal}"]\ar[rr,bend right=8ex,"s"swap]&B\oplus C'\ar[r,"{[}p{,}-c{]}"]&C
\end{tikzcd}
\]
is homotopy left exact. 

Let us describle {\em the dg category $\Cosp(\A)$ of homotopy cospans} in $\A$. It is defined as the dg category $\Fun_{\Ai_{\infty}}(k\Cosp,\A)$ of strictly unital $\Ai_{\infty}$-functors from $k\Cosp$ to $\A$, cf.~Subsection~\ref{subsec:homotopydiagrams}.
We identify a strictly unital $\Ai_{\infty}$-functor $F:k\Cosp\rightarrow\A$ with the corresponding diagram in $Z^0(\A)$
\begin{equation}\label{dia:cosp}
\begin{tikzcd}
A_2\ar[r,"g"]&A_3&A_1\ar[l,"k"swap]
\end{tikzcd}.
\end{equation}
Let $F$ and $F'$ be two strictly unital $\Ai_{\infty}$-functors from $k\Cosp$ to $\A$. 
We identify them with their corresponding diagrams. 
A strictly unital morphism $H:F\rightarrow F'$ of degree $n$ is represented by the following diagram
\begin{equation}\label{mor:spandegreen}
\begin{tikzcd}
A_2\ar[d,"h_2"swap]\ar[r,"g"]\ar[rd,"s_1"red,swap ,red]&A_3\ar[d,"h_3"swap]&A_1\ar[l,"k"swap]\ar[d,"h_1"]\ar[ld,"s_2"red,red]\\
A_2'\ar[r,"g'"swap]&A_3'&A_1'\ar[l,"k'"]
\end{tikzcd},
\end{equation}
where $|h_i|=n$ for $i=1$,$2$, $3$, and $|s_1|=|s_2|=n-1$, which can be visualized as a matrix 
\[
\begin{bmatrix}h_1&&\\
&h_2&\\
s_2&s_1&h_3
\end{bmatrix}.
\]
Put $H'=d(H)\in \Hom^{n+1}(F,F')$. Then $H'$ corresponds to the matrix:
\[
\begin{bmatrix}-d(h_1)&&\\
&-d(h_2)&\\
d(s_2)+j'h_1-(-1)^nh_3j&d(s_1)+f'h_2-(-1)^nh_3f&d(h_3)
\end{bmatrix}.
\]
The composition map is defined by matrix multiplication.
Note that we have a canonical restriction dg functor
\[
\Str: \Mor(\Mor(\A))\rightarrow \Cosp(\A),\;\;
\begin{tikzcd}
{B'}\ar[r,"{p'}"]\ar[d,"{b}"swap]\ar[rd,"s"blue,blue]&{C'}\ar[d,"{c}"{name=1},] &\;\ar[d,""{swap,name=2,white},white]&{C'}\ar[d,"c"]\\
{B}\ar[r,"{p}"swap]&{C} &{B}\ar[r,"p"swap]&C\ar[r,from=1,to=2,mapsto,red]
\end{tikzcd}.
\]
Dually, we have the dg category $\Sp(\A)$ of {\em homotopy spans} in $\A$ and a canonical restriction dg functor $\Str': \Mor(\Mor(\A))\rightarrow \Sp(\A)$.
\begin{definition}\label{def:homotopyspans}
The category $\mathcal H_{\Cosp}(\A)$ of {\em homotopy cospans} in $\A$ is defined as 
\[
H^0(\Cosp(\A)). 
\]
Dually, the category $\mathcal H_{\Sp}(\A)$ of {\em homotopy spans} in $\A$ is defined as 
\[
H^0(\Sp(\A)).
\]
\end{definition}
By Theorem~\ref{thm:Ainfinityrep}, we have equivalences of categories 
\[
\mathcal H_{\Sp}(\A)\iso \rep(k\Sp,\A)
\]
 and 
 \[
 \mathcal H_{\Cosp}(\A)\iso \rep(k\Cosp,\A).
 \]
\begin{definition}\label{def:Ainfinitypullback}
An object $S$  in $\mathcal H_{\Cosp}(\A)$, as given by (\ref{dia:cosp}), is said to admit a {\em homotopy pullback} if there is an object $X$ in $\mathcal H_{\Sq}(\A)$
\[
\begin{tikzcd}
A_0\ar[r,"f"]\ar[d,"j"swap]\ar[rd,blue,"s"{blue}]&A_1\ar[d,"k"]\\
A_2\ar[r,"g"swap]&A_3
\end{tikzcd}
\]
which is a homotopy pullback square.
In this case, the object $X$ is called the homotopy pullback of $S$. 
\end{definition}

Now, we define the notions of homotopy kernel (cokernel) of an object in $\Mor(\A)$.
Recall the morphism category $\Mor$ defined in Subsection \ref{subsec:homotopydiagrams}
\[
\begin{tikzcd}
0\ar[r]&1,
\end{tikzcd}
\]
and the dg category $\B$ defined in (\ref{3term})
\[
\begin{tikzcd}
0\ar[r,"f"] &1\ar[r,"j"]&2.
\end{tikzcd}
\]
where $|f|=|j|=0$, and $d(f)=d(j)=0$, $jf=0$.
 Let $i:k\Mor\rightarrow \B$ be the inclusion dg functor which sends object $0$ to $0$ and object $1$ to $1$.
 Let $i':k\Mor\rightarrow \B$ be the inclusion dg functor which sends object $0$ to $1$ and object $1$ to $2$.
 Let $\Res:\Fun_{\Ai_{\infty}}(\B,\A)\rightarrow\Fun_{\Ai_{\infty}}(k\Mor,\A)\iso\Mor(\A)$ (resp.~$\Res'$) be the restriction functor along $i$ (resp.~$i'$).  
\begin{definition}\label{def:homotopykernel} 
By {\em homotopy kernel} of an object $j:B\rightarrow C$ in $\Mor(\A)$, we mean a homotopy left exact squence $X$ 
\[
\begin{tikzcd}
A\ar[rr,bend right=8ex,"h"swap]\ar[r,"f"]&B\ar[r,"j"]&C.
\end{tikzcd}
\]
Sometimes, we say $f:A\rightarrow B$ or $A$ is the homotopy kernel of $j$. 
 
Dually, we define the {\em homotopy cokernel} of an object in $\Mor(\A)$. 
\end{definition}

 In Definition~\ref{def:Ainfinitypullback}, $\Str(X)$ is stated to be equal to $S$, and in Definition~\ref{def:homotopykernel}, $\Res'(X)$ is defined to be equal to $j$. Meanwhile, in Definition~\ref{homotopypullbackdef}, $LX$ is asserted to be isomorphic to $S$, cf.~Remark~\ref{rmk:isomorphism}.
The above definitions are compatible with each other, due to the following observation:
\begin{lemma}\label{Mor(A)and3term}
Let $j:B\rightarrow C$ be an object in $\Mor(\A)$. 
Consider a morphism $\theta$ from $j':B'\rightarrow C'$ to $j$ in $Z^0(\Mor(\A))$ as follows
\[
\begin{tikzcd}
B'\ar[r,"j'"]\ar[d,"h_1"swap]\ar[rd,"s_2"{red,swap},red]&C'\ar[d,"h_2"]\\
B\ar[r,"j"swap]&C\mathrlap{.}
\end{tikzcd}
\]
Suppose we are given an object
\[
\begin{tikzcd}
A'\ar[r,"f'"]\ar[rr,bend right=8ex,"h'"swap]&B'\ar[r,"j'"]&C'
\end{tikzcd}
\]
in $\mathcal H_{3t}(\A)$.
Then we have the following morphism $\mu$ in $Z^0\Fun_{\Ai_{\infty}}(\B,\A)$ which restricts to the morphism $\theta$
\[
\begin{tikzcd}
A'\ar[r,"f'"]\ar[rd,"0"{red,swap},red]\ar[rrd,"0"blue,blue]\ar[d,equal]\ar[rr,"h'",bend left=8ex]&B'\ar[r,"j'"]\ar[d,"h_1"swap]\ar[rd,"s_2"red,red]&C'\ar[d,"h_2"]\\
A'\ar[r,"h_1f'"swap]\ar[rr,"h_2h'-s_2f'"swap,bend right=8ex]&B\ar[r,"j"swap]&C\mathrlap{.}
\end{tikzcd}
\]
\end{lemma}
\begin{remark}
In Lemma~\ref{Mor(A)and3term}, the homotopy class of the  morphism $\mu$ is an isomorphism in $H^0(\Mor(\A))$ if and only if the homotopy class of the morphism $\theta$ is an isomorphism in $\mathcal H_{3t}(\A)$, cf.~Proposition \ref{termwiseequivalence}.
\end{remark}
\subsection{Some diagram lemmas}
In the rest of this section, we prove some basic diagram lemmas that will used in the sequel.

Let us recall the universal property of homotopy left exact sequences. 
Suppose we have 3-term h-complexes $X_1$ and $X_2$ with $X_2$ being homotopy left exact.
Then each morphism $\theta:\Res'(X_1)\rightarrow \Res'(X_2)$ in $H^0(\Mor(\A))$ extends uniquely to a morphism $\mu:X_1\rightarrow X_2$ in $\mathcal H_{3t}(\A)$ such that $\theta=\Res'(\mu)$. 
Indeed, we have the following
\begin{lemma}\label{strictmorphismlift}
Suppose we have 3-term h-complexes $X_i$ in $\A$, $i=1$, $2$, of the form
\[
\begin{tikzcd}
A_i\ar[rr,bend right=8ex,"h_i"swap]\ar[r,"f_i"]&B_i\ar[r,"j_i"]&C_i
\end{tikzcd}
\]
where $X_2$ is homotopy left exact. Then any morphism in $Z^0(\Mor(\A))$
\[
\begin{tikzcd}
B_1\ar[r,"j_1"]\ar[d,"b"swap]\ar[rd,"s_2"red,red]&C_1\ar[d,"h_3"]\\
B_2\ar[r,"j_2"swap]&C_2
\end{tikzcd}
\]
extends to a morphism in $Z^0(\Fun_{\infty}(\B,\A))$ 
\[
\begin{tikzcd}
A_1\ar[rd,"s_1"{red,swap},red]\ar[rrd,"t"{blue},blue]\ar[d,"a"swap]\ar[rr,bend left=8ex,"h_1"]\ar[r,"f_1"]&B_1\ar[r,"j_1"]\ar[d,"b"swap]\ar[rd,"s_2"{red},red]&C_1\ar[d,"h_3"]\\
A_2\ar[rr,bend right=8ex,"h_2"swap]\ar[r,"f_2"swap]&B_2\ar[r,"j_2"swap]&C_2\mathrlap{.}
\end{tikzcd}
\]
\end{lemma}
\begin{proof}
Put $u=h_2f_1$ and $v=h_3h_1-s_2f_1$. Then the pair $(u,v)$ satisfies $d(v)=-h_3j_1f_1-(j_2h_2-h_3j_1)f_1=-j_2h_2f_1=-j_2u$.

Since $X_2$ is homotopy left exact, we have a closed morphism $a:A_1\rightarrow A_2$ of degree $0$ such that there exists a pair $(-s_1,-t)$ where $s_1$ is a morphism from $A_1$ to $B_2$ of degree $-1$, $t$ is a morphism from $A_1$ to $C_2$ of degree $-2$ satisfying $u-f_2a=d(-s_1)$ and $v-h_2a=-d(-t)-j(-s_1)$. So the 6-tuple $(a,b,h_3,s_1,s_2,t)$ is a morphism from $X_1$ to $X_2$ in $Z^0(\Fun_{\infty}(\B,\A))$.
\end{proof}
Let $f:X\rightarrow Y$ be an object in $\Mor(\A)$ and 
\[
\begin{tikzcd}
&X^{\wedge}\ar[d,"f^{\wedge}"]\\
0\ar[r]&Y^{\wedge}
\end{tikzcd}
\]
the associated object $S$ in the category $\rep(k\Cosp,\A)$.
Recall that we have a fully faithful functor $\mathcal H_{3t}(\A)\hookrightarrow \rep(k\Sq,\A)$.
Then a homotopy kernel of the object $f$, which is in the category $\mathcal H_{3t}(\A)$, can be identified with a homotopy pullback of the cospan $S$, which is in the category $\rep(k\Sq,\A)$.
Hence by Proposition \ref{pullbackunique}, the homotopy kernel of an object in $\Mor(\A)$ is unique up to a unique isomorphism if it exists in $\mathcal H_{3t}(\A)$.
Similarly, the homotopy pullback of an object in $\mathcal H_{\Cosp}(\A)$ is unique up to a unique isomorphism if it exists in $\mathcal H_{\Sq}(\A)$.

\subsection{Restriction of morphisms in $\mathcal H_{3t}(\A)$} \label{res}
Let $f:A\rightarrow B$ and $f':A'\rightarrow B'$ be two objects in $\Mor(\A)$. 
Then a morphism $\eta:f\rightarrow f'$ in $H^0(\Mor(\A))$ leads to a class of objects in $\mathcal H_{\Sq}(\A)$ given by the following diagrams in $\A$
\[
\begin{tikzcd}
A\ar[r,"f"]\ar[d,"g"swap]\ar[rd,"h"]&B\ar[d,"j"]\\
A'\ar[r,"f'"swap]&B'
\end{tikzcd}.
\] 
By Lemma~\ref{squareepivalence}, they are isomorphic to each other: two homotopic morphisms from $f$ to $f'$ give rise to isomorphic objects in $\mathcal H_{\Sq}(\A)$.

Recall the restriction functors $\Res,\Res':\Fun_{\Ai_{\infty}}(\B,\A)\rightarrow \Mor(\A)$ defined before Definition \ref{def:homotopykernel}.
Suppose we have a morphism $\alpha:X\rightarrow X'$ in $\mathcal H_{3t}(\A)$ of the form
\begin{equation}\label{morphism3term}
\begin{tikzcd}
&A_0\ar[r,"f"]\ar[d,"h_0"swap]\ar[rd,"s_1"red,swap ,red]\ar[rrd,"t"blue,blue]\ar[rr,bend left = 8ex,"h"]&A_1\ar[d,"h_1"swap]\ar[rd,"s_2"red,red]\ar[r,"j"]&A_2\ar[d,"h_2"]\\
&A_0'\ar[r,"f'"swap]\ar[rr,bend right = 8ex,"h'"swap]&A_1'\ar[r,"j'"swap]&A_2'
\end{tikzcd}.
\end{equation}
Following the preceding discussion, the two morphisms $\Res(\alpha)$, $\Res'(\alpha)$ yield two objects in $\mathcal H_{\Sq}(\A)$, denoted by the same notations for clarity.

\begin{definition}As above, let $\alpha:X\rightarrow X'$ be a morphism in $\mathcal H_{3t}(\A)$. 
We call $\Res(\alpha)\in \mathcal H_{\Sq}(\A)$ the {\em restriction} of $\alpha:X\rightarrow X'$ to $f:A\rightarrow B$. 
Similarly, we call $\Res'(\alpha)$ the {\em restriction} of $\alpha:X\rightarrow X'$ to $j:B\rightarrow C$.
\end{definition}

\begin{proposition}\label{push}
Suppose we have objects $X$ and $X'$ in $\mathcal H_{3t}(\A)$ of the following form
\[
\begin{tikzcd}
A_0\ar[rr,bend right=8ex,"h"swap]\ar[r,"f"]&A_1\ar[r,"j"]&A_2,
\end{tikzcd}
\]
where, for $X'$, we add a prime symbol superscript to each term. 
Let $\alpha:X\rightarrow X'$ be a morphism in $\mathcal H_{3t}(\A)$ as in (\ref{morphism3term}).
\begin{itemize}
\item[1)] Suppose $h_2:A_2\rightarrow A_2'$ is a homotopy equivalence.
 Let $X''\in \mathcal H_{\Sq}(\A)$ be the restriction of $\alpha$ to $f:A_0\rightarrow A_1$. 
 Then the following statements hold:
\begin{itemize}
\item[(a)] If $X'$ is homotopy left exact, then $X$ is a homotopy left exact if and only if $X''$ is a homotopy pullback square.
\item[(b)] If $X''$ is a homotopy pushout square, then $X$ is homotopy right exact if and only if so is $X'$.
\end{itemize}
\item[2)]Suppose $h_0:A_0\rightarrow A_0'$ is a homotopy equivalence. 
Let $X'''$ be the restriction of $\alpha$ to $j:A_1\rightarrow A_2$. Then the following statements hold:
\begin{itemize}
\item[(a)] If $X$ is homotopy right exact, then $X'$ is homotopy right exact if and only if $X'''$ is a homotopy pushout square.
\item[(b)] If $X'''$ is a homotopy pullback square, then $X$ is homotopy left exact if and only if so is $X'$.
\end{itemize}
\end{itemize}
\end{proposition}
\begin{proof}
We prove 1) and then 2) follows by duality. 
For simplicity, we omit the symbol $^{\wedge}$ of the representable dg modules $A^{\wedge}$. 
Since $h_2$ is a homotopy equivalence, we may assume $h_2$ is the identity of ${A_2}$.
The morphism $\alpha$ is given by the following diagram in $\C(\A)$:
\[
\begin{tikzcd}
A_0\ar[rddd,"\begin{bmatrix}s_1\\t\end{bmatrix}"{swap,red},red]\ar[rd,"u=\begin{bmatrix}f\\h\end{bmatrix}"]\ar[dd,"h_0"swap]&&&\\
&V=\Sigma^{-1}\Cone(j)\ar[r,"{[}-1{,}0{]}"]\ar[dd,"\begin{bmatrix}h_1\;\;0\\-s_2\;\; 1\end{bmatrix}"]&A_1\ar[r,"j"]     \ar[dd,"h_1"] \ar[rdd,"s_2"red,red] &A_2\ar[dd,equal]\\ 
A_0'\ar[rd,"u'=\begin{bmatrix}f'\\h'\end{bmatrix}"swap]&&&\\
&V'=\Sigma^{-1}\Cone(j')\ar[r,"{[}-1{,}0{]}"swap]&A_1'\ar[r,"j'"swap]&A_2,
\end{tikzcd}
\]
where diagonal maps are the homotopies making the diagram commutative in $\mathcal H(\A)$.
Then $X''$ is given by
\[\begin{tikzcd}
A_0\ar[r,"f"]\ar[d,"h_0"swap]\ar[rd,"s_1"]&A_1\ar[d,"h_1"]\\
 A_0'\ar[r,"f'"swap]&A_1'
\end{tikzcd}\;.
\] 

Put $U=\Cone(f)$, $U'=\Cone(f')$ and $V_3=\Sigma^{-1}\Cone(h_1,f')$. 
We have a canonical morphism $r=(h,j):U\rightarrow A_2$ and similarly a morphism $r':U'\rightarrow A_2'$.
From the object $X$, we have the following diagram in $\D(\A)$ 
\[
\begin{tikzcd}
A_0\ar[r,"f"]\ar[d,"u"swap]&A_1\ar[r,""]\ar[d,equal]&U\ar[r]\ar[d,"r"]&\Sigma A_0\ar[d]\\ 
V\ar[r]&A_1\ar[r,"j"swap]&A_2\ar[r]\ar[d]&\Sigma V\ar[d]\\
 & &C(r)\ar[r,equal]&C(r)
\end{tikzcd}.
\]

Put $Y=\Sigma^{-1}\Cone(h_1)$ and $W=\Cone(u')$. 
Similarly, we have the following diagrams in $\D(\A)$
\[
\begin{tikzcd}
V\ar[r,"{[}-1{,}0{]}"]\ar[d,"\begin{bmatrix}h_1\;\;0\\-s_2\;\;1\end{bmatrix}"swap]&A_1\ar[r,"j"]\ar[d,"h_1"]&A_2\ar[d,equal]\\ 
V'\ar[r,"{[}-1{,}0{]}"]\ar[d,"\begin{bmatrix}0\;\;0\\-1\;\;0\end{bmatrix}"swap]&A_1'\ar[d,"\begin{bmatrix}0\\1\end{bmatrix}"]\ar[r,"j'"]&A_2\\ 
\Sigma Y \ar[r,equal]\ar[d,"-\begin{bmatrix}1\;\;0\\s_2\;\;j'\end{bmatrix}"swap]& \Sigma Y\ar[d,"{[}1{,}0{]}"]&\\
\Sigma V\ar[r,"{[}-1{,}0{]}"swap]&\Sigma A_1&
\end{tikzcd},
\begin{tikzcd}
Y\ar[r,"{[}-1{,}0{]}"]\ar[d,"\begin{bmatrix}1\;\;0\\0\;\;0\\0\;\;1\end{bmatrix}"swap]&A_1\ar[r,"h_1"]\ar[d,"\begin{bmatrix}1\\0\end{bmatrix}"]&A_1'\ar[d,equal]\\
 V_3\ar[r,"\begin{bmatrix}-1\;\;0\;\;0\\0\;\;-1\;\;0\end{bmatrix}"]\ar[d,"{[}0{,}-1{,}0{]}"swap]&A_1\oplus A_0'\ar[d,"{[}0{,}1{]}"]\ar[r,"{[}h_1{,}f'{]}"]&A_1'\\ 
 A_0' \ar[d,"\begin{bmatrix}0\\-f'\end{bmatrix}"swap]\ar[r,equal]& A_0'\ar[d,"0"]&\\
 \Sigma Y\ar[r,"{[}-1{,}0{]}"swap]&\Sigma A_1&
 \end{tikzcd},
 \]
 \[
\begin{tikzcd}
Y\ar[r,"\begin{bmatrix}1\;\;0\\0\;\;0\\0\;\;1\end{bmatrix}"]\ar[d,equal]&V_3\ar[rr,"{[}0{,}-1{,}0{]}"]\ar[d,dashed,blue,"v=\begin{bmatrix}1\;\;\;\;0\;\;\;\;0\\s_2\;\;-h'\;\;j'\end{bmatrix}"]&&A_0'\ar[d,"u'=\begin{bmatrix}f'\\h'\end{bmatrix}"]\\
 Y\ar[r,"\begin{bmatrix}1\;\;0\\s_2\;\;j'\end{bmatrix}"swap]&V\ar[rr,"\begin{bmatrix}h_1\;\;0\\-s_2\;\;1\end{bmatrix}"swap]\ar[d,dashed,blue]&&V'\ar[d,"\begin{bmatrix}0\;\;0\\1\;\;0\\0\;\;1\end{bmatrix}"]\\ 
 & W\ar[rr,equal]&&W
 \end{tikzcd}.
 \]
We have a morphism from the sequence in blue to the mapping triangle of $v$ as follows:
\[
\begin{tikzcd}
V_3\ar[r,"v"]\ar[d,equal]&V\ar[r,"\begin{bmatrix}0\;\;0\\h_1\;\;0\\-s_2\;\;1\end{bmatrix}"]\ar[d,equal]&W\ar[r,"\begin{bmatrix}0\;\;0\;\;0\\1\;\;0\;\;0\\0\;\;1\;\;0\end{bmatrix}"]\ar[d,"\theta"]&\Sigma V_3\ar[d,equal]\\
V_3\ar[r,"v"swap]&V\ar[r,"{[}0{,}1{]^{\intercal}}"swap]&\Cone(v)\ar[r,"{[}1{,}0{]}"swap]&\Sigma V_3\mathrlap{.}
\end{tikzcd}
\]
By direct inspection, we see that $\Cone(v)$ is isomorphic to the mapping cone of 
\[
\Cone(\Id_{\Sigma^{-1}A_1})\rightarrow W.
\]
The map $\theta$ is the canonical inclusion from $W$ to this cone.
It is straightforward to check that the rightmost square commutes and the middle square commutes up to homotopy.
Now we see that the composition of the canonical maps $A_0\xrightarrow{[f,h_0,s_1]^{\intercal}} V_3$ and $v:V_3\rightarrow V$ is homotopic to the canonical map $u: A_0\rightarrow V$. 

(a) Since $X'$ is homotopy left exact, the map $u':A_0'\rightarrow V'$ induces a quasi-isomorphism 
\[
\tau_{\leq 0}\RHom(A,\Sigma^{-i}A_0')\rightarrow\tau_{\leq 0 }\RHom(A,\Sigma^{-i}V')
\]
for each $A'\in \A$.
Thus the induced map
\[ 
\tau_{\leq 0}\Hom(A',V_3)\rightarrow \Hom(A', V)
\]
is a quasi-isomorphism for each $A'\in \A$. 
Therefore $X$ is homotopy left exact if and only if $X''$ is homotopy cartesian.

(b) Put $U_3=\Cone((f,h_0)^{\intercal})$.
Consider the following diagrams
\[
\begin{tikzcd}
&A_1\ar[r,equal]\ar[d]&A_1\ar[d]\\
A_0\ar[r]\ar[d,equal]&A_0'\oplus A_1\ar[r]\ar[d]&U_3\ar[d]\\ 
A_0\ar[r]&A_0'\ar[r]&U 
\end{tikzcd},
\begin{tikzcd}
A_1\ar[r]\ar[d,equal]&U_3\ar[r]\ar[d]&U\ar[d]\\ 
A_1\ar[r]&A_1'\ar[d]\ar[r]&U'\ar[d]\\
 & Z\ar[r,equal]& Z
 \end{tikzcd},
\begin{tikzcd}
U\ar[r]\ar[d]&A_2\ar[d,equal]\\
U'\ar[r]&A_2\end{tikzcd}.
\]

Since $X''$ is homotopy cocartesian, the canonical map $U_3\rightarrow A_1'$ induces a quasi-isomorphism 
\[
\tau_{\leq 0}\RHom(A_1',A')\rightarrow \tau_{\leq 0}\RHom(U_3, A')
\]
for each $A'\in\A$. This implies that the canonical map $U\rightarrow U'$ induces a quasi-isomorphism 
\[
\tau_{\leq 0}\RHom(U', A')\rightarrow \tau_{\leq 0}\RHom(U, A')
\]
for each $A'\in\A$. From the rightmost square, we infer that $X$ is homotopy right exact if and only if so is $X'$.
\end{proof}

\begin{corollary}\label{Comp}
Suppose we have the following diagram in $\A$ with lower square being homotopy pullback and the upper square an object in $\Mor(\Mor(\A))$. 
Then the outer square (i.e.~the square on the right hand side) is homotopy pullback if and only if so is the upper square. 
\[
\begin{tikzcd}
A_1\arrow[r,"f"]\arrow[d,"g"swap,""{name=1} ]\ar[rd,"h"red,red]&A_2\arrow[d,"j",""{name=2}]\\
A_3\ar[d,"g'"swap,""{name=3}]\ar[r,"f'"]\ar[rd,"h'"red,red]&A_4\ar[d,"j'",""{name=4}]\\
A_5\ar[r,"k'"swap]&A_6
\arrow[phantom,from=1,to=2,""]\arrow[phantom,from=3,to=4,""]
\end{tikzcd}\;\;,\;\;
\begin{tikzcd}
A_1\ar[r,"f"]\ar[d,"g'g"swap]\ar[rd,"h'g+j'h"{description},blue]&A_2\ar[d,"j'j"]\\
A_5\ar[r,"k'"swap]&A_6
\end{tikzcd}
\]
\end{corollary}
\begin{proof}
We may assume that the dg category $\A$ admits a zero object and direct sums.
Note that direct sums of homotopy cartesian squares remain homotopy cartesian.
Also, note that squares of the form
\[
\begin{tikzcd}
0\ar[r,"0"]\ar[d,"0"swap]\ar[rd,"0"red,red]&A\ar[d,equal]\\
0\ar[r,"0"swap]&A
\end{tikzcd}
\]
are trivially homotopy cartesian.

The statement follows by applying Proposition \ref{push} to one of the following two diagrams
\[
\begin{tikzcd}A_1\ar[rd,red,"s"]\ar[rr,bend left=8ex,blue,"h'g+j'h"blue]\ar[r,"u"]\ar[d,"g"swap]&A_2\oplus A_5\ar[r,"v"]\ar[d,"r"]&A_6\ar[d,equal] &&A_1\ar[r]\ar[d,equal]&A_2\oplus A_3\ar[d]\ar[r] &A_4\ar[d]\\
A_3\ar[rr,bend right=8ex,"h'"{blue,swap},blue]\ar[r,"w"swap]&A_4\oplus A_5 \ar[r,"z"swap] &A_6&&A_1\ar[r]&A_2\oplus A_5\ar[r]&A_6
\end{tikzcd}
\]
where $r=\begin{bmatrix}j\;\;0\\0\;\;\Id\end{bmatrix}$, $s=\begin{bmatrix}h\\0\end{bmatrix}$, $u=\begin{bmatrix}f\\g'g\end{bmatrix}$, $v={[}j'j{,}-k'{]}$, $w=\begin{bmatrix}f'\\g'\end{bmatrix}$ and $z={[}j'{,}-k'{]}$.
\end{proof}
Suppose we have a homotopy pullback square $X_2$
\[
\begin{tikzcd}
B\ar[r,"b"]\ar[d,"e"swap]\ar[rd,red,"s_2"]&C\ar[d,"f"]\\
E\ar[r,"h"swap]&F
\end{tikzcd}
\]
and an object $X_3$ in $\mathcal H_{\Sq}(\A)$
\[
\begin{tikzcd}
A\ar[r,"c"]\ar[d,"d"swap]\ar[rd,red,"s_3"]&C\ar[d,"f"]\\
D\ar[r,"hg"swap]&F
\end{tikzcd},
\]
where $g$ is a morphism from $D$ to $E$.
Then the triple $(c,gd,s_3)$ satisfies $d(s_3)=hgd-fc=-[-h,\;f]\begin{bmatrix}gd\\c\end{bmatrix}$, and thus, by Lemma~\ref{lem:3termhcomplex}, we have a morphism $a:A\rightarrow B$ such that there exists a triple $(s,s_1,t)$ such that 
\[
\begin{bmatrix}gd\\c\end{bmatrix}-\begin{bmatrix}e\\b\end{bmatrix}a=\begin{bmatrix}d(s_1)\\d(s)\end{bmatrix}\]
 and 
 \[
 s_3-s_2a=-d(t)-[-h,\;f]\begin{bmatrix}s_1\\s\end{bmatrix}.
 \]
 So we have the following diagram
\[
\begin{tikzcd}
A\ar[d,"d"swap]\ar[rr,bend left=12ex,"s",dashed]\ar[r,dashed,"a"]\ar[rr,bend left=6ex,"c"]\ar[rd,"s_1"{red,swap},red,dashed]\ar[rrd,"t"{blue},blue,dashed,bend left=2ex]&B\ar[r,"b"]\ar[d,"e"swap]\ar[rd,"s_2"red,red,]&C\ar[d,"f"]\\
D\ar[r,"g"swap]&E\ar[r,"h"swap]&F
\end{tikzcd}
\]
where the left hand square $X_1$ is an object in $\mathcal H_{\Sq}(\A)$ and $(-s,0,-t)$ serves as a homotopy between $X_3$ and the composition of $X_1$ and $X_2$, all regarded as morphisms in $Z^0(\Mor(\A))$.
So we have
\begin{corollary}\label{pastinglaw:second}
Keep the notations as above. 
If we have homotopy pullback squares $X_2$ and $X_3$, then there exists a pair $(a,s_1)$ such that $X_1$ is also homotopy pullback and that the composition of $X_1$ and $X_2$ is homotopic to $X_3$, all regarded as morphisms in $Z^0(\Mor(\A))$.
\end{corollary}
\subsection{Homotopy pushouts/pullbacks of homotopy short exact sequences}
Let $X$ be a homotopy right exact sequence over $\A$
\[
\begin{tikzcd}
A\ar[rr,bend right=8ex,"h"swap]\ar[r,"f"]&B\ar[r,"j"]&C.
\end{tikzcd}
\]
Let $\overline{\alpha}:A\rightarrow A'$ be a morphism in $H^0(\A)$.
Let $\alpha$ be a representative for $\overline{\alpha}$.
Let $S\in \mathcal H_{\Sp}(\A)$ be the following homotopy span in $\A$:
\[
\begin{tikzcd}
A\ar[r,"f"]\ar[d,"\alpha"swap]&B\\
A'& 
\end{tikzcd}.
\]
It is evident that it is determined by $\overline{\alpha}$ up to an isomorphism and this isomorphism restricts to the identity on each component in $H^0(\A)$.
Assume that the homotopy span $S$ admits a homotopy pushout.

\begin{proposition}\label{cons}
We have a homotopy right exact sequence $X'$ and a morphism $\mu:X\rightarrow X'$
\begin{equation}\label{const}
\begin{tikzcd}
&A\ar[r,"f"]\ar[d,"h_0=\alpha"swap]\ar[rd,"s_1"red,swap ,red]\ar[rrd,"t"blue,blue]\ar[rr,bend left = 8ex,"h"]&B\ar[d,"h_1"swap]\ar[rd,"s_2"red,red]\ar[r,"j"]&C\ar[d,equal]\\
&A'\ar[r,"f'"swap]\ar[rr,bend right = 8ex,"h'"swap]&B'\ar[r,"j'"swap]&C
\end{tikzcd}
\end{equation}
such that the restriction of $\mu$ to $f:A\rightarrow B$ is a homotopy pushout of $S$ (cf.~Subsection \ref{res}). 
Furthermore, if both $X$ and $X'$ are homotopy short exact, then the homotopy pushout of $S$ is homotopy bicartesian.
\end{proposition}
\begin{proof}
Let $X''$ be the pushout of $S$
\[
\begin{tikzcd}
&A\ar[r,"f"]\ar[d,"h_0=\alpha"swap]\ar[rd,"s_1"red,swap ,red]&B\ar[d,"h_1"]\\
&A'\ar[r,"f'"swap]&B'
\end{tikzcd}\;.
\]
 Consider the morphism 
 \[
 (h,-j,0):M=\Cone(A\xrightarrow{[f,\alpha]^{\intercal}} B\oplus A')\rightarrow C.
 \] 
 Since $X''$ is the homotopy pushout of $S$, we have a morphism $j':B'\rightarrow C$ such that there exists a graded morphism $(-t, s_2, h'):M\rightarrow C$ of degree $-1$ such that 
 \[
 d(-t,s_2,h')=(-j's_1+h, -j+j'h_1, -j'f').
 \] 
 So we have 
 \[
 (-d(t)-h'\alpha-s_2f, d(s_2), d(h'))=(-j's_1+h, -j+j'h_1, -j'f').
 \]
 Thus we have a morphism $\mu:X\rightarrow X'$ of the form (\ref{const}), where $X'$ is the 3-term h-complex on the second row. 
By Proposition \ref{push}, we deduce that $X'$ is homotopy right exact and that if both $X$ and $X'$ are homotopy short exact sequences, then $X''$ is homotopy bicartesian. 
\end{proof}

The following is a direct consequence of Lemma \ref{Mor(A)and3term} and Proposition \ref{cons}.
\begin{corollary}\label{cok}
Let $\A$ be an additive dg category. 
For a homotopy cocartesian square $X$ as follows
\[
 \begin{tikzcd}
 A\ar[r,"f"]\ar[d,"g"swap]\ar[rd,"s_1"red,red]&B\ar[d,"h"]\\
 C\ar[r,"k"swap]&D
 \end{tikzcd}\;,
 \]
the object $f$ in $\Mor(\A)$ has a homotopy cokernel if and only if the object $k$ has a homotopy cokernel. 
If this is the case, there exists a morphism in $\mathcal H_{3t}(\A)$ from the homotopy cokernel of $f$ to that of $k$ such that its restriction to the third term is the identity, and its restriction to $f:A\rightarrow B$ is isomorphic to $X$.
\end{corollary}

\begin{lemma}\label{univ}
Suppose we are given a homotopy cartesian square $X\in \mathcal H_{\Sq}(\A)$ of the form
\[
\begin{tikzcd}
A\ar[r,"f"]\ar[d,"g"swap,""{name=2}]\ar[rd,"s"red,red]&B\ar[d,"j"]\\
C\ar[r,"k"swap]&D 
\end{tikzcd}, 
\]
and a morphism $\alpha$ from $f':A'\rightarrow B$ to $k:C\rightarrow D$ in $H^0(\Mor(\A))$
\[
\begin{tikzcd}
A'\ar[r,"f'"]\ar[d,"g'"swap]\ar[rd,"s'"red,red]&B\ar[d,"j"]\\
C\ar[r,"k"swap]&D
\end{tikzcd}\;.
\] 
We still denote by $X$ the associated morphism from $f:A\rightarrow B$ to $k:C\rightarrow D$ given by the homotopy cartesian square $X$. Then there exists a morphism $\beta: f'\rightarrow f$ in $H^0(\Mor(\A))$ such that $X\circ\beta=\alpha$ and that $\beta$ restricts to $\Id_{B}$.
\end{lemma}
\begin{proof}
Consider the morphism 
\[
A'\xrightarrow{[f',-g',s']^{\intercal}} M= \Sigma^{-1}\Cone(B\oplus C\xrightarrow{[j,k]} D).
\]
 Since $X$ is homotopy cartesian, we have a morphism $r:A'\rightarrow A$ such that there exists a graded morphism $[u,v,w]^{\intercal}:A'\rightarrow M$ of degree $-1$, such that 
 \[
 d([u,v,w]^{\intercal})=[f,-g,s]^{\intercal}\circ r-[f',-g',s']^{\intercal}.
 \]
 So we have $d(u)=fr-f'$, $d(v)=g'-gr$ and $d(w)=-ju-kv+s'-sr$.
 
 We have the following diagram
\[
\begin{tikzcd}
A'\ar[rrd,bend left=8ex,"f'"]\ar[rdd,bend right=8ex,"g'"swap]\ar[rrdd,dashed,bend left=5ex,red,"s'"]\ar[rd,dashed,"r"swap]&&\\
&A\ar[r,"f"]\ar[d,"g"swap,""{name=2}]\ar[rd,"s"{red,swap},red]&B\ar[d,"j"]\\
&C\ar[r,"k"swap]&D\mathrlap{.} 
\end{tikzcd} 
\]
So we have a morphism $\beta:f'\rightarrow f$ in $H^0(\Mor(\A))$ as follows
\[
\begin{tikzcd}
A'\ar[r,"f'"]\ar[d,"r"swap]\ar[rd,"u"red,red]&B\ar[d,equal]\\
A\ar[r,"f"swap]&B
\end{tikzcd}\;
\] 
such that the composition of $\beta$ with $X$ is homotopic to $\alpha$.
\end{proof}
We also have the following direct consequence of Lemma \ref{Mor(A)and3term} and Proposition \ref{cons}.
\begin{lemma}\label{deflationcomposition}
Let $j:B\rightarrow C$ and $j':C\rightarrow D$ be two objects in $H^0(\Mor(\A))$.
Assume that they both admit homotopy kernels, shown as follows, which are homotopy short exact
\[
\begin{tikzcd}
A\ar[r,"f"]\ar[rr,"h"swap, bend right=8ex]&B\ar[r,"j"]&C
\end{tikzcd}\;\;,\;\;
\begin{tikzcd}
A'\ar[r,"f'"]\ar[rr,"h'"swap, bend right=8ex]&C\ar[r,"j'"]&D.
\end{tikzcd}
\]
If the homotopy cospan $L$
\[
\begin{tikzcd}
&B\ar[d,"j"]\\
A'\ar[r,"f'"swap]&C
\end{tikzcd}
\]
admits a homotopy pullback, then we have the following diagram
\[
\begin{tikzcd}
A\ar[dd,bend right=12ex,"h''"{swap,blue},blue]\ar[r,equal]\ar[d,"u"swap]\ar[rd,"s'"{red},red]\ar[rdd,"t"{red},red]&A\ar[d,"f"swap]\ar[dd,bend left=8ex,"h"blue,blue, near end]&\\
E\ar[r,"w"swap,near start]\ar[d,"v"swap]\ar[rr,bend left=6ex,"h'''"{blue,near end},blue]\ar[rd,"s"{red,swap},red]&B\ar[r,"j'j"swap]\ar[d,"j"swap]&D\ar[d,equal]\\
A'\ar[r,"f'"swap]\ar[rr,"h'"{swap,blue},bend right=8ex,blue]&C\ar[r,"j'"swap]&D
\end{tikzcd}
\]
where
\begin{itemize}
\item[a)] The left bottom square is a homotopy pullback square.
\item[b)] The 3-term h-complex on the leftmost column is a homotopy left exact sequence. 
\item[c)] The 3-term h-complex in the middle row is homotopy left exact. It is homotopy short exact if the homotopy pullback of $L$ is homotopy short exact.
\item[d)] The 6-tuple $(\Id_{A},w,f',s',s,t)$ is a morphism in $\mathcal H_{3t}(\A)$ from the 3-term h-complex in the first column to that in the second column.
\item[e)] We have $h'''=j's+h'v$ and the 6-tuple $(v,j,\Id_{D},s,0,0)$ yields a morphism in $\mathcal H_{3t}(\A)$ from the 3-term h-complex in the second row to that in the bottom row.
\end{itemize}
\end{lemma}
Recall the diagram functor $\Dia: H^0(\Mor(\A)) \to \Mor(H^0(\A))$ which is an epivalence.
 By definition, an object $f: A \to B$ of $\Mor(\A)$ is a homotopy equivalence if $\Dia(f)$ is an isomorphism in $H^0(\A)$.
Therefore, an object $f: A \to B$ of $\Mor(\A)$ is a homotopy equivalence if and only if
this holds for any object $f': A' \to B'$ such that $\Dia(f)$ and $\Dia(f')$ are 
isomorphic in $\Mor(H^0(\A))$.
It is straightforward to verify the following lemma.
\begin{lemma}\label{homotopyequivalencestable} 
Consider an object
 \[
\begin{tikzcd}
A\ar[r,"f'"]\ar[d,"g"swap]\ar[rd,"h"red,red]&B\ar[d,"j"]\\
A'\ar[r,"f'"swap]&B'
\end{tikzcd}\;\;.
\] 
of $\mathcal H_{\Sq}(\A)$ which is homotopy cartesian.
If the morphism $j$ is a homotopy equivalence, then so is the morphism $g$.
\end{lemma}
\section{Exact dg categories}\label{sec:exactdgcategory}
In this section, we introduce the notion of exact structure on an additive dg category, 
based on the notion of homotopy short exact sequence, cf.~Definition~\ref{exactdgstructure}. 
We then establish that for an exact dg category $(\A,\mathcal S)$, there exists a canonical extriangulated structure $(H^0(\A),\mathbb E,\mathfrak s)$ on $H^0(\A)$, cf.~Theorem~\ref{thm:extriangulatedstructure}. 
Furthermore, we demonstrate, in Theorem~\ref{bijectionstructures}, the existence of a canonical bijection between the lattice of exact substructures of $(\A,\mathcal S)$ and the lattice of closed subbifunctors of $\mathbb E$.

\subsection{Exact dg structure}
Let $\A$ be a dg category over $k$ which is {\em additive}, i.e.~$H^0{\A}$ is additive. 
For an object $f:A\rightarrow B$ in $H^0(\Mor(\A))$, we denote by $[f]$ its isomorphism class and by $\overline{f}$ the corresponding morphism in $H^0(\A)$, i.e.~the object $\Dia(f)$ in $\Mor(H^0(\A))$. 
Recall that we have an epivalence $\Dia: H^0(\Mor(\A))\rightarrow \Mor(H^0(\A))$.
Therefore, if a property of objects in $H^0(\Mor(\A))$ is stable under isomorphisms, we will also say that the corresponding objects in $\Mor(H^0(\A))$ have the same property.
Note that if the object $f:A\rightarrow B$ is isomorphic to $f':A'\rightarrow B'$ and $g:B\rightarrow C$ is isomorphic to $g':B'\rightarrow C'$ in $H^0(\Mor(\A))$, it is not necessarily true that $gf$
is isomorphic to $g'f'$. 
But this will not cause confusion because in what follows, we consider objects in $H^0(\Mor(\A))$ up to isomorphism. 
For example, when we say ``compositions of deflations are deflations", 
it means that $gf$ is a deflation whenever $g$ and $f$ are deflations and then $g'f'$ is also a deflation, 
although $g'f'$ and $gf$ may not be isomorphic.
\begin{definition}\label{exactdgstructure}
An {\em exact structure} on $\A$ is a class $\mathcal{S}\subseteq \mathcal H_{3t}(\A)$ stable under isomorphisms, consisting of homotopy short exact sequences (called {\em conflations})
\[
\begin{tikzcd}
A\ar[r, tail,"i"]\ar[rr,bend right=8ex,"h"swap]&B\ar[r,two heads, "p"]&C\\
\end{tikzcd} 
\]
where $i$ is called an {\em inflation} and $p$ is called a {\em deflation}, such that the following axioms are satisfied:
\begin{itemize}
\item[Ex0]$\Id_{0}$ is a deflation.
\item[{Ex}1]Compositions of deflations are deflations.
\item[{Ex}2]Given a deflation $p:B\rightarrow C$ and any map $c: C'\rightarrow C$ in $Z^0(\A)$, the cospan 
$
\begin{tikzcd}
B\ar[r,"p", two heads]&C&C'\ar[l,"c"swap]
\end{tikzcd}
$
admits a homotopy pullback 
\[
\begin{tikzcd} 
{B'}\ar[r,"{p'}",two heads]\ar[d,"{b}"swap]\ar[rd,"s"blue,blue]&{C'}\ar[d,"{c}"]\\
{B}\ar[r,"{p}"swap,two heads]&{C}
\end{tikzcd}
\]
and ${p'}$ is also a deflation.
\item[$\Ex2^{op}$]Given an inflation $i: A\rightarrow B$ and any map $a:A\rightarrow A'$ in $Z^0(\A)$, the span
$
\begin{tikzcd}
A'&A\ar[r,tail,"i"]\ar[l,"a"swap]&B
\end{tikzcd}
$
admits a homotopy pushout 
\[
\begin{tikzcd}
{A}\ar[r,"{i}",tail]\ar[d,"{a}"swap]\ar[rd,"s"blue,blue]&{B}\ar[d,"{j}"]\\
{A'}\ar[r,"{i'}"swap,tail]&{B'}
\end{tikzcd}
\]
and ${i'}$ is also an inflation.
\end{itemize}
We call $(\A,\mathcal {S})$, or simply $\A$, an {\em exact dg category}.
\end{definition} 
\begin{remark}
Let us emphasize that our notion of exact dg category is completely different from Positselski's notion
of exact DG-category \cite{Positselski21}. For example, an exact structure in our sense can be
transported along a quasi-equivalence, cf.  Remark~\ref{truncationexactdgstructure} b), which is not at all the case for exact structures in the sense of Positselski. His principal motivation is to axiomatise
situations where we have a {\em strongly pretriangulated} dg category $\A$ whose category $Z^0(\A)$
is moreover endowed with a Quillen exact structure (for example the category of all 
complexes with components in a given Quillen exact category). By contrast, we aim at axiomatising
a class of {\em not necessarily pretriangulated} dg categories endowed with additional
structure.
\end{remark}
\begin{definition}\label{def:exactmorphism}
Let $(\A,\mathcal S)$ and $(\A',\mathcal S')$ be exact dg categories.
A morphism $F:\A\rightarrow \A'$ in $\Hqe$ is {\em exact} if the induced functor $\mathcal H_{3t}(\A)\rightarrow \mathcal H_{3t}(\A')$ sends objects in $\mathcal S$ to objects in $\mathcal S'$.
We denote by $\Hqe_{ex}(\A,\A')$ the subset of $\Hqe(\A,\A')$ consisiting of exact morphisms.
\end{definition}
\begin{remark}
Note that the class $\mathcal S\subseteq \mathcal H_{3t}(\A)$ is stable under isomorphisms. 
Let $\overline{\imath}$ be a morphism in $H^0(\A)$. 
If a representative $i$ of $\overline{\imath}$ is an inflation, then any representative in 
$\overline{\imath}$ is also an inflation.
Thus, being an inflation is a property of the morphism $\overline{\imath}$.
We have an even stronger property: If we view $\ol{\imath}:A_0\rightarrow A_1$ as an 
object in $\Mor(H^0(\A))$, then being an inflation is stable under isomorphisms in $\Mor(H^0(\A))$, cf.~Lemma \ref{Mor(A)and3term}.

In the setting of Axiom ${\Ex}2$, since the homotopy pullback is unique up to a unique isomorphism 
in $\mathcal H_{3t}(\A)$, by Lemma \ref{epi}, the object ${j'}$ is unique up to a canonical isomorphism
in $H^0(\Mor(\A))$.
\end{remark}
\begin{remark}\label{truncationexactdgstructure}Let $\A$ and $\A'$ be additive dg categories.
\begin{itemize}
\item[a)] By Lemma \ref{truncationhomotopycartesian} and its dual, the exact dg structures on $\A$ are in bijection with the exact structures on $\tau_{\leq 0}\A$.
\item[b)] Let $F:\A\rightarrow \A'$ be a quasi-equivalence of dg categories.
Then $F$ induces an equivalence of categories $\mathcal H_{3t}(\A)\iso\mathcal H_{3t}(\A')$ which 
preserves and reflects the property of being homotopy short exact.
Thus, the quasi-equivalence $F$ induces a bijection between the class of exact dg structures 
on $\A$ and that on $\A'$. 

For two objects $C$ and $A$ in $\A$, consider the set of `equivalence classes of extensions'
$\mathbb E(C,A)$ (resp.~$\mathbb E'(FC,FA)$), to be defined in Definition \ref{equivalencerelation}, which is associated with the exact structure on $\A$ (resp.~$\A'$). 
The quasi-equivalence $F$ induces a bijection between $\mathbb E(C,A)$ and $\mathbb E'(C,A)$, cf.~Lemma \ref{quasiequivalencebifunctor}. 
\item[c)] Let $\A$ be an exact dg category and $\A\rightarrow \pretr(\A)$ be the inclusion of $\A$ into its pretriangulated hull.
Let $\A'$ be the additive closure of $\A$ in $\pretr(\A)$.
Since $H^0(\A)$ is an additive category, the inclusion of $\A\rightarrow \A'$ is a quasi-equivalence.
By b), we may replace $\A$ by $\A'$ and assume that the dg category $\A$ is such that $Z^0(\A)$ is an additive category.

\end{itemize}
\end{remark}
Note that an object $f:A\rightarrow B$ in $H^0(\Mor(\A))$ is a deflation (resp. an inflation) if and only if it admits a homotopy kernel (resp. homotopy cokernel) (cf.~Definition \ref{def:homotopykernel}) which is a conflation. Observe that homotopy equivalences are both inflations and deflations by axioms 
Ex0, {Ex}2 and {Ex}$2^{op}$.

\begin{example}
Let $\A$ be an additive category which we consider as a dg category concentrated in degree 0. 

Let $A$, $B$ and $C$ be objects in $\A$ and consider a 3-term h-complex $X$ in $\A$
\[
\begin{tikzcd}
A\ar[rr,"h"swap,bend right=8ex]\ar[r,"f"]&B\ar[r,"j"]&C.
\end{tikzcd} 
\]
Since $\A$ is concentrated in degree zero, the homotopy $h$ is zero. 
By Example \ref{ordinary}, the 3-term h-complex $X$ is a homotopy short exact sequence if and only if the corresponding sequence $0\rightarrow A\xrightarrow {f}B\xrightarrow{j}C\rightarrow 0$ is a kernel-cokernel pair in the additive category $\A$. 
Also, two kernel-cokernel pairs are isomorphic (resp.~equivalent) if and only if the corresponding homotopy short exact sequences are isomorphic (resp.~equivalent). 

By \cite[Appendix A]{Keller90}, the axioms of a Quillen exact structure on $\A$ correspond to the axioms of an exact dg structure on $\A$.
Therefore, to endow an additive category $\A$ with a Quillen exact structure is the same as to endow $\A$ with an exact dg structure.
\end{example}
\begin{example}\label{exm:pretriangulated}
 Let $\A$ be a pretriangulated dg category.
 By Remark \ref{truncationexactdgstructure}, we may assume that the dg category $\A$ is strictly pretriangulated.
 Then each morphism $f:A\rightarrow B$ admits a homotopy cokernel
 \[
 \begin{tikzcd}
 A\ar[r,"f"] \ar[rr, bend right=8ex,"h"swap]&B\ar[r,"j"]&\Cone(f).
 \end{tikzcd}
 \]
 Dually, each morphism admits a homotopy kernel. 
 Also, a $3$-term homotopy complex in $\A$ is homotopy left exact if and only if it is homotopy right exact.
 It is not hard to check that the class of all homotopy short exact sequences in $\A$ defines an exact structure on $\A$.
 If not stated otherwise, we always consider this maximal exact structure on a pretriangulated dg category.
\end{example}
\begin{example-definition}\label{exactdgextensionclosed}
A full dg subcategory $\A'$ of an exact dg category $\A$ is {\em extension-closed} provided that for each conflation in $\A$
\[
\begin{tikzcd}
A\ar[r,"f"]\ar[rr,bend right=8ex,"h"swap]&B\ar[r,"j"]&C,
\end{tikzcd}
\]
if $A$ and $C$ belong to $\A'$, then so does $B$.

Let $\A'$ be an extension-closed subcategory of an exact dg category $(\A,\mathcal S)$.
Let $X$ be a conflation in $\A$ with all terms in $\A'$.
Then $X$ can be seen as a 3-term h-complex in $\A'$.
By Remark \ref{subcategory}, the object $X$ is a homotopy short exact sequence in $\A'$.

We have an inclusion functor $\mathcal H_{3t}(\A')\hookrightarrow \mathcal H_{3t}(\A)$.
Let $\mathcal S'$ be the class of objects in $\mathcal H_{3t}(\A')$ whose image in $\mathcal H_{3t}(\A)$ belongs to $\mathcal S$. 
We show that the class $\mathcal S'$ defines an exact dg structure on $\A'$.
Axiom ${\Ex}0$ is straightforward to check. 
Axiom ${\Ex}1$ follows from Lemma~\ref{deflationcomposition} and the fact that $\A'$ is extension-closed in $\A$.
Consider the span $S$ of a morphism in $\A'$ along an inflation in $\A'$.
It admits a homotopy pushout in $\A$, and by Proposition \ref{cons}, 
the homotopy pushout $X$ is an extension of objects in $\A'$ and hence remains an object in $\A'$. 
Again by Proposition \ref{cons}, the pushout of the inflation in $\A'$ is an inflation in $\A$ and 
the conflation that this inflation belongs to is in $\mathcal S'$.
Thus Axiom $\Ex2^{op}$ is satisfied.
Dually, one checks Axiom $\Ex2$.
\end{example-definition}

\begin{proposition}\label{property}
Suppose $(\A, \mathcal S)$ is an exact dg category. 
The following statements hold:
\begin{itemize}
\item[a)] The diagram
\[
\begin{tikzcd}
A\ar[r,"{[}1\ 0{]}^{\intercal}"]\ar[rr,"0"swap,bend right =8ex]&A\oplus B\ar[r,"{[}0\ 1{]}"]\ar[r]&B\\
\end{tikzcd}
\]
 is a conflation for $A,B\in \A$.
\item[b)]If a morphism $g:B\rightarrow C$ in $Z^0(\A)$ admits a homotopy kernel and for some $h:B'\rightarrow B$ the composition $gh:B'\rightarrow C$ is a deflation, then $g$ is a deflation.
\item[b)$^{\text{op}}$]If a morphism $f:A\rightarrow B$ admits a homotopy cokernel and for some $e:B\rightarrow B'$ the composition $ef:A\rightarrow B'$ is an inflation, then $f$ is an inflation.
\item[c)]\label{direct}The direct sum of two conflations is a conflation.
\item[d)]Axiom ${\Ex1}^{op}$ holds.
\end{itemize}
\end{proposition}
\begin{proof}Suppose we are in the situation of Axiom {Ex}2. 

\bigskip\noindent
{\em Claim 1:} $B\oplus C'\xrightarrow{[j,c]} C$ is a deflation.

We have a conflation $X\in\mathcal H_{3t}(\A)$ of the form
\[
\begin{tikzcd}
A\ar[r,"f"]\ar[rr,bend right=8ex,"h"swap]&B\ar[r,"j"]&C\mathrlap{.}\\
\end{tikzcd}
\]
By Axiom {Ex}2, the cospan $B\xrightarrow{j}C\xleftarrow{c}C'$
admits a homotopy pullback $Y$:
\[
\begin{tikzcd} 
 {B'}\ar[r," {j'}"]\ar[d," {b}"swap]\ar[rd,"s_2"red,red]& {C'}\ar[d," {c}"]\\
 {B}\ar[r," {j}"swap]& {C}
\end{tikzcd}\;\;.  
\]
By the dual of Proposition \ref{cons}, we have a morphism $\mu:X'\rightarrow X$ in $\mathcal H_{3t}(\A)$
\begin{equation}
\begin{tikzcd}\label{dia:diagramlemmaexact}
&A\ar[r,"f'"]\ar[d,equal]\ar[rd,"s_1"red,swap ,red]\ar[rrd,"t"blue,blue]\ar[rr,bend left = 8ex,"h'"]&B'\ar[d,"b"swap]\ar[rd,"s_2"red,red]\ar[r,"j'"]&C'\ar[d,"c"]\\
&A\ar[r,"f"swap]\ar[rr,bend right = 8ex,"h"swap]&B\ar[r,"j"swap]&C
\end{tikzcd}
\end{equation}
where $X'$ is the 3-term h-complex in the first row.
By Axiom {Ex}2, the morphism ${j'}$ is also a deflation. 
So $X'$ is a conflation and in particular a homotopy short exact sequence. 
Hence, by the dual of Proposition~\ref{cons}, $Y$ is homotopy bicartesian.

Let $X''$ be the 3-term h-complex
\begin{equation}
\begin{tikzcd}\label{cplx:rightexact}
 {B'}\ar[r,"\begin{bmatrix} {j'}\\ {b}\end{bmatrix}"]\ar[rr,"s_2"swap,bend right=8ex]& {C'}\oplus  {B}\ar[r,"{[} {c}{,}-{j}{]}"]&C\\
\end{tikzcd}\;\;. 
\end{equation}
Then $X''$ is a homotopy short exact sequence. 

\bigskip\noindent
{\em Claim 2:} $X''$ is a conflation.

Consider the conflation $X$ and the object $f':A\rightarrow B'$ in $H^0(\Mor(\A))$.
The span  $S:B' \xleftarrow{f'}A\xrightarrow{f}B$ 
admits a homotopy pushout $Z$
\[
\begin{tikzcd}
A\ar[rd,"s_3"{red},red]\ar[r,"f"]\ar[d,"f'"swap]& B\ar[d,"h_1=\begin{bmatrix}0\\\Id\end{bmatrix}"]\\
B'\ar[r,"f''=\begin{bmatrix}j'\\b\end{bmatrix}"swap]&C'\oplus B
\end{tikzcd}
\] 
where $s_3=\begin{bmatrix}-h'\\s_1\end{bmatrix}$. 
By Axiom {Ex}$2^{op}$, the morphism $B'\xrightarrow{[b,j']^{\intercal}} B\oplus C'$ is an inflation. 
So by Proposition \ref{cons}, there exists a morphism $X\rightarrow \tilde{X''}$ in $\mathcal H_{3t}(\A)$
\[
\begin{tikzcd}
&A\ar[r,"f"]\ar[d,"f'"swap]\ar[rd,"s_3"red,swap ,red]\ar[rrd,"t"blue,blue]\ar[rr,bend left = 8ex,"h"]&B\ar[d,"h_1"swap]\ar[rd,"s_2"red,red]\ar[r,"j"]&C\ar[d,equal]\\
&B'\ar[r,"f''"swap]\ar[rr,bend right = 8ex,"h''"swap]&C'\oplus B\ar[r,"j''"swap]&C
\end{tikzcd}
\]
where $\tilde{X''}$ is the 3-term h-complex in the second row.
By the uniqueness of homotopy cokernels, 
the homotopy right exact sequence $\tilde{X''}$ is isomorphic to $X''$.
By Axiom Ex$2^{op}$, $\tilde{X''}$ is a conflation. 
This proves Claim 2.
Then the morphism $B\oplus C'\xrightarrow{[j,c]} C$ is a deflation.
This proves Claim 1. 
Note that during the proof we only use Axioms {Ex}2 and {Ex}$2^{op}$, 
so the dual of the claims also holds.

Now we are ready to prove $\mathrm{b})$.  
Suppose we have a homotopy left exact sequence 
\[
\begin{tikzcd}
K\ar[r,"f"] \ar[rr,"h"swap,bend right=8ex]&B\ar[r,"g"]&C
\end{tikzcd}
\]
From the cospan 
$
\begin{tikzcd}
B'\ar[r,"gh", two heads]&C&B\ar[l,"g"swap]
\end{tikzcd}
$,
we know that the morphism $B\oplus B'\xrightarrow{[g,gh]} C$ is a deflation by claim 1. 
Then the morphism
\[
{[}g\ 0{]}:B\oplus B'\xrightarrow{\begin{bmatrix}1&-h\\0&1\end{bmatrix}} B\oplus B'\xrightarrow{\begin{bmatrix}g,gh\end{bmatrix}} C
\]
is also a deflation by Axiom {Ex}1. 
The 3-term h-complex
\[
\begin{tikzcd}
K\oplus B'\ar[r,"u"]\ar[rr,"{[}h{,}\;0{]}"swap,bend right=8ex]&B\oplus B'\ar[r,"v"]&C\mathrlap{,}\\
\end{tikzcd}
\]
where $u=\begin{bmatrix}f&0\\0&\Id_{B'}\end{bmatrix}$ and $v={[}g{,}\;0{]}$, is homotopy left exact, as the direct sum of a homotopy left exact sequence and a trivial 3-term h-complex.
Therefore it is a conflation. 
We have the following morphism of homotopy left exact sequences 
\[
\begin{tikzcd}
K\oplus B'\ar[rrd,"0"{blue,near end},blue,bend right=2ex]\ar[rd,"0"{red,swap},red,bend right=2ex]\ar[r,"u"]\ar[d,"{[}1{,} 0{]}"swap]\ar[rr,bend left=8ex,"{[}h{,}\;0{]}"] &B\oplus B'\ar[rd,"0"red,red]\ar[d,"{[}1{,}0{]}"swap]\ar[r,"v"]&C\ar[d,equal]\\
K\ar[r,"f"swap] \ar[rr,"h"swap,bend right=8ex]&B\ar[r,"g"swap]&C
\end{tikzcd} 
\]
where the left square is homotopy bicartesian, we infer that the morphism $g:B\rightarrow C$ is a deflation by {Ex}$2^{op}$. This proves Statement b).

Since the composition 
\[
\Id_B:B\xrightarrow{\begin{bmatrix}0\\1\end{bmatrix}} A\oplus B\xrightarrow{[0,\;1]} B
\] 
is a deflation and the morphism $A\oplus B\xrightarrow{[0,\;1]} B$ admits a homotopy kernel,
it is a deflation by Statement b). 
So Statement a) follows.

Let $X_i$, $i=1$, $2$, be a conflation of the form
\[
\begin{tikzcd}
A_i\ar[r,"f_i"]\ar[rr,bend right=8ex,"h_i"swap]&B_i\ar[r,"j_i"]&C_i\mathrlap{\;.}
\end{tikzcd}
\]
Clearly the direct sum $X_1\oplus X_2$ is homotopy short exact and the morphism
\[
\begin{tikzcd}
B_1\oplus B_2\ar[r,"\begin{bmatrix}j_1\ 0\\0\ j_2\end{bmatrix}"]&C_1\oplus C_2 
\end{tikzcd} 
\]
can be written as a composition 
\[\begin{tikzcd}
B_1\oplus B_2\ar[r,"r=\begin{bmatrix}1\ 0\\0\ j_2\end{bmatrix}"]&B_1\oplus C_2\ar[r,"s=\begin{bmatrix}j_1\ 0\\0\ 1\end{bmatrix}"]&C_1\oplus C_2
\end{tikzcd}.
\] 
Observe that the morphism $B_1\oplus B_2\xrightarrow{r} B_1\oplus C_2$ is a homotopy pullback of the deflation $j_2:B_2\rightarrow C_2$ along $B_1\oplus C_2\rightarrow  C_2 $ and hence is a deflation. 
Similarly, the morphism $B_1\oplus C_2\xrightarrow{s} C_1\oplus C_2$ is also a deflation. 
So the composition $B_1\oplus B_2\rightarrow C_1\oplus C_2$ is a deflation. 
This proves Statement c).

It remains to prove Axiom ${\Ex1}^{op}$ and then one proves b)$^{op}$ dually. 
Let $f:A\rightarrow B$ be an inflation with the homotopy cokernel given by
\[
\begin{tikzcd}
 {A}\ar[r," {f}"]\ar[rr,bend right=8ex,"h"swap]& {B}\ar[r,"j"]&C
\end{tikzcd}
\] 
Let $f':B\rightarrow B'$ be another inflation. 
Our aim is to show that the morphism $f'f$ is also an inflation. 
By Axiom ${\Ex2}^{op}$, the span $C\xleftarrow{j}B\xrightarrow{f'}B'$
admits a homotopy pushout 
\[
\begin{tikzcd}
 {B}\ar[r," {f'}"]\ar[d," {j}"swap]\ar[rd,"s"red,red]& {B'}\ar[d," {j'}"]\\
 {C}\ar[r," {f''}"swap]& {C'}
\end{tikzcd}
\]
which is homotopy bicartesian. 
Thus by the dual of Claim 1, 
the morphism $[ {j'}, {f''}]: {B'}\oplus  {C}\rightarrow  {C'}$ is a deflation. 
By the dual of Corollary \ref{cok}, the morphism ${j'}$ admits a homotopy kernel which is isomorphic to $f'f:A\rightarrow B'$.
We have the following commutative diagram in $\A$
\[
\begin{tikzcd}
 {B'}\oplus  {B}\ar[rd,"{[}0{,}-s{]}"red,red]\ar[r,"\begin{bmatrix}1\ 0\\0\  {j}\end{bmatrix}"]\ar[d,"{[}1{,} {f'}{]}"swap]& {B'}\oplus  {C}\ar[d,"{[} {j'}{,} {f''}{]}"]\\
 {B'}\ar[r," {j'}"swap]& {C'}
\end{tikzcd}.
\]
By Axiom ${\Ex1}$, the morphism $[ {j'}, {f''}]\begin{bmatrix}1&0\\0&{j}\end{bmatrix}$ is a deflation.
By the above diagram, the morphism ${j'}[1, {f'}]$ is homotopic to it and hence is also a deflation. 
So by Statement b), the morphism ${j'}$ is a deflation. 
Hence the morphism $f'f$ is an inflation.
This proves Statement d).
\end{proof}

The following diagram lemma is useful. 
It is a direct consequence of Lemma \ref{univ} and Axiom $\Ex2$ (or $\Ex2^{op}$).
\begin{lemma}\label{fact}
Let $(\A,\mathcal S)$ be an exact dg category. 
Suppose we are given conflations $X$
\[
 \begin{tikzcd}
 A\ar[r,"f"]          \ar[rr, bend right=8ex,"h"swap]                  &B\ar[r,"j"]&C
\end{tikzcd}
\]
and $X'$
\[
 \begin{tikzcd}
 A'\ar[r,"f'"]          \ar[rr, bend right=8ex,"h'"swap]                  &B'\ar[r,"j'"]&C'
\end{tikzcd}.
\]
Let $\alpha:X\rightarrow X'$ be a morphism in $\mathcal H_{3t}(\A)$. 
Denote by $\overline a:A\rightarrow A'$ (resp.~$\overline c:C\rightarrow C'$) the restriction of $\alpha$ to $A$ (resp.~$C$). 
Then $\alpha$ can be written as a composition $X\xrightarrow{\beta} \tilde{X}\xrightarrow{\gamma} X'$, where $\tilde{X}$ is a conflation with ends $A'$ and $C$, and the morphism $\beta$ restricts to $\overline a:A\rightarrow A'$ and $\Id_{C}$, and the morphism $\gamma$ restricts to $\Id_{A'}$ and $\overline{c}:C\rightarrow C'$.
\end{lemma}
\begin{remark}If we use the biadditive bifunctor $\mathbb E$ which is to be defined in \ref{bi}, then we have $[X']\in \mathbb E(C,A')$ and the morphism $\beta$ (resp. $\gamma$) shows that $[\tilde{X}]=\overline a_*[X]$ (resp. $[\tilde{X}]=\overline c^{*}[X']$).
\end{remark}
\begin{lemma}\label{lem:diagramlemmaexact}
Let $\A$ be an exact dg category. Consider a morphism of conflations given by diagram (\ref{dia:diagramlemma2})
\begin{equation}
\begin{tikzcd}\label{dia:diagramlemma2}
A\ar[r,tail,"f"]\ar[d,equal]\ar[rr,bend left=8ex,"h"]\ar[rd,red,"s_1"{swap}]\ar[rrd,blue,"t"{near end,description}]&B\ar[r,"j",two heads]\ar[d,"b"{near start,swap}]\ar[rd,red,"s_2"]&C\ar[d,"c"]\\
A\ar[r,tail,"f'"swap]\ar[rr,bend right=8ex,"h'"{swap}]&B'\ar[r,"j'"{swap},two heads]&C'
\end{tikzcd}.
\end{equation}
Then the diagram (\ref{dia:diagramlemma3})
\begin{equation}
\begin{tikzcd}\label{dia:diagramlemma3}
B\ar[r,"j",two heads]\ar[d,"b"{near start,swap}]\ar[rd,red,"s_2"]&C\ar[d,"c"]\\
B'\ar[r,"j'"{swap},two heads]&C'
\end{tikzcd}
\end{equation}
 is a homotopy pullback square, and hence the corresponding 3-term homotopy complex is a conflation, by Claim 1 in the proof of Proposition~\ref{property}.
\end{lemma}
\begin{proof}
The same proof for Claim 2 in Proposition~\ref{property} applies here as well.
Note that in that proof, we only use the homotopy right exactness of $3$-term homotopy complex~(\ref{cplx:rightexact}).
\end{proof}
Combining Lemma~\ref{fact}, Lemma~\ref{lem:diagramlemmaexact}, and Lemma~\ref{homotopyequivalencestable}, we have
\begin{corollary}\label{middleterm}
Let $\A$ be an exact dg category.
If a morphism $\alpha:X\rightarrow X'$ of conflations restricts to homotopy equivalences on both ends, it restricts to a homotopy equivalence between the middle terms.
\end{corollary}
\begin{definition}\label{equivalencerelation}
Let $(\A,\mathcal S)$ be a small exact dg category and $A,C$ two objects in $H^0\A$. 
Let $X_i$, $i=1,2$, be two conflations of the form
\begin{equation}\label{twoconflations}
\begin{tikzcd}
A\ar[r,"f_i"]\ar[rr,"h_i"swap,bend right=8ex]&B_i\ar[r,"j_i"]&C
\end{tikzcd}. 
\end{equation}
A morphism $\theta:X_1\rightarrow X_2$ in $\mathcal H_{3t}(\A)$ is called an {\em equivalence} if it restricts to $\Id_{A}$ and $\Id_{C}$. 
Note that by Corollary \ref{middleterm}, an equivalence of conflations is necessarily an isomorphism in $\mathcal H_{3t}(\A)$.
For a conflation, we denote by $[X]$ the equivalence class to which $X$ belongs.

We define $\mathbb E_{\mathcal S}(C,A)$ to be the set of equivalence classes of conflations 
\[
\begin{tikzcd}
A\ar[r,"f"]\ar[rr,"h"swap,bend right=8ex]&B\ar[r,"j"]&C
\end{tikzcd}
\]
with fixed ends $A$ and $C$.
When there is no risk of confusion, we will simply denote it by $\mathbb E(C,A)$.
\end{definition} 

\begin{lemma}\label{quasiequivalencebifunctor}
Let $F:\A\rightarrow \A'$ be a quasi-equivalence of dg categories. 
Suppose $\A$ is an exact dg category.
Equip $\A'$ with the induced exact structure from the quasi-equivalence $F$.
For two objects $C$, $A$ in $\A$, consider the sets $\mathbb E(C,A)$ (resp.~$\mathbb E'(FC,FA)$), which is associated with the exact structure on $\A$ (resp.~$\A'$). 
The quasi-equivalence $F$ induces a bijection 
\[
\mathbb E(C,A)\rightarrow \mathbb E'(FC,FA).
\]
\end{lemma}
\begin{proof}
It is straightforward to verify that the above map is an injection.
We now demonstrate that it is a surjection.
For an object $X\in\mathcal H_{3t}(\A)$, we denote by $FX$ the object in $\mathcal H_{3t}(\A')$ induced by $F$.

Let $[X']$ be an element in $\mathbb E'(FC,FA)$. 
Suppose $X'$ is of the form
\[
\begin{tikzcd}
FA\ar[r,"f'"]\ar[rr,"h'"swap,bend right=8ex]&B'\ar[r,"j'"]&FC\mathrlap{.}
\end{tikzcd}
\]
Since $F:\A\rightarrow\A'$ is a quasi-equivalence, we have a homotopy equivalence $FB\iso B'$ for some $B\in\A$.
Hence we have an isomorphism in $H^0(\Mor(\A'))$ of the form
\[
\begin{tikzcd}
FB\ar[d]\ar[rd,"s"red,red]\ar[r,"F{(}j{)}"]&FC\ar[d,equal]\\
B'\ar[r,"j'"swap]&FC\mathrlap{.}
\end{tikzcd}
\]
By Lemma \ref{Mor(A)and3term}, we may assume $B'=FB$ and $j'=F(j)$ for some $j:B\rightarrow C$ in $Z^0(\A)$.
The morphism $j:B\rightarrow C$ is a deflation. 
Suppose its homotopy kernel $Y$ is of the form
\[
\begin{tikzcd}
A'\ar[r,"\tilde{f}"]\ar[rr,"\tilde{h}"swap,bend right=8ex]&B\ar[r,"j"]&C\mathrlap{.}
\end{tikzcd}
\]
Then we have an isomorphism $\theta: X'\rightarrow FY$ in $\mathcal H_{3t}(\A)$ which restricts to the identity of $j':FB\rightarrow FC$ in $H^0(\Mor(\A'))$.
Let $\overline{a'}:FA'\rightarrow FA$ be the restriction of $\theta$ to $FA'$.
Then we have a morphism $a:A\rightarrow A'$ such that $F(a)$ is a homotopy inverse of $a'$.
We have an isomorphism in $H^0(\Mor(\A'))$ of the form
\[
\begin{tikzcd}
A\ar[r,"f=\tilde{f}a"]\ar[rd,"0"red,red]\ar[d,"a"swap]&B\ar[d,equal]\\
A'\ar[r,"\tilde{f}"swap]&B
\end{tikzcd}
\]
By Lemma \ref{Mor(A)and3term}, we deduce that $X'$ is equivalent to $FX$ where $X$ is of the form
\[
\begin{tikzcd}
A\ar[r,"f"]\ar[rr,"h"swap,bend right=8ex]&B\ar[r,"j"]&C\mathrlap{.}
\end{tikzcd}
\] 

\end{proof}

\subsection{The biadditive bifunctor $\mathbb E$}\label{bi}
Suppose we are given an element $[X]\in \mathbb E(C,A)$, where $X$ is as follows
\begin{equation}\label{X}
\begin{tikzcd}
A\ar[r,"{f}"]       \ar[rr,bend right=8ex,"{h}"swap]      &B\ar[r,"{j}"]         &C
\end{tikzcd},
\end{equation}
and a map $\overline{a}:A\rightarrow A'$ in $H^0(\A)$. 
Below we will prove that there exists a unique equivalence class of conflations $[X']\in \mathbb E(C,A')$, 
such that for each representative $X$ of $[X]$, 
there exists a representative $X'$ of $[X']$ with a morphism $\theta:X\rightarrow X'$ 
which restricts to the identity on $C$ and the morphism $\overline{a}:A\rightarrow A'$.  

By Axiom ${\Ex2}^{op}$, the existence of $[X']$ is clear from Proposition \ref{cons}. 
Let us clarify its uniqueness. 
Suppose we have morphisms $\theta_1:X\rightarrow X_1$ and $\theta_2:X\rightarrow X_2$ with the required property.
So the objects $X_1  $ and $X_2  $ are both conflations.
By Proposition \ref{cons}, we may assume that the restriction of $\theta_1$ to $f:A\rightarrow B$ is a homotopy cocartesian square. 
Let $S_i\in\mathcal H_{\Sq}(\A)$, $i=1$, $2$, be the restriction of $\theta_i$ to $f:A\rightarrow B$. 
Assume that $X_i  $, $i=1$, $2$, is of the form
 \[
 \begin{tikzcd}
 A'  \ar[r,"f_i"]\ar[rr,bend right=8ex,"h_i"swap]&B_i\ar[r,"j_i"]&C
 \end{tikzcd}\;,
 \]
 and that $S_1$ is a homotopy pushout of the cospan $L: A'\xleftarrow{a}A\xrightarrow{f}B$,
 where $a$ is a morphism in $Z^0(\A)$ representing the morphism $\overline{a}$. The homotopy square $S_2$ restricts to a cospan $L'$ and there exists a (non-unique) isomorphism $L\rightarrow L'$ that restricts to the identity of $f$.
 By the universal property of homotopy pushouts (cf.~Remark \ref{pullbackuniversal}), there is a unique morphism $\mu$ in $\mathcal H_{\Sq}(\A)$ from $S_1$ to $S_2$ that restricts to the isomorphism $L\rightarrow L'$.
 We continue to denote by $S_i$ the morphism in $H^0(\Mor(\A))$ from $f$ to $f_i$ given by the square $S_i$ for $i=1$, $2$.
 According to Lemma~\ref{squareepivalence}, the morphism $\mu:S_1\rightarrow S_2$ yields morphisms $\Id_{f}$, $S_1$, $S_2$, and a morphism $\alpha$ from $f_1:A'\rightarrow B_1$ to $f_2:A'\rightarrow B_2$ that restricts to $\Id_{A'}$ and satisfies $\alpha\circ S_1=S_2\circ \Id_{f}$ in $H^0(\Mor(\A))$. 
 By the universal property of homotopy pushouts, the morphism $\alpha$ induces a morphism $\theta_3:X_1\rightarrow X_2$ in $\mathcal H_{3t}(\A)$ such that $\theta_3\circ \theta_1=\theta_2$. 
 Then, the morphism $\theta_3$ is an equivalence between $X_1$ and $X_2$.
 Since $\theta_3$ is necessarily an isomorphism by Corollary \ref{middleterm}, we conclude that the morphism $\alpha:S_1\rightarrow S_2$ is an isomorphism.
\begin{proposition}The operation $\overline{a}:A\rightarrow A'\mapsto \overline{a}_*:\mathbb E(C,A)\rightarrow \mathbb E(C,A')$, 
where an element $[X]$ is sent to $\overline{a}_*[X]{\coloneqq}[X']$ which is specified by the above mentioned property, 
together with its dual operation $\overline{c}:C'\rightarrow C\mapsto \overline{c}^{*}:\mathbb E(C,A)\rightarrow \mathbb E(C',A)$, 
makes $\mathbb E:H^{0}(\A)^{op}\times H^{0}(\A)\rightarrow  \Set$ into a bifunctor.
\end{proposition}

\begin{remark} \label{rk:extriangulated-structure}
In Subsection~\ref{canonicalstructure}, we will prove that $H^0(\A)$ carries a canonical extriangulated structure in the
sense of Nakaoka--Palu \cite{NakaokaPalu19} whose extension bifunctor is $\mathbb{E}$.
\end{remark}

\begin{proof}Suppose we are given $[X]\in \mathbb E(C,A)$, $\overline{a}:A\rightarrow A'$ and $\overline{c}:C'\rightarrow C$, where $X$ is given by the diagram (\ref{X}).

Put $\overline a_*[X]=[X_1]$, $\overline c^{*}[X_1]=[X_2]$ and $\overline c^*[X]=[X_3]$. 
Then the conflations $X_1$, $X_2$ and $X_3$ take the following form, respectively
 \[
 \begin{tikzcd}
 A'\ar[r,"f_1"]\ar[rr,bend right=8ex,"h_1"swap]&E_1\ar[r,"j_1"]&C &A'\ar[r,"f_2"]\ar[rr,bend right=8ex,"h_2"swap]&E_2\ar[r,"j_2"]&C'&A\ar[r,"f_3"]\ar[rr,bend right=8ex,"h_3"swap]&E_3\ar[r,"j_3"]&C'
 \end{tikzcd}.
 \]
 
By the definition of $[X_2]$, we have a homotopy cartesian square which gives rise to a morphism $\alpha$ in $H^0(\Mor(\A))$ from $j_2:E_2\rightarrow C'$ to $j_1:E_1\rightarrow C$. 
We also have a morphism $\beta$ from $j_3:E_3\rightarrow C'$ to $j_1:E_1\rightarrow C'$. 
By Lemma \ref{univ}, we have a morphism $\gamma$ from $j_3:E_3\rightarrow C'$ to $j_2:E_2\rightarrow C'$ which restricts to $\Id_{C'}$ and such that $\alpha\circ \gamma=\beta$. 
The associated diagram in $H^0(\A)$ is as follows:
\[
\begin{tikzcd}[cramped,sep=small]
&A\ar[rr,"\overline{f_3}"]\ar[dd,dotted,bend left=6ex]\ar[ld,equal]&
&E_3\ar[rr,"\overline{\jmath_3}"]\ar[dd, dotted,bend left=6ex]\ar[ld]&
&C'\ar[dd,equal]\ar[ld,"\overline{c}"]\\
A\ar[rr,"\overline{f}"{description}]\ar[dd,"\overline{a}"swap,bend left=6ex]&&B\ar[rr,"\overline{\jmath}"{description}]\ar[dd, bend left=6ex]
&&C\ar[dd,equal]
&\\
&A'\ar[rr,"\overline{f_2}"{description}]\ar[ld,equal]&
&E_2\ar[rr,"\overline{\jmath_2}"{description,near start}]\ar[ld]&
&C'\ar[ld,"\overline{c}"]\\
A'\ar[rr,"\overline{f_1}"swap]&&E_1\ar[rr,"\overline{\jmath_1}"swap]
&&
C&
\end{tikzcd}
\]
By universal property of the homotopy pullbacks (cf.~Remark \ref{pullbackuniversal}), the morphism $\gamma$ induces a morphism from $X_3$ to $X_2$. 
It clearly restricts to $\overline a:A\rightarrow A'$. 
Therefore $\overline a_*[X_3]=[X_2]$.
\end{proof}

Our next goal is to demonstrate that $\mathbb E(C,A)$ forms an abelian group for every pair $(C,A)$ of objects in $H^0(\A)$, ensuring that $\mathbb E:H^0(\A)^{op}\times H^0(\A)\rightarrow \Ab$ is a biadditive bifunctor. 

We begin by describing the addition operation on $\mathbb E(C,A)$.
Let $[X]$ and $[X']$ be two elements in $\mathbb E(C,A)$ given by the diagram (\ref{twoconflations}). 
Put $\overline a=[1\ 1]:A\oplus A\rightarrow A$ and $\overline c=[1\ 1]^{\intercal}:C\rightarrow C\oplus C$. Recall from Proposition \ref{direct} that the direct sum $X\oplus X'$ is also a conflation. 
Then the sum $[X]+[X']$ is defined to be $\overline c^{*}\overline a_*[X\oplus X']$. 
\begin{lemma} The above addition operation is well-defined, associative and commutative and that the equivalence class of 
\[
\begin{tikzcd}
A\ar[r,"\begin{bmatrix}1\\0\end{bmatrix}"]\ar[rr,"0"swap,bend right=8ex]&A\oplus C\ar[r,"{[}0{,}1{]}"]&C
\end{tikzcd}
\]
is a zero element. 
\end{lemma}
\begin{proof}
It is straightforward to check the well-definedness and commutativity.
To show that the above conflation is a zero element, observe that we have the following morphism in $\mathcal H_{3t}(\A)$
\[
\begin{tikzcd}
A\ar[rrd,"0"{blue,near start,description},blue]\ar[d,"{[}0{,}1{]^{\intercal}}"swap]\ar[rr,bend left=8ex,"h"]\ar[r,"f"]\ar[rd,"s"{red,swap},red,bend right=2ex]&B\ar[rd,"0"red,red,bend left=2ex]\ar[d,"{[}0{,}\;j{,}\;1{]^{\intercal}}"]\ar[r,"j"]&C\ar[d,"{[}1{,}\;1{]^{\intercal}}"]\\
A\oplus A\ar[rr,bend right=8ex,"h'"swap]\ar[r,"f'"swap]&A\oplus C\oplus B\ar[r,"j'"swap]&C\oplus C
\end{tikzcd}
\]
where $f'=\begin{bmatrix}1&0\\0&0\\0&f\end{bmatrix}$, $j'=\begin{bmatrix}0&1&0\\0&0&j\end{bmatrix}$, $h'=\begin{bmatrix}0&0\\0&h\end{bmatrix}$, $s=[0,\;h,\;0]^{\intercal}$.

To show the associativity, observe that we have the equality
\[
\begin{bmatrix}1&0\\0&1\\0&1\end{bmatrix}\circ\begin{bmatrix}1\\1\end{bmatrix}=\begin{bmatrix}1&0\\1&0\\0&1\end{bmatrix}\circ\begin{bmatrix}1\\1\end{bmatrix}:A\rightarrow A\oplus A\oplus A.
\]
\end{proof}
Conflations in the equivalence class of zero elements are called {\sl splitting}. 

\begin{proposition}\label{split}Let $X$ be a conflation as follows
\[
\begin{tikzcd}
A\ar[r,"f"]\ar[rr,bend right=8ex,"h"swap]&B\ar[r,"j"]&C
\end{tikzcd}.  
\]
The following statements are equivalent:
\begin{itemize}
\item[1)]$X$ is a splitting conflation.
\item[2)]$\overline f$ is a split monomorphism in $H^{0}(\A)$.
\item[2)$^{\text{op}}$]$\overline{\jmath}$ is a split epimorphism in $H^{0}(\A)$.
\end{itemize} 
\end{proposition}
\begin{proof}
The implication $1)\Rightarrow 2)$ follows from Corollary \ref{middleterm}. 

We show $2)\Rightarrow 1)$. 

Suppose $\overline f$ is a split monomorphism in $H^{0}(\A)$. 
Then we have an isomorphism from $A\rightarrow A\oplus C'$ to $f:A\rightarrow B$ in $H^0(\Mor(\A))$, which restricts to $\Id_{A}$. 
So their homotopy cokernels are isomorphic and hence $C'$ is homotopy equivalent to $C$. 
So we have an isomorphism in $H^0(\Mor(\A))$ from $f'=(1\ 0)^{\intercal}: A\rightarrow A\oplus C$ to $f:A\rightarrow B$ which restricts to $\Id_{A}$. 
Then it induces an isomorphism $\alpha$ in $\mathcal H_{3t}(\A)$ 
\[
\begin{tikzcd}
A\ar[d,"h_0"swap]\ar[r,"{[}1{,}0{]^{\intercal}}"]\ar[rr,bend left=8ex,"0"]&A\oplus C\ar[r,"{[}0\ 1{]}"]\ar[d,"h_1"]&C\ar[d,"h_2"]\\
A\ar[r,"f"swap]\ar[rr,bend right=8ex,"h"swap]&B\ar[r,"j"swap]&C
\end{tikzcd}
\]
where we have omitted the diagonal morphisms and $h_0$ is homotopic to $\Id_{A}$.
Let $c$ be a homotopy inverse of $h_2$. 
Then the morphism $c$ gives rise to a morphism $\beta$ 
\[
\begin{tikzcd}
A\ar[d,equal]\ar[r,"{[}1{,}0{]^{\intercal}}"]\ar[rr,bend left=8ex,"0"]&A\oplus C\ar[r,"{[}0\ 1{]}"]\ar[d,"\begin{bmatrix}1\;\;0\\0\;\;c
\end{bmatrix}"]&C\ar[d,"c"]\\
A\ar[r,"{[}1{,}0{]^{\intercal}}"swap]\ar[rr,bend right=8ex,"0"]&A\oplus C\ar[r,"{[}0\ 1{]}"swap]&C
\end{tikzcd}\;\;.
\]
So the composition $\alpha\circ \beta $ is an equivalence between $X$ and a splitting conflation and hence $X$ is a splitting conflation.
\end{proof}

\begin{corollary}\label{zero}Let $X$ be a conflation of the form
\[
 \begin{tikzcd}
 A\ar[r,"f"]\ar[rr,"h"swap,bend right=8ex]&B\ar[r,"j"]&C
 \end{tikzcd}.
 \]
\begin{itemize}
\item[1)] The element $\overline{\jmath}^*[X]$ is zero in $\mathbb E(B,A)$. 
\item[2)]For any zero morphism $0:W\rightarrow C$ in $H^0(\A)$, the corresponding element $0^*[X]$ is zero in $\mathbb E(W,A)$.
\end{itemize}
\end{corollary}

\begin{proposition}
Let $[X]$ be an element in $\mathbb E(C,A)$ where $X$ is of the form
\[
\begin{tikzcd}
A\ar[r,"f"]\ar[rr,"h"swap,bend right=8ex]&B\ar[r,"j"]&C
\end{tikzcd}.
\]
 Then $[-X]$, where $-X$ is of the form
\[
\begin{tikzcd}
A\ar[r,"-f"]\ar[rr,"-h"swap,bend right=8ex]&B\ar[r,"j"]&C
\end{tikzcd}
\]
is an inverse of $[X]$ under the addition defined above. So this addition makes $\mathbb E(C,A) $ into an abelian group. 
\end{proposition}
\begin{proof}
It is straightforward to check that the element $[-X]$ is well-defined. 
Consider the following diagram in $\A$
\[
\begin{tikzcd}
A\ar[rr,bend left=8ex,"h"]\ar[d,swap,"\begin{bmatrix}-1\\1\end{bmatrix}"]\ar[r,"f"]&B\ar[d,"\begin{bmatrix}1\\1\end{bmatrix}"]\ar[r,"j"]&C\ar[d,"\begin{bmatrix}1\\1\end{bmatrix}"]\\
A\oplus A\ar[rr,"h'"swap,bend right=8ex]\ar[r,"f'"swap]&B\oplus B\ar[r,"j'"swap]&C\oplus C
\end{tikzcd},
\]
where the diagonal morphisms are zero and where $f'=\begin{bmatrix}-f&0\\0&f\end{bmatrix}$, $j'=\begin{bmatrix}j&0\\0&j\end{bmatrix}$ and $h'=\begin{bmatrix}-h&0\\0&h\end{bmatrix}$. This diagram gives rise to a morphism $\alpha:X\rightarrow -X\oplus X$ in $\mathcal H_{3t}(\A)$. 
Next, we apply Lemma \ref{fact} to the morphism $\alpha$. 
Put $a=[1, 1]:A\oplus A\rightarrow A$, $a'=[-1, 1]^{\intercal}:A\rightarrow A\oplus A$ and $c=[1,1]^{\intercal}:C\rightarrow C\oplus C$. 
Then we have $\alpha=\gamma\circ\beta$ where $\beta:X\rightarrow\tilde{X}$, $\gamma:\tilde{X}\rightarrow -X\oplus X$, and $[\tilde{X}]=\overline{a'}_{*}[X]=\overline c^*[-X\oplus X]$.
So we have
\[
[X]+[-X]=\overline a_{*}\overline c^*[-X\oplus X]=\overline a_{*}[\tilde{X}]=\overline a_{*}\overline {a'}_{*}[X]=0_*[X]=0.
\] 
This shows that $[-X]$ is an additive inverse of $[X]$.

\end{proof}

\begin{proposition} The bifunctor $\mathbb E$ is a biadditive bifunctor from $H^0(\A)^{op}\times H^0(\A)$ to the category $\Ab$ of abelian groups.
\end{proposition}
\begin{proof}
Let $X_i$, $i=1$, $2$, be a conflation of the form 
\[
\begin{tikzcd}
A\ar[r,"f_i"]\ar[rr,"h_i"swap,bend right=8ex]&B_i\ar[r,"j_i"]&C
\end{tikzcd}
\]
and $\overline b:A\rightarrow A'$ a morphism in $H^0(\A)$. 
Put $a=[1,1]:A\oplus A\rightarrow A$, $a'=[1,1]:A'\oplus A'\rightarrow A'$, $c=[1,1]^{\intercal}:C\rightarrow C\oplus C$ and $d=\begin{bmatrix}b&0\\0&b\end{bmatrix}:A\oplus A\rightarrow A'\oplus A'$. 
We show that $\overline{b}_{*}:\mathbb E(C,A)\rightarrow\mathbb E(C,A')$ is a morphism of abelian groups.
Put $[\tilde{X_i}]=\overline b_*[X_i]$, $i=1,2$. 
Then, by definition, we have $d_*[X_1\oplus X_2]=[\tilde{X_1}\oplus \tilde{X_2}]$. 
So we have
\[
\begin{aligned}
\overline b_*([X_1]+[X_2])&=\overline b_*\overline a_*\overline c^*[X_1\oplus X_2]=\overline a'_*\overline d_*\overline c^{*}[X_1\oplus X_2]\\
&=\overline a'_*\overline c^{*}\overline d_*[X_1\oplus X_2]=\overline a'_*\overline c^{*}[\tilde{X_1}\oplus \tilde{X_2}]\\
&=\tilde{X_1}+\tilde{X_2}=\overline b_*[X_1]+\overline b_*[X_2].
\end{aligned}
\] 
So the map $\overline{b}_*$ is a morphism of abelian groups.

Let $X$ be a conflation of the form
\[
\begin{tikzcd}
A\ar[r,"f"]\ar[rr,"h"swap,bend right=8ex]&B\ar[r,"j"]&C
\end{tikzcd}.
\]
Suppose we have morphisms $\overline b_i:A\rightarrow A'$ in $H^0(\A)$ for $i=1$, $2$. We show that $(\overline b_1+\overline b_2)_{*}=(\overline b_1)_*+(\overline b_2)_*$.
Put $d'=\begin{bmatrix}b_1&0\\0&b_2
\end{bmatrix}$ and $a''=[1,1]^{\intercal}:A\rightarrow A\oplus A$. 
Put $[Y_i]=(\overline b_i)_*[X]$ for $i=1$, $2$. 
We have a canonical morphism $\alpha:X \rightarrow X\oplus X$ which is componentwise given by the morphism $[1,1]^{\intercal}$. 
By Lemma \ref{fact}, the morphism $\alpha$ factorizes as $\alpha=\gamma\circ \beta$ where we have $\beta: X\rightarrow \tilde{X}$, $\gamma:\tilde{X}\rightarrow X\oplus X$ and $[\tilde{X}]=\overline a''_*[X]=\overline c^*[X\oplus X]$. 
Then we have
\[
\begin{aligned}
(\overline b_1+\overline b_2)_*[X]&=\overline a'_*\overline d'_*\overline a''_*[X]=\overline a'_*\overline d'_*\overline c^*[X\oplus X]\\
&=\overline a'_*\overline c^*\overline d'_{*}[X\oplus X]=\overline a'_*\overline c^*[Y_1\oplus Y_2]\\
&=Y_1+Y_2=(\overline b_1)_*[X]+(\overline b_2)_*[X].
\end{aligned}
\] 
Thus we have
$(\overline b_1+\overline b_2)_*=(\overline b_1)_*+(\overline b_2)_*$. 
This, together with its dual, shows that the bifunctor $\mathbb E$ is biadditive.
\end{proof}
For two elements $\delta_i=[X_i]\in \mathbb E(C_i,A_i)$, $i=1$, $2$, let $\delta_1\oplus \delta_2$ be the element in 
\[
\mathbb E(C_1\oplus C_2, A_1\oplus A_2)\simeq \mathbb E(C_1,A_1)\oplus \mathbb E(C_2, A_2)\oplus \mathbb E(C_1, A_2)\oplus \mathbb E(C_2, A_1)
\]
corresponding to the element $(\delta_1,\delta_2,0,0)$.
\begin{lemma}\label{sum}$\delta_1\oplus \delta_2$ is the equivalence class of $X_1\oplus X_2$.
\end{lemma}
\begin{proof}Denote by $a_i:A_1\oplus A_2\rightarrow A_i$ the canonical projection for $i=1$, $2$. 
Denote by $c_i:C_i\rightarrow C_1\oplus C_2$ the canonical inclusion for $i=1$, $2$. 
It is enough to show that we have $c_j^*(a_i)_*[X_1\oplus X_2]=0$ for $i\neq j$ and $c_i^*(a_i)_*[X_1\oplus X_2]=[X_i]$.
We show the case when $i=1$. 
We have $(a_1)_*[X_1\oplus X_2]=[X_1\oplus Y]$ where $Y$ is the conflation
\[
\begin{tikzcd}
0\ar[r]\ar[rr,"0"swap,bend right=8ex]&C_2\ar[r,equal]&C_2
\end{tikzcd}.
\]
So we have $c_1^*(a_1)_*[X_1\oplus X_2]=[X_1]$. 
By Corollary~\ref{zero}, we have $c_2^*(a_1)_*[X_1\oplus X_2]=0$.
\end{proof}
\subsection{The canonical extriangulated structure on $H^0(\A)$}\label{canonicalstructure}
In this subsection, we show that $H^0(\A)$ carries a canonical extriangulated structure, as announced
in Remark~\ref{rk:extriangulated-structure}.

We define the realisation $\mathfrak s$ of the bifunctor $\mathbb E$ as follows: for an element $\delta=[X]\in \mathbb E(C,A)$, where $X$ is as follows
\[
\begin{tikzcd}
A\ar[r,"f"]\ar[rr,bend right=8ex,"h"swap]&B\ar[r,"j"]&C
\end{tikzcd},
\]
put $\mathfrak s(\delta)=[A\xrightarrow{\overline{f}} B\xrightarrow{\overline{\jmath}} C]$. 
This is well-defined by Corollary \ref{middleterm}.
\begin{lemma}The map $\mathfrak s$ defined above is an additive realization of $\mathbb E$.
\end{lemma}
\begin{proof}We first show that $\mathfrak s$ is a realization.
Suppose we have $\delta=[X]\in \mathbb E(C,A)$, $\delta'=[X']\in \mathbb E(C',A')$, $\overline{a}:A\rightarrow A'$ and $\overline{c}:C\rightarrow C'$ such that $\overline{a}_*[X]=\overline{c}^*[X']=[X'']\in \mathbb E(C,A')$. Put $\mathfrak s(\delta)=[A\xrightarrow{\overline{f}} B\xrightarrow{\overline{\jmath}} C]$ and $\mathfrak s(\delta')=[A'\xrightarrow{\overline{f'}} B'\xrightarrow{\overline{\jmath'}} C']$.
We have a composition of morphism of conflations $X\rightarrow X''\rightarrow X'$ which restricts to $\overline{a}:A\rightarrow A'$ and $\overline{c}:C\rightarrow C'$. 
This morphism then gives a morphism $\overline{b}:B\rightarrow B'$ which makes the following diagram in $H^0(\A)$ commutative
\[
\begin{tikzcd}
A\ar[r,"\overline{f}"]\ar[d,"\overline{a}"swap]& B\ar[r,"\overline{\jmath}"]\ar[d,"\overline{b}"] &C\ar[d,"\overline{c}"]\\
A'\ar[r,"\overline{f'}"swap]&B'\ar[r,"\overline{\jmath'}"swap]&C'
\end{tikzcd}.
\]

We now show that $\mathfrak s$ is additive.
By definition, for $0\in\mathbb E(C,A)$, we have 
\[
\mathfrak s(0)=[A\xrightarrow{\begin{bmatrix}0\\1\end{bmatrix}}A\oplus C\xrightarrow{[0{,}1]}C].
\]
By Lemma \ref{sum}, we have $\mathfrak s(\delta\oplus \delta')=\mathfrak s(\delta)\oplus \mathfrak s(\delta')$.
\end{proof}
\begin{remark}\label{truncationextriangulated}
By Remark \ref{truncationexactdgstructure}, the exact dg structures on $\A$ are the same as 
the exact dg structures on $\tau_{\leq 0}\A$. 
Note that we have $H^0(\tau_{\leq 0}\A)=H^0(\A)$. 
By definition, the triple $(H^0(\A), \mathbb E,\mathfrak s)$ remains unchanged when we replace $\A$ by $\tau_{\leq 0}\A$. 
\end{remark}
\begin{theorem}\label{thm:extriangulatedstructure}
The triple $(H^0(\A),\mathbb E,\mathfrak s)$ forms an extriangulated category.
\end{theorem}
\begin{proof}

$(\mathrm{ET3})$ Let $\delta=[X]\in \mathbb E(C,A)$, $\delta'=[X']\in \mathbb E(C',A')$. 
Suppose they are realized as $\mathfrak s(\delta)=[A\xrightarrow{\overline{f}} B\xrightarrow{\overline{\jmath}} C]$ and $\mathfrak s(\delta')=[A'\xrightarrow{\overline{f'}} B'\xrightarrow{\overline{\jmath'}} C']$.
For any commutative square in $H^0(\A)$
\[
\begin{tikzcd}
A\ar[r,"\overline{f}"]\ar[d,"a"swap]&B\ar[r,"\overline{\jmath}"]\ar[d,"b"]&C\\
A'\ar[r,"\overline{f'}"swap]&B'\ar[r,"\overline{\jmath'}"swap]&C'
\end{tikzcd},
\]
the left hand square can be lifted (in a non-unique way) to a morphism from $f:A\rightarrow B$ to $f':A'\rightarrow B'$ in $H^0(\Mor(\A))$. 
It then induces a morphism $X\rightarrow X'$ in $\mathcal H_{3t}(\A)$. 
The restriction gives a morphism $c:C\rightarrow C'$ in $H^0(\A)$.
By  Lemma \ref{fact},  we have $a_*\delta=c^*\delta'$. 
This proves (ET3). 
Axiom \text{(ET3)$^{op}$} is proved dually.

$(\mathrm{ET4})$ Consider $\delta=[X]\in \mathbb E(D,A)$ and $\delta'=[X']\in\mathbb E(F,B)$ which are realized respectively by $A\xrightarrow{\overline{f}}B\xrightarrow{\overline{\jmath}}D $ and $B\xrightarrow{\overline{f'}}C\xrightarrow{\overline{\jmath'}}F$. 

Since a composition of inflations is an inflation, the object $f'f:A\rightarrow C$ is an inflation and there exists a conflation $X''$ such that $[X'']$ is realized by $A\xrightarrow{\overline{f'f}}C\xrightarrow{\overline{h}}E$.
The commutative diagram 
\[
\begin{tikzcd}
A\ar[r,"\overline{f}"]\ar[d,equal]&B\ar[d,"\overline{f'}"]\\
A\ar[r,"\overline{f'f}"swap]&C
\end{tikzcd}
\]
in $\D(\A)$ can be lifted to a morphism in $H^0(\Mor(\A))$ from $f:A\rightarrow B$ to $f'f:A\rightarrow C$. 
This morphism induces a morphism from $X$ to $X''$ which restricts to a homotopy bicartesian square $Y_1$ whose image in $\Fun(\mathrm{Sq},\D(\A))$ is isomorphic to
\[
\begin{tikzcd}
B\ar[r,"\overline{\jmath}"]\ar[d,"\overline{f'}"swap]&D\ar[d,"\overline{d}"]\\
C\ar[r,"\overline{h}"swap]&E
\end{tikzcd}\;.
\]
Let $\overline{\jmath}_* [X']=[X''']\in \mathbb E(F,B)$. 
We have a morphism $X'\rightarrow X'''$ which, by restriction to $f':B\rightarrow C$, yields a homotopy bicartesian square $Y_2$. 
Both $Y_1$ and $Y_2$ are homotopy pushouts of $S:$
\[
\begin{tikzcd}
B\ar[r,"j"]\ar[d,"f'"swap]&D\\
C&
\end{tikzcd}.
\] 
So there exists an isomorphism $Y_1\rightarrow Y_2$ in $\D(\mathrm{Sq})$ which is compatible with the isomorphisms $R(Y_i)\rightarrow S$ where $R:\D(\mathrm{Sq})\rightarrow \D(\mathrm{Sp})$ is the restriction functor.
So $[X''']$ can be realized by $D\xrightarrow{\overline{d}}E\xrightarrow{\overline{e}}F$ such that $\overline{\jmath'}=\overline{e}\overline{h}$. This shows (ET4). Axiom \text(ET4$^{op}$) can be shown dually.
\end{proof}
\begin{definition}\label{def:algebraicextriangulated}
An extriangulated category $\C$ is {\em algebraic}, if it is equivalent, as an extriangulated category, to $(H^0(\A), \mathbb E,\mathfrak s)$ for an exact dg category $\A$.
\end{definition}
By Remark~\ref{truncationextriangulated}, the underlying dg category of an exact dg category, which enhances an extriangulated category, could always be taken to be connective.
\subsection{Subcategories, closed subbifunctors and exact dg substructures}
In this subsection, for an exact dg category $(\A,\mathcal S)$, we demonstrate the existence of a canonical bijection between the lattice of exact substructures of $(\A,\mathcal S)$ and the lattice of closed subbifunctors of $\mathbb E$, where $(H^0(\A),\mathbb E,\mathfrak s)$ is the associated extriangulated category of $(\A,\mathcal S)$, cf.~Theorem~\ref{bijectionstructures}.
By leveraging Ogawa's notion of defect~\cite[Definition 2.4]{Ogawa21} and Iyama--Nakaoka--Palu's notion of almost split extension~\cite[Definition 2.1]{IyamaNakaokaPalu24}, for certain class of exact dg categories, we establish a bijection between the class of exact substructures and the Boolean lattice of subsets of a certain set of objects, cf.~Corollary~\ref{cor:KrullSchmidt}.

Let $(\C,\mathbb E,\mathfrak s)$ be an extriangulated category. 
\begin{definition}[\cite{HerschendLiuNakaoka21}, Lemma 3.15]Let $\mathbb F\subseteq \mathbb E$ be an additive subbifunctor. The following are equivalent.
\begin{itemize}
\item[(1)]$\mathbb F$ is closed on the right, i.e.~for any $\mathfrak s|_{\mathbb F}$-conflation $A\xrightarrow{f} B\xrightarrow{j} C$, the sequence
\[
\mathbb F(-,C)\xRightarrow{f_{*}} \mathbb F(-,B)\xRightarrow{j_{*}} \mathbb F(-,A)
\]
is exact.
\item[(2)]$\mathbb F$ is closed on the left, i.e.~for any $\mathfrak s|_{\mathbb F}$-conflation $A\xrightarrow{f} B\xrightarrow{j} C$, the sequence
\[
\mathbb F(A,-)\xRightarrow{j^{*}} \mathbb F(B,-)\xRightarrow{f^{*}} \mathbb F(C,-)
\]
is exact.

\end{itemize}
Thus, we simply say $\mathbb F\subseteq \mathbb E$ is {\em closed}, if either of the conditions is satisfied. 
\end{definition}
\begin{proposition}[\cite{HerschendLiuNakaoka21}, Proposition 3.16]\label{closedsubbifunctor}
For any additive subbifunctor $\mathbb F\subseteq\mathbb E$, the following are equivalent.
\begin{itemize}
\item[(1)]$(\C,\mathbb F,\mathfrak s|_{\mathbb F})$ is extriangulated.
\item[(2)]$\mathfrak s|_{\mathbb F}$-inflations are closed under composition.
\item[(3)]$\mathfrak s|_{\mathbb F}$-deflations are closed under composition.
\item[(4)]$\mathbb F$ is closed.
\end{itemize}
\end{proposition}
\begin{definition}[\cite{Ogawa21}, Definition 2.4]\label{defect}
Let $(\C,\mathbb E,\mathfrak s)$ be a small extriangulated category. 
Let $\delta\in\mathbb E(C,A)$ be an $\mathbb E$-extension. 
Take a realization $A\xrightarrow{x}B\xrightarrow{y}C$ of $\delta$ 
and define $\tilde{\delta}$ to be the cokernel of $\C(?,y):\C(?,B)\rightarrow \C(?,C)$ in $\Mod \C$. 
The functor $\tilde{\delta}$ is called the {\em defect} of $\delta$, 
or the defect of an $\mathfrak s$-conflation $A\xrightarrow{x}B\xrightarrow{y}C$. 
We denote by $\Def \mathbb E$ the subcategory of $\Mod \C$ consisting of $\C$-modules which are isomorphic to defects of some $\mathfrak s$-conflations.
\end{definition}
\begin{remark}\label{rmk:defectses}
Suppose we have an $\mathbb E$-extension $\delta\in\mathbb E(C,A)$ and a morphism $c:C'\rightarrow C$ in $\C$.
By \cite[Proposition 1.20]{LiuNakaoka19}, we have the following diagram
\begin{equation}\label{dia:defectses}
\begin{tikzcd}
A\ar[r,"u"]\ar[d,equal]&E\ar[d,"g"]\ar[r,"v"]&C'\ar[d,"c"]\ar[r,dashed,"c^*(\delta)=\theta"]&\,\\
A\ar[d,"u"swap]\ar[r,"x"]&B\ar[d,"\begin{bmatrix}1\\0\end{bmatrix}"]\ar[r,"y"]\ar[d,""]&C\ar[d,equal]\ar[r,dashed,"\delta"]&\,\\
E\ar[r,"\begin{bmatrix}g\\v\end{bmatrix}"swap]&B\oplus C'\ar[r,"{[}y{,}\,-c{]}"swap]&C\ar[r,dashed,"u_*(\delta)=\mu"swap]&\,
\end{tikzcd}.
\end{equation}
A diagram chasing shows that there is a short exact sequence of defects
\[
0\rightarrow \tilde{\theta}\rightarrow \tilde{\delta}\rightarrow \tilde{\mu}\rightarrow 0
\]
\end{remark}
\begin{theorem}[\cite{Enomoto21}, Theorem B]\label{thm:closedsubbifunctorserresubcategory}
Let $(\C,\mathbb E,\mathfrak s)$ be a small extriangulated category. 
Then the map $\mathbb F\mapsto \Def\mathbb F$ defines an isomorphism between the following posets, where the poset structures are given by inclusion.\begin{itemize}
\item[(1)]The poset of closed subbifunctors of $\mathbb E$.
\item[(2)]The poset of Serre subcategories of $\Def\mathbb E$.
\end{itemize}
\end{theorem}

\begin{theorem}\label{bijectionstructures}
Let $(\A,\mathcal S)$ be an exact dg category and $(H^0(\A),\mathbb E,\mathfrak s)$ the associated algebraic extriangulated category. The following posets, where the poset structures are given by inclusion, are isomorphic.
\begin{itemize}
\item[(1)]The poset of exact substructures of $(\A,\mathcal S)$.
\item[(2)]The poset of closed subbifunctors of $\mathbb E$.
\item[(3)]The poset of Serre subcategories of $\Def \mathbb E$.
\end{itemize}
\end{theorem}
\begin{proof}
The isomorphism between items (2) and (3) is given by Theorem~\ref{thm:closedsubbifunctorserresubcategory}.
We show the isomorphism between items (1) and (2).

Let $\mathbb F\subset \mathbb E$ be a closed subbifunctor. 
Let $\mathcal S_{\mathbb F}\subseteq\mathcal S$ be the class of conflations $X$
\[
\begin{tikzcd}
A\ar[r,"f"]\ar[rr,bend right=8ex,"h"swap]&B\ar[r,"j"]&C
\end{tikzcd}
\]  
whose equivalence class $[X]\in \mathbb F(C,A)\subseteq\mathbb E(C,A)$.
We first show that the class $\mathcal S_{\mathbb F}$ is closed under isomorphisms.
Let $X$ be a conflation in $\mathcal S_{\mathbb F}$ and $\alpha:X\rightarrow X'$ an isomorphism in $\mathcal H_{3t}(\A)$.
Then $X'$ is a conflation in $\mathcal S$.
By Lemma \ref{fact}, the isomorphism $\alpha$ factorises as the composition of isomorphisms $X\rightarrow \tilde{X}$ and $\tilde{X}\rightarrow X'$ and $[\tilde{X}]=\overline{a}_*[X]=\overline{c}^*[X']$ for some isomorphisms $\overline a$ and $\overline c$ in $H^0(\A)$. 
So we have $[X']=(\overline{c}^{-1})^*\overline{a}_*[X]$ and $X'$ also belongs to $\mathcal S_{\mathbb F}$.
Since $\mathbb F$ is an additive subbifunctor, Axiom $\Ex0$ for $\mathcal S_{\mathbb F}$ is satisfied.
By Proposition \ref{closedsubbifunctor}, $\mathfrak s|_{\mathbb F}$-deflations are closed under compositions. 
Since we have showed $\mathcal S_{\mathbb F}$ is closed under isomorphisms, deflations are closed under compositions.
Thus Axiom ${\Ex1}$ is satisfied.
Suppose we are in the context of Axiom $\Ex2$. 
So we have a cospan
\[
\begin{tikzcd}
&C'\ar[d,"c"]\\
B\ar[r,"p"swap,two heads]&C
\end{tikzcd}
\]
where $p:B\rightarrow C$ is a deflation in $\mathcal S_{\mathbb F}$.
Put $[X']=\overline{c}^*[X]$. 
The morphism from $X'$ to $X$ then restricts to the homotopy pullback of the above cospan.
Since $X'$ also belongs to $\mathcal S_{\mathbb F}$, the homotopy pullback of $p$ is also a deflation in $\mathcal S_{\mathbb F}$.
Thus Axiom ${\Ex2}$ is satisfied.
Dually, one shows Axiom ${\Ex2^{op}}$ for $\mathcal S_{\mathbb F}$.

Let $\mathcal S'\subseteq \mathcal S$ be an exact dg substructure. 
Let $(H^0(\A),\mathbb F_{\mathcal S'},\mathfrak s')$ be the corresponding extriangulated structure.
The inclusion $\mathcal S'\subseteq \mathcal S$ induces a natural transformation of additive bifunctors $\mathbb F_{\mathcal S'}\rightarrow \mathbb E$ which is compatible with the realizations $\mathfrak s$ and $\mathfrak s'$.
By Proposition \ref{split}, the natural transformation $\mathbb F_{\mathcal S'}\rightarrow \mathbb E$ is an inclusion and thus defines a closed subbifunctor of $\mathbb E$ by Proposition \ref{closedsubbifunctor}.

It is direct to check that the maps $\mathbb F\mapsto \mathcal S_{\mathbb F}$ and $\mathcal S'\mapsto \mathbb F_{\mathcal S'}$ are inverse to each other.
\end{proof}
\begin{corollary}\label{cor:exactsubstructure}
Let $(\C,\mathbb E,\mathfrak s)$ be an algebraic extriangulated category.
Let $\mathbb F\subset\mathbb E$ be a closed subbifunctor.
Then the extriangulated category $(\C,\mathbb F,\mathfrak s|_{\mathbb F})$ is still an algebraic extriangulated category.
\end{corollary}
\begin{remark}
We have an injection 
\[
\{\text{Exact dg structures on $\A$}\}\to \{\text{Extriangulated structures on $H^0(\A)$}\}.
\]
This map is in general not a surjection. 
For example, when $\A$ is concentrated in degree $0$, 
then the exact dg structures on $\A$ are exactly the Quillen exact structures on $\A$. 
But $\A$ may have nontrivial extriangulated structures, e.g.~if $\A$ is an algebraic triangulated category.
\end{remark}
\begin{definition}[\cite{IyamaNakaokaPalu24}, Definitions 2.1, 2.7]
Let $(\C,\mathbb E,\mathfrak s)$ be an extriangulated category. 
A non-split (i.e.~non zero) $\mathbb E$-extension $\delta\in\mathbb E(C,A)$ is said to be {\em almost split} if it satisfies the following conditions.
\begin{itemize}
\item[(AS1)] $a_*(\delta)=0$ for any non-section $a\in\C(A,A')$.
\item[(AS2)] $c^{*}(\delta)=0$ for any non-retraction $c\in\C(C',C)$.
\end{itemize}
A sequence of morphisms $A\xrightarrow{x}B\xrightarrow{y}C$ in $\C$ is called an {\em almost split sequence} if it realizes some almost split extension $\delta\in \mathbb E(C,A)$.
\end{definition}

A non-zero object $A\in\C$ is {\em endo-local} if $\End_{\C}(A)$ is local.
For any almost split $\mathbb E$-extension $\delta\in\mathbb E(C,A)$, we have that $A$ and $C$ are both endo-local, cf.~\cite[Proposition 2.5]{IyamaNakaokaPalu24}.
Almost split extensions are unique in the following sense.
\begin{proposition}[\cite{IyamaNakaokaPalu24}, Proposition 2.6]
Let $A\in\C$ be any object. If there are non-split extensions $\delta\in\mathbb E(C,A)$ and $\delta'\in\mathbb E(C',A)$ satisfying (AS2), then there is an isomorphism $c\in\C(C',C)$ such that $c^{*}(\delta)=\delta'$.
Dually, if non-split extensions $\rho\in\mathbb E(C,A)$ and $\rho'\in\mathbb E(C,A')$ satisfy (AS1) for some $A$, $A'$, $C\in \C$, then there is an isomorphism $a\in\C(A,A')$ such that $a_{*}(\rho)=\rho'$.
\end{proposition}

\begin{definition}[\cite{IyamaNakaokaPalu24}, Definition 3.1]
We say that $\C$ has {\em right almost split extensions} if for any endo-local non-projective object $A\in\C$, there exists an almost split extension $\delta\in\mathbb E(A,B)$ for some $B\in\C$.
Dually, we say that $\C$ has {\em left almost split extensions} if if for any endo-local non-injective object $B\in\C$, there exists an almost split extension $\delta\in\mathbb E(A,B)$ for some $A\in\C$.
We say that $\C$ has {\em almost split extensions} if it has right and left almost split extensions.

\end{definition}
Let $k$ be a field.  Suppose $\C$ is Krull--Schmidt.
Let $A\xrightarrow{x}B\xrightarrow{y}C$ be a conflation realizing an almost split $\mathbb E$-extension $\delta\in\mathbb E(C,A)$.
Then its defect $\tilde{\delta}$, i.e.~the cokernel of $\C(?,y):\C(?,B)\rightarrow \C(?,C)$ in $\Mod\C$, is the simple module at $C$.
So we have $\tilde{\delta}(C)=\End_{\C}(C)/\rad\End_{\C}(C)$, and $\tilde{\delta}(D)=0$ for any indecomposable $D\in\C$ that is not isomorphic to $C$.
Conversely, if a simple module $S_{C}$ lies in $\Def \mathbb E$ for an indecomposable object $C\in\C$, then, by definition, there is a defect $\delta\in \mathbb E(C,A)$ such that $\tilde{\delta}$ is isomorphic to $S_C$. 
It is clear that $\delta$ satisfies (AS2).
By \cite[Lemma 2.15]{IyamaNakaokaPalu24}, we may assume that $A$ is indecomposable, and then $\delta$ is an almost split extension, whose defect is necessarily isomorphic to $S_C$.
The following corollary generalizes an analogous result in exact categories, cf.~\cite[Theorem 2.26]{FangGorsky22}.
\begin{corollary}\label{cor:KrullSchmidt}
Let $k$ be a field. 
Let $(\A,\mathcal S)$ be an exact dg category such that the following conditions hold
\begin{itemize}
\item[1)] The $k$-category $H^0(\A)$ is Krull--Schmidt.
\item[2)] The extriangulated category $(H^0(\A),\mathbb E,\mathfrak s)$ is {\em admissible}, i.e.~every object in $\Def \mathbb E$ has finite length (cf.~\cite{Enomoto18}).
\end{itemize}
Then the lattice of exact substructures of $(\A,\mathcal S)$ is isomorphic to the Boolean lattice of the subsets of the set of isomorphism classes of indecomposable non-projective objects in $(H^0(\A),\mathbb E,\mathfrak s)$.
\end{corollary}
\begin{proof}
Let $\M$ be a set of representatives for each isomorphism class of non-projective indecomposable objects in $\C=H^0(\A)$. 
Let us show that for each $C\in \M$, there exists an almost split extension with ending term $C$.
Indeed, since $C$ is non-projective and $\C$ is Krull--Schmidt, there exists a non-zero extension $\delta\in\mathbb E(C,A)$ with $A$ indecomposable.
By item 2), each defect has finite length. Let us choose $\delta$ such that the length of $\tilde{\delta}$ is minimal among such extensions.
Let $c:C'\rightarrow C$ be a non-retraction morphism in $\C$.
Let us consider the diagram~(\ref{dia:defectses}) in Remark~\ref{rmk:defectses}. 
So there is a surjection from $\tilde{\delta}$ to $\tilde{\mu}$ where $\mu=u_*(\delta)$ for some morphism $u:A\rightarrow E$.
Since $C$ is indecomposable and $\delta$ is non-zero and $c$ is a non-retraction, the $\mathbb E$-extension $\mu$ is nonzero. 
We decompose $E$ into indecomposables $E=\oplus_{i} E_i$.
Then for some $p_i:E\rightarrow E_i$, we have $\mu_i=(p_i)_*(\mu)\neq 0$.
There is a surjection $\tilde{\mu}\rightarrow \tilde{\mu_i}$. So the composition $\tilde{\delta}\twoheadrightarrow \tilde{\mu}\twoheadrightarrow\tilde{\mu_i}$ is an isomorphism by the minimality of the length of $\tilde{\delta}$.
Then $\tilde{\delta}$ is isomorphic to $\tilde{\mu}$ and hence $\tilde{\theta}=0$. 
It follows that $c^*(\theta)=0$ and hence we have that $\tilde{\delta}$ is isomorphic to $S_C$ and that $\delta$ is an almost split extension.

We choose an almost split extension $\delta_C$ for each $C\in\M$ and put $\mathcal R=\{\delta_C,C\in\M\}$.
Since any defect takes zero value at each projective object in $\C$, 
the category $\Def\mathbb E$ can be identified with a full subcategory of $\Mod (\add \M)$.  
The defects $\tilde{\delta_{C}}$ are the simple modules in $\Mod (\add \M)$ and hence we have that $\Def \mathbb E$ is the full subcategory of $\Mod (\add\M)$ consisting of objects which have finite lengths. 
Recall that in a small finite length abelian category, the lattice of Serre subcategories is isomorphic to the lattice of the subsets of the set of isomorphism classes of simple objects.
Therefore, by Theorem~\ref{bijectionstructures}, the lattice of exact substructures of $(\A,\mathcal S)$ is isomorphic to the Boolean lattice of the subsets of the set of isomorphism classes of indecomposable non-projective objects in $\C$.
\end{proof}

\section{Example}\label{sec:example}
In this section, we illustrate our theory of exact dg categories using a typical class of examples: dg categories of two-term complexes of finitely generated projective modules over finite-dimensional algebras.
We demonstrate, by providing explicit examples, that there exist short exact sequences 
which do not belong to the exact structure inherited from the ambient pretriangulated dg categories. 
Furthermore, we offer a concrete example of a (hereditary!) algebra where the lattice of exact structures is given explicitly.
\subsection{Homotopy category of two-term complexes}
\label{subsec:2-term}
Throughout this section, let $k$ be a field and $A$ a finite-dimensional $k$-algebra. 
Let $\mathcal H^{[-1,0]}(\proj A)$ be the full subcategory of $\T=\mathcal H^b(\proj A)$ consisting of two-term complexes $P^{-1} \to P^0$ of finitely generated projective $A$-modules, where $P^{i}$ is concentrated in degree $i$. 
Let $\A=\C^{b}_{dg}(\proj A)$ be the canonical enhancement of
${\mathcal H}^b(\proj A)$ and $\A'$ its full dg subcategory on the objects $P^{-1} \to P^0$.
Then $\A'$ is stable under extensions in the pretriangulated dg category $\A$ and thus inherits a canonical exact dg structure.
For an $A$-module $M$, we denote by $M^*$ its $A$-dual.

\begin{proposition}\label{prop:twoterm}
There exists a homotopy short exact sequence in $\A'$ which is not homotopy short exact in 
$\A$ if and only if there exists an $A$-module $M$ satisfying the following properties
\begin{itemize}
\item[(1)] $\pd_A M= 1$, and
\item[(2)] $\pd_{A^{op}} M^*=1$, and
\item[(3)] $M$ is reflexive, i.e.~the canonical map $M\rightarrow M^{**}$ is an isomorphism.
\end{itemize}
\end{proposition}

\begin{proof}
Suppose we are given such an $A$-module $M$. 
We consider a minimal projective presentation of $M$
\[
0\rightarrow P_1\rightarrow P_0\rightarrow M\rightarrow 0.
\]
By taking the $A$-dual, we obtain a left exact sequence of left $A$-modules
\[
0\rightarrow M^*\rightarrow P_0^*\rightarrow P_1^*.
\]
Put $N=M^*$. We then take a minimal projective presentation of $N$
\[
0\rightarrow Q_1\rightarrow Q_0\rightarrow N\rightarrow 0.
\]
Again, by taking the $A$-dual, we arrive at a left exact sequence of right $A$-modules
\[
0\rightarrow M\rightarrow Q_0^*\rightarrow Q_1^*,
\]
since $M$ is assumed to be reflexive. 
The composition $P_0\rightarrow M\rightarrow Q_0^{*}$ then yields the following $3$-term homotopy complex
\[
\begin{tikzcd}
(P_1\rightarrow P_0)\ar[rr,"0"{red,swap},red,bend right=8ex]\ar[r]&(0\rightarrow Q_0^*)\ar[r] &(0\rightarrow Q_1^*)
\end{tikzcd}
\]
which is homotopy short exact in $\A'$, but not in $\A$.

Suppose we have a $3$-term homotopy complex in $\A'$ 
\begin{equation}
\begin{tikzcd}\label{dia:twoterm}
P^{\bullet}\ar[r,"f"]\ar[rr,bend right=8ex,"h"{red,swap},red]&Q^{\bullet}\ar[r,"g"] &R^{\bullet}
\end{tikzcd}.
\end{equation}
We form the following diagram in $\A$
\[
\begin{tikzcd}
P^{\bullet}\ar[r,"f"]\ar[d,"v"swap]&Q^{\bullet}\ar[r,"\begin{bmatrix}0\\1\end{bmatrix}"]\ar[d,equal]&U=\Cone(f)\ar[r,"{[}1{,}0{]}"]\ar[d,"u"swap]\ar[rd,"s"red,red]&\Sigma P^{\bullet}\ar[d]\\
V=\Sigma^{-1}\Cone(g)\ar[r,"{[}-1{,}0{]}"swap]&Q^{\bullet}\ar[r,"g"swap]&R^{\bullet}\ar[r,"\begin{bmatrix}0\\1\end{bmatrix}"swap]\ar[d]&\Sigma V\ar[d]\\
&&N\ar[r,equal]&N
\end{tikzcd}
\]
where $u={[}-h{,}g{]}$, $v=\begin{bmatrix}-f\\-h\end{bmatrix}$ and $s=\begin{bmatrix}0&0\\-1&0\end{bmatrix}$.
By definition, the $3$-term homotopy complex (\ref{dia:twoterm}) is homotopy short exact if and only if, for any $S^{\bullet}\in \A'$, the maps $u$ and $v$ induce canonical quasi-isomorphisms of complexes
\begin{equation}\label{seq:1}
\tau_{\leq 0}\Hom_{\A}(S^{\bullet}, P^{\bullet})\iso \tau_{\leq 0}\Hom_{\A}(S^{\bullet}, V),
\end{equation}
and 
\begin{equation}\label{seq:2}
\tau_{\leq 0}\Hom_{\A}(R^{\bullet},S^{\bullet})\iso \tau_{\leq 0}\Hom_{\A}(U, S^{\bullet}).
\end{equation}
We observe that the formula (\ref{seq:1}) holds for any $S^{\bullet}\in \A'$ if and only if it holds for $S^{\bullet}=A$ if and only if $P^{\bullet}\iso \tau_{\leq 0}V$. 
Similarly, the formula (\ref{seq:2}) holds for any $S^{\bullet}\in \A'$ if and only if it holds for $S^{\bullet}=\Sigma A$ if and only if $\Hom(R^{\bullet}, A)\iso \tau_{\leq 0} \Hom(U,A)$. 
Put $M=\ker(Q^{0}\oplus R^{-1}\xrightarrow{[-g^{0}, -d]} R^{0})$. 
Then the 3-term homotopy complex (\ref{dia:twoterm}) is homotopy left exact if and only if the following is a projective presentation of $M$
\[
\begin{tikzcd}
0\ar[r]&P^{-1}\ar[r,"{[}d{,} f^{-1}{]^{\intercal}}"]&P^0\oplus Q^{-1}\ar[r]& M\ar[r]&0.
\end{tikzcd}
\]
In other words, the morphism $g:Q^{\bullet}\rightarrow R^{\bullet}$ admits a homotopy kernel if and only if $\pd_{A}M\leq 1$.

Now, we suppose that the $3$-term homotopy complex (\ref{dia:twoterm}) is homotopy short exact in $\A'$.
Then, by the observation made earlier, we have $\pd_{A}(M)\leq 1$.
By taking the $A$-dual, we see that $M^{*}\iso \ker ((P^0)^{*}\oplus (Q^{-1})^{*}\rightarrow (P^{-1})^{*})$ and the following is a projective presentation of $M^{*}$
\begin{equation}
\begin{tikzcd}\label{seq:M*}
0\ar[r]&(R^{0})^{*}\ar[r]&(R^{-1})^{*}\oplus (Q^{0})^{*}\ar[r]& M^{*}\ar[r]&0.
\end{tikzcd}
\end{equation}
Hence, we also have $\pd_{A^{op}}(M^{*})\leq 1$.
By taking the $A$-dual of the sequence (\ref{seq:M*}) and by the definition of $M$, we see that the canonical morphism $M\rightarrow M^{**}$ is an isomorphism. 
Therefore, the $A$-module $M$ is reflexive. 
If the left $A$-module $M^{*}$ is projective, then the sequence (\ref{seq:M*}) splits and hence the morphism $Q^{0}\oplus R^{-1}\xrightarrow{[-g^{0}, -d]} R^{0}$ is a surjection.
It follows that $V$ is quasi-isomorphic to $\tau_{\leq 0}V$ and the $3$-term homotopy complex (\ref{dia:twoterm}) is homotopy short exact in $\A$.

\end{proof}
We give examples of algebras $A$ admitting such modules.
Note that such an algebra can neither be hereditary, nor local nor self-injective.
Consider the algebras given by the quivers with relations
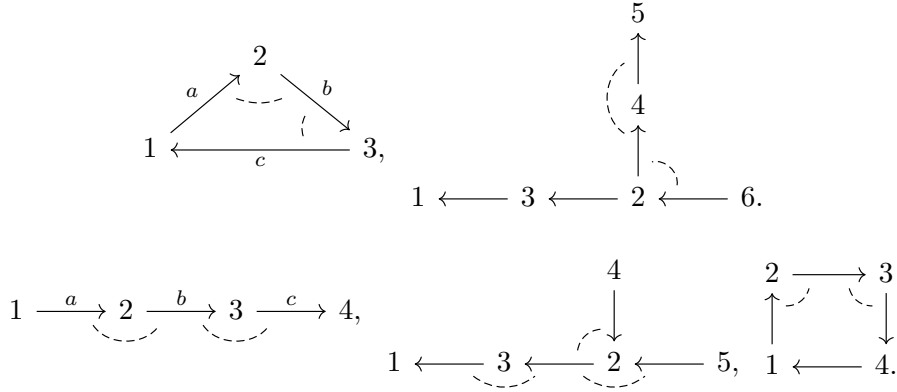
\[
\begin{tikzcd}
 &2\ar[rd,""{name=b,swap,near start},""{name=b',swap, near end},"b",""{name=b'',swap}]&\\
1\ar[ru,""{name= a,swap,near end},"a"]&&3,\ar[ll,""{name=c, swap,near start},"c"]\arrow[from = a, to = b,bend right=20, dashed, no head]\ar[from=b'',to=c,bend right=20,dashed,no head]
\end{tikzcd}
\begin{tikzcd}
&&5&\\
&&4\ar[u,""{name=4}]&\\
1&3\ar[l]&2\ar[u,""{name=3,near end},""{name=2,swap,near start}]\ar[l]&6.\ar[l,""{near end,swap,name=1}]\ar[r,from=1,to=2,dashed,no head,bend right=12ex]\ar[r,from=3,to=4,dashed,no head,bend left=12ex]
\end{tikzcd}
\]
\[
\begin{tikzcd}
1\ar[r,""{name=1,swap, near end},"a"]&2\ar[r,""{name=2,swap,near start},""{name=3,swap, near end},"b"]&3\ar[r,""{name=4,swap,near start},"c"]&4\ar[r,bend right=12ex, from =1,to=2,dashed,no head]\ar[r,bend right=12ex,from=3,to=4,dashed,no head],
\end{tikzcd}
\begin{tikzcd}
&&4\ar[d,""{name=5,swap,near end}]&\\
1&3\ar[l,""{name=1,near start}]&2\ar[l,""{name=2,near end},""{name=3,near start},""{name=6,near start, swap}]&5,\ar[l,""{name=4,near end}]\ar[r,from=1,to=2,dashed,no head,bend right=8ex]\ar[r,from =3,to=4,dashed,no head,bend right=8ex]\ar[r,from=5,to=6,dashed,no head,bend right =8ex]
\end{tikzcd}
\begin{tikzcd}
2\ar[r,""{name=2,near start,swap},""{name=3,near end, swap}]&3\ar[d,""{name=4,near start, swap}]\\
1\ar[u,""{name=1,swap,near end}]&4.\ar[l]\ar[r,from=1,to=2,dashed,no head,bend right=8ex]\ar[r,from=3,to=4,no head,dashed, bend right=8ex]
\end{tikzcd}
\]
Then the simple modules $S_2$ at the vertices 2 satisfy the above properties.

\begin{example}
Let $A$ be the path $k$-algebra $k(\overrightarrow{A_2})$ of the quiver $A_2$. 
As the algebra $A$ is hereditary, the inherited exact structure on $\A'$ is its greatest exact dg structure, as can be seen by Proposition~\ref{prop:twoterm}. 
Let $(\C=H^0(\A'),\mathbb E,\mathfrak s)$ be the corresponding extriangulated category.
We denote by $S_2$ the two-term complex $P_1\rightarrow P_2$ where $P_i$ is the indecomposable projective $A$-module corresponding to the vertex $i$. 
Then the AR-quiver of $\C$ is as follows
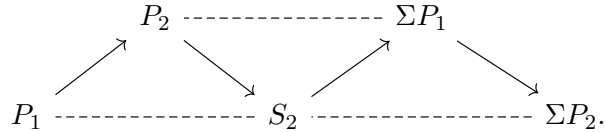
\[
\begin{tikzcd}
&P_2\ar[rd]&&\Sigma P_1\ar[rd]\ar[ll,dashed,no head]&\\
P_1\ar[ru]&&S_2\ar[ru]\ar[ll,dashed,no head]&&\Sigma P_2.\ar[ll,dashed,no head]
\end{tikzcd}
\]

We take the following AR-sequences:
\[
\begin{tikzcd}
\alpha: P_2\ar[r,tail]&S_2\ar[r,two heads]&\Sigma P_1,
\end{tikzcd}
\]
\[
\begin{tikzcd}
\beta: P_1\ar[r,tail]&P_2\ar[r,two heads]&S_2,
\end{tikzcd}
\]
\[
\begin{tikzcd}
\gamma:S_2\ar[r,tail]&\Sigma P_1\ar[r,two heads]&\Sigma P_2.
\end{tikzcd}
\]
By Corollary~\ref{cor:KrullSchmidt}, the lattice of exact structures on $\A'$ is isomorphic to the following Boolean lattice:
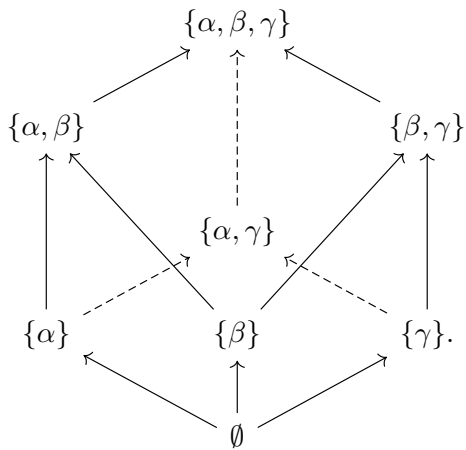
\[
\begin{tikzcd}
&\{\alpha,\beta,\gamma\}&\\
\{\alpha,\beta\}\ar[ru]&&\{\beta,\gamma\}\ar[lu]\\
&\{\alpha,\gamma\}\ar[uu,dashed]&\\
\{\alpha\}\ar[uu]\ar[ru,dashed]&\{\beta\}\ar[luu]\ar[ruu]&\{\gamma\}.\ar[uu]\ar[lu,dashed]\\
&\emptyset\ar[u]\ar[lu]\ar[ru]&
\end{tikzcd}
\]
\end{example}

\section{Stable dg categories}\label{sec:stable}
In this section, we focus on a specific class of additive dg categories, the class of {\em stable dg categories}.
The property of being stable is invariant under isomorphisms in the localization $\Hqe_{\leq 0}$ of $\dgcat$ at the connective quasi-equivalences, as defined in Definition~\ref{def:connective quasi-equivalence}.
Typical examples include (connective covers of) pretriangulated dg categories.
For a Jacobi-finite quiver with potential $(Q,W)$, we observe that $\F_{dg}$ is stable, where $\F_{dg}\subseteq \per_{dg}(\Gamma_{Q,W})$ is the canonical dg enhancement of the {\em fundamental domain} \cite[2.2]{Amiot09} $\F\subseteq \per(\Gamma_{Q,W})$.
This provides a class of non-connective dg categories that are stable but not pretriangulated.
\subsection{Stable dg categories}
We have the following key notion of this section. It is a dg analog of Lurie's notion of stable $\infty$-category \cite[Definition 1.1.1.9]{LurieHA}.
\begin{definition}\label{def:stable}
Let $\A$ be an additive dg category. 
It is {\em  stable} if the following conditions are satisfied:
\begin{itemize}
\item[(a)] the dg category $\A$ admits homotopy kernels and homotopy cokernels;
\item[(b)] a  3-term homotopy complex in $\A$
is homotopy left exact if and only if it is homotopy right exact.
\end{itemize} 
\end{definition}
It is straightforward to check that the property of being stable is invariant under connective quasi-equivalences, cf.~Definition~\ref{def:connective quasi-equivalence}.
\begin{example}
By Example~\ref{exm:pretriangulated}, a pretriangulated dg category $\A$, and hence its connective cover $\tau_{\leq 0}\A$, is stable.
\end{example}
\begin{remark}
The notion of stable dg category is more flexible than the notion of pretriangulated dg category, in the sense that it is preserved under connective quasi-equivalences, which is not the case at all for the notion of pretriangulated dg category. In a subsequent paper, we will show that they are essentially the same notion, in the sense that each stable dg category is isomorphic to a pretriangulated dg category in $\Hqe_{\leq 0}$.
\end{remark}
Let $\A$ be a stable dg category. It is clear that the following $3$-term homotopy complex in $\A$
\[
\begin{tikzcd}
0\ar[r,"\Id"]\ar[rr,bend right=8ex,"0"swap]&0\ar[r,"\Id"]&0
\end{tikzcd}
\]
is homotopy short exact.  Since the dg category $\A$ admits homotopy pullbacks which are homotopy bicartesian, each morphism in $Z^0(\A)$ is a deflation. Dually, each morphism in $Z^0(\A)$ is an inflation. Thus, it is straightforward to verify the following proposition.
\begin{proposition}\label{prop:stablestructure}
Let $\A$ be a stable dg category and $\mathcal S$ the class of all homotopy short exact sequences in $\A$. Then $(\A,\mathcal S)$ is an exact dg category.
\end{proposition}

\begin{theorem}\label{thm:stabletriangulated}
Let $\A$ be a small exact dg category.
Then the extriangulated structure on $H^0(\A)$ by Theorem~\ref{thm:extriangulatedstructure} is triangulated if and only if $\A$ is a stable dg category with the canonical exact structure.
\end{theorem}
\begin{proof}
According to \cite[Corollary 7.6]{NakaokaPalu19}, an extriangulated category is tringulated if and only if it is a Frobenius extriangulated category with the full subcategory of projective-injectives being zero.

Assume the dg category $\A$ is stable with the canonical exact structure. 
Let $A$ be an object in $H^0(\A)$. By the definition of stability, it admits conflations
\[
\begin{tikzcd}
A'\ar[r]&0\ar[r]&A,
\end{tikzcd}
\] 
and 
\[
\begin{tikzcd}
A\ar[r]&0\ar[r]&A''.
\end{tikzcd}
\]
Hence, the extriangulated category $(H^0(\A),\mathbb E,\mathfrak s)$ is Frobenius with the full subcategory of projective-injectives being zero. 
Therefore, it is a triangulated category.

Conversely, let us assume that the extriangulated category $(H^0(\A),\mathbb E,\mathfrak s)$ is Frobenius with the full subcategory of projective-injectives being zero. 
So for any object $A\in \A$, the morphism $A\rightarrow 0$ is an inflation.
By the dual of Claim 1 in the proof of Proposition~\ref{property}, we deduce that any morphism $f:A\rightarrow B$ is an inflation. 
In particular, the dg category $\A$ admits homotopy cokernels which are homotopy short exact.
Dually, it admits homotopy kernels which are homotopy short exact.
It follows that a $3$-term homotopy complex in $\A$ is homotopy left exact if and only if it is homotopy right exact.

\end{proof}
\subsection{Amiot's construction of fundamental domain}
The following example is taken from~\cite{Amiot09}. 
It serves as an illustration of non-connective dg categories that are stable but not pretriangulated. 
Recall that a dg functor $F:\F_{dg}\rightarrow \C_{dg}$ between pretriangulated dg categories necessarily induces a triangle functor $H^0(F):H^0(\F_{dg})\rightarrow H^0(\C_{dg})$. 
It also demonstrates that  a dg functor  between stable dg categories does not necessarily induce a triangle functor. 

Let $k$ be a field. Recall that a dg $k$-algebra $A$ is {\em homologically smooth} if the object $A$, viewed as an $A^{e}$-module, is perfect. 
It is {\em connective} if for each $p>0$, the space $H^p(A)$ vanishes. It is {\em bimodule $3$-CY} if there is an isomorphism in $\D(A^{e})$
\[
R\Hom_{A^{e}}(A,A^{e})\iso \Sigma^{-3}A.
\]
Let $A$ be a homologically smooth, connective, bimodule $3$-CY dg $k$-algebra, such that $H^0(A)$ is finite-dimensional.
Note that the opposite dg algebra $A^{op}$ also satisfies these properties.
For example, one could take the Ginzburg dg algebra $\Gamma_{Q,W}$ associated with a Jacobi-finite quiver with potential $(Q,W)$.
Let $\per(A)$ be the perfect derived category of $A$, which is the thick subcategory of the derived category $\D=\D(A)$ generated by $A$.
A dg $A\mbox{-}$module is {\em perfectly valued} if its total cohomology is finite-dimensional and we denote by $\pvd(A)$ the triangulated category of perfectly valued dg $A\mbox{-}$modules.
By~\cite[Lemma 4.1]{Keller08d}, $\pvd(A)$ is a full subcategory of $\per(A)$ and we have a bifunctorial isomorphism 
\[
D\Hom_{\D}(L,M)\iso \Hom_{\D}(M,\Sigma^{3}L)
\]
for all objects $L$ in $\pvd(A)$ and all objects $M$ in $\D(A)$.
Recall the canonical t-structure $(\D^{\leq 0},\D^{\geq 0})$ on $\D$, cf.~Lemma~\ref{lemma:tstructure}, where $\D^{\leq 0}$ (resp.~$\D^{\geq 0}$) is the full subcategory consisting of dg $A$-modules whose cohomology is concentrated in non-positive (resp.~non-negative) degrees.
We denote $\tau_{\leq 0}$ as the associated truncation functor. Note that on a dg $A$-module $M$, the underlying complex of the truncation $\tau_{\leq 0} M$ is obtained by applying the truncation $\tau_{\leq 0}$ to the underlying complex of $M$, while preserving the induced $A$-action. 
The {\em fundamental domain $\F=\F({A})$} \cite[2.2]{Amiot09} is defined as
\[
\F=\D^{\leq 0}\cap {^{\perp}\D^{\leq -2}}\cap \per(A),
\]
where $^{\perp}(-)$ denotes the Hom orthogonal.
Put $\P=\add(A)$.
By~\cite[Lemma 2.10]{Amiot09}, it can also be described as $\F=\P\ast \Sigma \P$.
Let $\C$ be the Verdier quotient $\per(A)/\pvd(A)$. 
By \cite[Proposition 2.9]{Amiot09}, the projection functor $\pi:\per(A)\rightarrow \C$ induces a $k$-linear equivalence 
\[
\pi: \F\iso \C.
\]
We add the subscript $dg$ to the aforementioned categories to denote their dg enhancements.
In particular, $\C_{dg}$ is the dg quotient $\per_{dg}(A)/\pvd_{dg}(A)$.
Since $\F_{dg}\subset \per_{dg}(A)$ is not stable under shifts, it is not a pretriangulated dg category.
The following proposition provides a dg enhancement of the above $k$-linear equivalence. 
In particular,  it implies that the dg category $\F_{dg}$ is stable.
\begin{proposition}
The dg quotient functor $\per_{dg}(\A)\rightarrow \C_{dg}$ induces a connective quasi-equivalence between $\F_{dg}$ and $\C_{dg}$.
\end{proposition}
\begin{proof}
It suffices to show that for objects $X$ and $Y$ in $\F$ and $i\leq 0$, the quotient functor $\pi:\per(A)\rightarrow \C$ induces a canonical isomorphism
\begin{equation}\label{map:connectivequasiequivalence}
\Hom_{\per(A)}(X,\Sigma^{i}Y)\iso \Hom_{\C}(X,\Sigma^i Y).
\end{equation}

We begin by showing that the map (\ref{map:connectivequasiequivalence}) is an injection.
Suppose $f:X\rightarrow \Sigma^i Y$ is a morphism in $\per(A)$ that factors through an object $M$ in $\pvd(A)$. Since $X\in \D^{\leq 0}$, we may assume that $M\in \pvd(A)^{\leq 0}$. Then, we have 
\[
D\Hom_{\per(A)}(M,\Sigma^iY)\iso \Hom_{\per(A)}(\Sigma^i Y,\Sigma^3 M)=0
\]
since $Y\in {^{\perp}\D^{\leq -2}}$. Consequently, the morphism $f$ is also zero and the map (\ref{map:connectivequasiequivalence}) is an injection.
Now, let us demonstrate the surjectivity of the map (\ref{map:connectivequasiequivalence}).
Consider a roof $b/s$
\[
\begin{tikzcd}
&N\ar[rd,"b"]\ar[ld,Rightarrow,"s"swap]&\\
X&&\Sigma^i Y
\end{tikzcd}
\]
where $\Cone(s)\in \pvd(A)$.
Since $X\in \per(A)^{\leq 0}$,we may assume that $\Cone(s)\in \pvd(A)^{\leq 0}$.
Then we have 
\[
D\Hom_{\per(A)}(\Sigma^{-1}\Cone(s),\Sigma^iY)\iso \Hom_{\per(A)}(\Sigma^i Y,\Sigma^2 \Cone(s))=0
\]
since $Y\in {^{\perp}\D^{\leq -2}}$. Thus, the morphism $b$ factors through the morphism $s$ and this shows the surjectivity of the map (\ref{map:connectivequasiequivalence}).
\end{proof}
Our next aim is to describe the homotopy short exact sequences in $\F_{dg}$.
For this, we use the quasi-equivalence $\Sigma \RHom_{A}(-,A):\F_{dg}^{op}\iso \F_{dg}(A^{op})$, and it suffices to describe homotopy kernels in $\F_{dg}$.
By~\cite[Lemma 2.5]{Amiot09}, the space $H^p(A)$ is finite-dimensional for each $p$.
Consequently, the triangulated category $\per(A)$ is Hom-finite, and the t-structure $(\D^{\leq 0},\D^{\geq 0})$ restricts to a t-structure $(\per(A)^{\leq 0}, \per(A)^{\geq 0})$ on $\per(A)$.
Now, let us elucidate the homotopy kernels of morphisms in $\F_{dg}(A)$.
Let $g:M\rightarrow N$ be a morphism in $Z^0\F_{dg}(A)$.
Define $L=\Sigma^{-1}\Cone(g)$ as its homotopy kernel in $\per_{dg}(A)$.
So we have the following homotopy short exact sequence in $\per_{dg}(A)$
\[
\begin{tikzcd}
L\ar[r,"f'"]\ar[rr,"h'"{swap},bend right=8ex]& M\ar[r,"g"]&N.
\end{tikzcd}
\]
\begin{lemma} 
The truncation $\tau_{\leq 0}L$ lies in $\F$ and the inclusion $i:\tau_{\leq 0}L\rightarrow L$ yields a homotopy short exact sequence in $\F_{dg}(A)$
\begin{equation}
\begin{tikzcd}\label{dia:homotopykernelamiot}
\tau_{\leq 0}L\ar[r,"f"]\ar[rr,"h"{swap},bend right=8ex]& M\ar[r,"g"]&N,
\end{tikzcd}
\end{equation}
where $f=f'i$ and $h=h'i$.
\end{lemma}
\begin{proof}
It is evident that $L$, as an extension of $\Sigma^{-1}N$ and $M$, belongs to ${^{\perp}\D^{\leq -2}}$. 
Furthermore, it is clear that $L\in \D^{\leq 1}$, indicating that $\tau_{\geq 1}L$ lies in the shift of the heart of t-structure on $\per(A)$. Specifically, we have $\tau_{\geq 1}L\in \pvd(A)$.
Thus, for each $M\in {\D^{\leq -2}}$, we have
\[
D\Hom_{\D}(\Sigma^{-1}\tau_{\geq 1}L, M)\iso \Hom_{\D}(M,\Sigma^{2}\tau_{\geq 1}L)=0.
\]
This implies that $\tau_{\leq 0}L$, as an extension of $\Sigma^{-1}\tau_{\geq 1}L$ and $L$, belongs to ${^{\perp}\D^{\leq -2}}$.
Therefore, $\tau_{\leq 0}L$ lies in $\F$.

Now, let us verify that the homotopy 3-term complex (\ref{dia:homotopykernelamiot}) is homotopy short exact in $\F_{dg}(A)$. 
Indeed, let $K\in \F$ and $i\leq 0$.  
Then $\Hom(K,\Sigma^i\tau_{\geq 1}L)=0$ and 
\[
D\Hom_{\per(A)}(\tau_{\geq 1}L,\Sigma^{i}K)\iso \Hom_{\per(A)}(\Sigma^iK, \Sigma^3\tau_{\geq 1}L)=0.
\]
Thus, the canonical map $\tau_{\leq 0}L\rightarrow L$ induces an isomorphism
\[
\Hom_{\per(A)}(K,\Sigma^i \tau_{\leq 0}L)\iso \Hom_{\per(A)}(K,\Sigma^i L).
\]
The canonical morphism $\Cone(f)\rightarrow N$ induces an isomorphism
\[
\Hom_{\per(A)}(N,\Sigma^i K)\iso \Hom_{\per(A)}(\Cone(f), \Sigma^i K).
\]
This confirms that (\ref{dia:homotopykernelamiot}) is homotopy short exact.
\end{proof}
\begin{example}
\rm{
In the special case where $(Q,W)$ is given by the quiver $A_1$ with a trivial potential $W=0$,
 the associated Ginzburg dg algebra $\Gamma_{Q,W}$ is the polynomial algebra $k[t]$, where $t$ has degree $-2$.
Consequently, the associated fundamental domain $\F$ is equivalent, as a $k$-category, to the category $\Vect^{\mathbb Z/2\mathbb Z}$ of finite-dimensional super vector spaces $V=V_0\oplus V_1$.
 The category $\Vect^{\mathbb Z/2\mathbb Z}$ admits a canonical dg enhancement $\Vect^{\mathbb Z/2\mathbb Z}_{dg}$: for any two objects $V$ and $V'$ in $\Vect^{\mathbb Z/2\mathbb Z}$, the Hom complex $\Hom(V,V')$ is defined such that 
 \[
 \Hom^{2n}(V,V')=\Hom_{k}(V_1,V_1')\oplus \Hom_{k}(V_2,V_2'),
 \]
 \[
 \Hom^{2n-1}(V,V')=\Hom_{k}(V_1,V_2')\oplus \Hom_{k}(V_2,V_1'),
 \]
  for $n\leq 0$, and 
  \[
  \Hom^1(V,V')=\Hom_{k}(V_1,V_2'), \quad
  \Hom^{\geq 2}(V,V')=0,
  \]
  and the differential is trivial.
  By the results presented in this subsection, we obtain that the dg category $\Vect^{\mathbb Z/2\mathbb Z}_{dg}$ is stable. Its homotopy kernels and homotopy cokernels are given as follows. Let $f:V\rightarrow V'$ be a morphism of super vector spaces. 
 Its homotopy kernel is 
 \[
 \begin{tikzcd}
\ker(f)\oplus \Sigma \cok(f)\ar[r,"{[}\mathrm{inc}{,}0{]}"] \ar[rr,"0"swap, bend right=6ex]&V\ar[r,"f"]&V'
 \end{tikzcd},
 \]
 where $\mathrm{inc}:\ker(f)\rightarrow V$ is the canonical inclusion, and $\Sigma$ is the shift functor defined by sending $V=V_0\oplus V_1$ to $\Sigma V=V_1\oplus V_0$. 
 The homotopy cokernel of $f$ is given by 
 \[
 \begin{tikzcd}
 V\ar[r,"f"]\ar[rr,"0"swap,bend right=6ex]&V'\ar[r,"{[}\pr{,} 0{]^{\intercal}}"]&\cok(f)\oplus \Sigma \ker(f),
 \end{tikzcd}
 \]
 where $\pr:V'\rightarrow \cok(f)$ is the canonical projection.
 }
\end{example}

\section*{Acknowledgement}
This paper is part of the author's Ph.D thesis~\cite{Chen23}.
While writing this paper, the author was financially supported by Xiaomi Youth Scholar.
He is greatly indebted to his supervisor Bernhard Keller and cosupervisor Xiao-Wu Chen for their constant support and encouragement without which this paper would not exist.
He is grateful to Mikhail Gorsky, Gustavo Jasso, Julian K\"ulshammer, Yann Palu, Bertrand To\"en and Yilin Wu for discussions and useful comments. 

	\def\cprime{$'$} \def\cprime{$'$}
	\providecommand{\bysame}{\leavevmode\hbox to3em{\hrulefill}\thinspace}
	\providecommand{\MR}{\relax\ifhmode\unskip\space\fi MR }
	\providecommand{\MRhref}[2]{%
		\href{http://www.ams.org/mathscinet-getitem?mr=#1}{#2}
	}
	\providecommand{\href}[2]{#2}
%


	\bibliographystyle{amsplain}
	\bibliography{stanKeller}

\end{document}